\documentclass[12pt]{article}
\usepackage{xcolor}
\usepackage{srcltx}
\usepackage{amsmath}
\usepackage{amssymb}
\usepackage{mathtools}
\usepackage{amsthm}
\usepackage{amstext}
\usepackage{amsopn}
\usepackage{placeins}
\usepackage{subfigure}
\usepackage{tikz}
\usepackage{graphicx}
\usepackage{bm}
\usepackage{amssymb}
\usepackage{cases}
\usepackage{booktabs}
\usepackage{empheq}
\usepackage{enumitem}
\DeclareMathOperator{\Range}{R} 
\newtheorem{theorem}{Theorem}[section]
\newtheorem{lemma}[theorem]{Lemma}
\newtheorem{proposition}[theorem]{Proposition}

\theoremstyle{definition}
\newtheorem{definition}[theorem]{Definition}

\theoremstyle{remark}
\newtheorem{remark}[theorem]{Remark}

\numberwithin{equation}{section}

\newcommand{\ba}{\begin{array}}
	\newcommand{\ea}{\end{array}}

\newcommand{\la}{\lambda}

\newcommand{\ds}{\displaystyle}

\usepackage{appendix}
\usepackage{titlesec}
\usepackage{microtype}

\begin{document}
\date{}
\title{ \bf\large{Global dynamics in a reaction-diffusion competition model with edge behavior}}
\author{
Kuiyue Liu\; and\; Shanshan Chen\footnote{Corresponding author, Email: chenss@hit.edu.cn}\\[-1mm]
{\small Department of Mathematics, Harbin Institute of Technology}\\[-2mm]
{\small Weihai, Shandong 264209, P. R. China}\\[1mm]
}

\maketitle
\begin{abstract}
In this paper, we investigate a two-species competition model in a landscape consisting of a finite number of adjacent patches. For the two-patch scenario, by treating edge behavior at the interface as a strategy, it has been shown that there exists an ideal free distribution (IFD) strategy, which is a globally evolutionarily stable strategy. Specifically, when the resident species follows the IFD strategy and the mutant species does not, the mutant species is unable to invade the resident population. Building on this foundation, our work focuses on exploring the dynamics of the system when neither species can adopt the IFD strategy.
 We demonstrate that if the strategies of both species either exceed or fall below the IFD strategy, the mutant species can outcompete and eliminate the resident species, provided that its strategy is closer to the IFD strategy and its diffusion rates are equal to or slower than those of the resident species.
Furthermore, if the strategies of the two species lie on opposite sides of the IFD strategy, the two species can coexist. This result is further extended to the case of an arbitrary but finite number of patches.
\\[2mm]
\end{abstract}
\section{Introduction}
Reaction-diffusion equations are often the framework of choice to model biological invasions. For two competing species, it was demonstrated in \cite{4-Dockery1998,3-Hasting1983} that the species with a slower diffusion rate could outcompete and eliminate its faster counterpart if they are identical except for their diffusion rates, and one can also refer to \cite{5-He-Ni2016,6-LamNi2012,2-Lou-JDE-2006} and references therein for the weak competition case in spatially heterogeneous environments.

As a result of habitat fragmentation, the landscape may comprise multiple patches. By ignoring the spatial structure within each patch,
discrete patch models (systems of ordinary differential equations) can be used to model population dynamics. The researchers also addressed  questions similar to those studied in reaction-diffusion models. For instance, one can refer to
\cite{49-Chen-2022,50-Chen-2022,17-Cheng-Lin-2019,18-Gourley-Kuang-2005,3-Hasting1983,19-Lin-Lou-2014,20-Slavik-2020} and references therein for the dynamics of two competing species. 

Taking into account the spatial structure of patches, numerous experiments have illustrated that there exists edge behavior of individuals, and individuals may have movement preference at interfaces between patches \cite{23-Reeve-2010,24-Reeve-2008,21-Ries-2001,25-Schtickzelle-2003,22-Schultz-2001}.
The continuous density conditions at interfaces were first proposed in \cite{26-Pacala-1982} and \cite{27-Shigesada-1986} for finite and infinite two-patch landscapes, respectively. Further and related studies can be found in \cite{29-Cruywagen-1996,28-Freedman-1989,30-Lutscher-2006} and references therein.

The habitat preference of individuals can lead to density discontinuities at interfaces.
Building on this observation and the work of Ovaskainen and
Cornell \cite{31-Ovaskainen-2003}, Maciel and Lutscher \cite{32-Maciel-2013} first introduced the discontinuous density conditions and proposed the following single-species reaction-diffusion model with two patches:
\begin{equation}\label{single}
	\begin{cases}
		(u_1)_t=d_1(u_1)_{xx}+\ds r_1u_1(1-\frac{u_1}{k_1}),\;\;0<x<x_1,\\
		(u_2)_t=d_2(u_2)_{xx}+\ds r_2u_2(1-\frac{u_2}{k_2}),\;\;x_1<x<L,\\
		u_2(x_1^+,t)=pu_1(x_1^-,t),\; d_2(u_2)_x(x_1^+,t)=d_1(u_1)_x(x_1^-,t),\\
		(u_1)_x(0^+)=(u_2)_x(L^-)=0.
	\end{cases}
\end{equation}
(They also investigated an infinite landscape consisting of periodically alternating patches of two types.)
Here $u_i$ denotes the density of the species in patch $i$, $d_i$ denotes the diffusion coefficient in patch $i$,
and $r_i$ and $k_i$ represent the intrinsic growth rate and carrying capacity of the species, respectively. Moreover,
$p:=(d_1/d_2)[{\alpha}/{(1-\alpha)}]$,
where $\alpha$ characterizes the patch preference at the interface $x=x_1$. The density of the species is discontinuous
at the interface if $p\ne1$. Zaker et al. \cite{33-Zaker-BMB-2019} showed that edge behavior influences the shape of the steady state of model \eqref{single}, and total biomass could exceed total carrying capacity under certain conditions. Then Tchouaga and Lutscher \cite{34-TCHOUAGA-2023} studied another single-species model with Allee effect and found that bi-stability can occur, characterized by  two nonzero stable steady states. Moreover, one can  refer to \cite{37-Alqawasmeh-2019,35-Zhang-2022,44-Zhang-2024,40-Hamel-2024,38-Maciel-2015,39-Shigesada-2015} for propagation phenomena in reaction-diffusion models with edge behavior.

There are also significant findings on two-species competition models similar to \eqref{single}.
For a competition model in a one-dimensional infinite landscape, Maciel and Lutscher \cite{49-Maciel-2018} found that diffusion rate and patch preference influence competitive outcomes through muti-scale analysis and homogenization.
Then Maciel et al. \cite{36-Maciel-2020} considered the following model with $n=2$:
	\begin{subequations}\label{m1}
	\begin{numcases}{}
		(u_i)_t=d_i( u_i)_{xx}+r_i u_i\left(1-\ds\frac{u_i+v_i}{k_i}\right),\;\;\;\; x\in \Omega_i,\;t>0,\; i=1,2,\cdots,n, \label{m1-1}\\
		(v_i)_t=\hat d_i( v_i)_{xx}+r_i v_i\left(1-\ds\frac{u_i+v_i}{k_i}\right),\;\; \;\; x\in \Omega_i,\;t>0, \; i=1,2,\cdots,n,\label{m1-2}\\
		u_{i+1}(x_i^{+},t)=p_i u_i(x_i^{-},t),\;d_{i+1}( u_{i+1})_{x}(x_i^{+},t)=d_i ( u_{i})_{x}(x_i^{-},t),\;\;\;\;t>0,\label{m1-3}\\
		v_{i+1}(x_i^{+},t)=\hat p_iv_i(x_i^{-},t),\;\hat d_{i+1}( v_{i+1})_{x}(x_i^{+},t)=\hat d_i ( v_{i})_{x}(x_i^{-},t),\;\;\;\;\;t>0,\label{m1-4}\\
		( u_1)_x(0^{+},t)=(u_n)_x(L^{-},t)=( v_1)_x(0^{+},t)=( v_n)_x(L^{-},t)=0,\;\;\;t>0.\label{m1-5}
	\end{numcases}
\end{subequations}
Here the landscape is finite and composed of
$n$ adjacent patches $\{\Omega_i\}_{i=1}^n$, where $\Omega_i:=(x_{i-1},x_i)$ represents the $i$-th patch with length $x_i-x_{i-1}>0$ for $i=1,\cdots,n$; $x_0=0$ and $x_n=L$ denote the boundaries of the landscape and the no-flux boundary conditions are imposed at $x=0,L$; $u_i$ and $v_i$ denote the densities of two competing species in patch $i$; and $d_i$, $\hat d_i>0$ denote the diffusion rates of the two species in patch $i$. The two species admit the same per capita growth rate
with intrinsic growth rate $r_i>0$ and carrying capacity $k_i>0$ in each patch.
At the interface $x=x_i$ ($i=1,\cdots,n-1$), the discontinuous density conditions (if $p_i\ne 1$) and balanced flux conditions are imposed for the two species, where the superscripts $\pm$  denote one-sided limits from the right and left, respectively, and $p_i,\hat p_i>0$  are  composite parameters with
\begin{equation*}
	p_i:=\dfrac{\alpha_i}{1-\alpha_i}\dfrac{d_i}{d_{i+1}},\;\;\hat p_i:=\dfrac{\hat\alpha_i}{1-\hat\alpha_i}\dfrac{\hat d_i}{\hat d_{i+1}},\;\;i=1,\cdots,n-1.
\end{equation*}
Here $\alpha_i,\hat\alpha_i\in(0,1)$ represent the patch preference of the two species at the interface $x=x_i$. Specifically, $\alpha_i$ (resp. $\hat\alpha_i$) is the probability that the species $\mathbf u$ (resp. $\mathbf v$) at the interface $x_i$ moves to patch $i+1$, and $1-\alpha_i$ (resp. $1-\hat\alpha_i$) is the probability that the species $\mathbf u$ (resp. $\mathbf v$) at the interface $x_i$ moves to patch $i$.
	Denote
	\begin{equation}\label{ok} \mathbf{\overline{k}}=\left(\dfrac{k_2}{k_1},\cdots,\dfrac{k_n}{k_{n-1}}\right),\;\;\mathbf{p}=(p_1,\cdots,p_{n-1}),\;\;\mathbf{\hat p}=(\hat p_1,\cdots,\hat p_{n-1}).
	\end{equation}
For the two-patch case ($n=2$), Maciel et al. \cite{36-Maciel-2020} showed that $\mathbf{\overline{k}}$ is an IFD strategy. Moreover, if the resident species $\mathbf u$ adopts this IFD strategy, the mutant species $\mathbf v$ with a different strategy ($\mathbf{\hat{ p}}\ne\mathbf{\overline{k}}$) cannot successfully invade. This result can be easily extended to the case where $n$ is finite but arbitrary.

A natural question arise:
what is the dynamics of model \eqref{m1} if neither $\mathbf u$ nor $\mathbf v$ can adopt the IFD strategy? In this paper, we aim to address this question.
For the two-patch scenario ($n=2$), $\mathbf{\overline{k}}={k_2}/{k_1}$, $\mathbf{p}=p_1$ and $\mathbf{\hat p}=\hat p_1$, and we show that the following results hold:
\begin{enumerate}
\item [(i)] If $p_1$ and ${\hat p_1}$ lie on the same side of IFD strategy
(i.e., $ p_1,{\hat p_1}>{k_2}/{k_1}$ or $ {k_2}/{k_1}> p_1,{\hat p_1}$), the mutant species $\mathbf v$ can outcompete and eliminate the resident species $\mathbf u$ if its strategy ${\hat p}$ is closer to  ${k_2}/{k_1}$ and its diffusion rates are are equal to or slower than those of the resident species;
\item [(ii)] If $ p_1$ and $ {\hat p_1}$ lie on opposite sides of IFD strategy (i.e., $ p_1>{k_2}/{k_1}> {\hat p_1}$ or $ {\hat p_1}>{k_2}/{k_1}> p_1$), the two species can coexist.
\end{enumerate}
The general form of this result for the case $n\ge2$ is presented in Theorem \ref{theorem3.8}.

The above results can be understood within the framework of adaptive dynamics. Detailed discussions on adaptive dynamics can be found in \cite{53-Dieckmann-1996,52-Diekmann-2003,54-Geritz-1998,lam2022,55-SmithPrice} and references therein. One key concept is a global (resp. local) evolutionarily stable strategy (ESS): a strategy is called a global (resp. local) ESS if a population adopting it cannot be invaded by a small population of mutants using a different (resp. different nearby) strategy \cite{57-Cantrell-2010,55-SmithPrice} (see Definition \ref{51}). A related concept is a global (resp. local) neighborhood invader strategy (NIS): a strategy is called a global (resp. local) NIS if a small population of mutants adopting it can always invade a population using a  different (resp. different nearby) strategy \cite{56-Apaloo,57-Cantrell-2010} (see Definition \ref{52}). Finally,
a strategy is called a local convergent stable strategy (CSS) if, for any resident strategy in its neighborhood, a small population of mutants with a strategy closer to it can always invade. Moreover, the definition of a global CSS follows by removing the neighborhood restriction (see Definition \ref{53}).

In the two-patch scenario, let the composite parameters $p_1$ and $\hat{p}_1$ be the strategies of $\mathbf{u}$ and $\mathbf{v}$, respectively, treating all other parameters as identical to isolate their effects.
The above-mentioned result in \cite{36-Maciel-2020} implies that the IFD strategy is a global ESS and NIS.
Our result (i) strengthens this by proving that the IFD strategy is also a global CSS. Moreover, result (ii) parallels findings in reaction-diffusion-advection competition models \cite{57-Cantrell-2010}, illustrating that two species can coexist when their strategies are on opposite sides of the IFD strategy.

The rest of the paper is organized as follows. Sect. 2 is devoted to the invasion analysis for model \eqref{m1}, and the local stability of the semi-trivial steady state is obtained. In Sect. 3, we studied the global dynamics of model \eqref{m1}.	For simplicity, we list some notations used throughout the paper.
For vectors $\mathbf w=(w_1,\cdots,w_n)\in\mathbb R^n$ and $\mathbf z=(z_1,\cdots,z_n)\in\mathbb R^n$, we denote
\begin{equation}\label{truct}
	[\mathbf w]_{\ell}:=(w_1,\cdots,w_{\ell}),\;\;[\mathbf w]^{\sharp}_{\ell}:=(w_{n-\ell+1},\cdots,w_{n})\;\;\text{for}\;\;1\le \ell\le n,
\end{equation}
and write
\begin{equation}\label{order1}
	\mathbf w\gg(\text{resp.}\ge)\mathbf z \;\;\text{if}\;\; w_i>(\text{resp.}\ge)z_i\;\;\text{for}\;\;i=1,\cdots,n,
\end{equation}
and $\mathbf w>\mathbf z$ if $\mathbf w\ge\mathbf z$ and $\mathbf w\neq \mathbf z$.  Denote by $(u)_x$ the derivative of $u$ with respect to $x$.

\section{Local dynamics and invasion analysis}
In each patch $\Omega_i$ ($i=1,2,\cdots,n$), we view species $u_i$ as the resident species and $v_i$ as the mutant species. It follows from Theorem \ref{apt1} in Appendix \ref{appA} that \eqref{m1} admits only two semi-trivial steady states $(\mathbf u^*,\mathbf 0)$ and $(\mathbf 0, \mathbf v^*)$, where $$\mathbf u^*=(u_1^*,\cdots,u_n^*)\in C^2(\overline{\Omega}_1)\times \cdots \times C^2(\overline{\Omega}_n)$$ is the unique positive solution of
\begin{subequations}\label{m2}
	\begin{numcases}{}
		d_i (u^*_{i})_{xx}+r_i u^*_i\left(1-\frac{u^*_i}{k_i}\right)=0, \;\;\;\;x\in \Omega_i,\;\; i=1,2,\cdots,n, \label{m2a}\\
		u^*_{i+1}(x_i^{+})=p_iu^*_i(x_i^{-}),\;\;d_{i+1}(u^*_{i+1})_x(x_i^{+})=d_i(u^*_i)_x(x_i^{-}),\label{m2b}\\
		(u^*_1)_x(0^{+})=(u^*_n)_x(L^{-})=0.\label{m2c}
	\end{numcases}
\end{subequations}

In this section, we analyze the local dynamics of model \eqref{m1} and establish a sufficient condition under which species $\mathbf v=(v_1,\cdots,v_n)$ can invade when rare. Specifically, the mutant species $\mathbf v=(v_1,\cdots,v_n)$ can (resp. cannot) invade when rare if the semi-trivial steady state $(\mathbf u^*,\mathbf 0)$ is unstable (resp. stable). The stability of $(\mathbf u^*,\mathbf 0)$ is determined by the principal eigenvalue of the following eigenvalue problem:
\begin{subequations}\label{m3}
	\begin{numcases}
		\ds \hat d_i (\phi_i)_{xx}+r_i\left(1-\dfrac{u^*_i}{k_i}\right) \phi_i=\lambda \phi_i, \;\; x\in \Omega_i,\;\; i=1,2,\cdots,n, \label{m3-a}\\
		\phi_{i+1}(x_i^{+})=\hat p_i\phi_i(x_i^{-}),\;\;\hat d_{i+1}(\phi_{i+1})_{x}(x_i^{+})=\hat d_i (\phi_{i})_{x}(x_i^{-}),\label{m3-b}\\
		(\phi_{1})_{x}(0^{+})=(\phi_{n})_{x}(L^{-})=0.\label{m3-c}
	\end{numcases}
\end{subequations}
Proposition \ref{prin-exist} in Appendix \ref{AppendixC} guarantees the existence of the principal eigenvalue  for problem \eqref{m3}.
\begin{lemma}
Let $\lambda_1 (\mathbf {\hat p},\mathbf {\hat d},\mathbf 1-\mathbf{u^*}/\mathbf k)$ be the principal eigenvalue of the eigenvalue problem \eqref{m3}.	
Then, the semi-trivial steady state $(\mathbf u^*,\mathbf 0)$ is stable if $\lambda_1 (\mathbf {\hat p},\mathbf {\hat d},\mathbf 1-\mathbf{u^*}/\mathbf k)<0$ and unstable if $\lambda_1 (\mathbf {\hat p},\mathbf{ \hat d},\mathbf 1-\mathbf{u^*}/\mathbf k )>0$.\\
(Here we abuse the notation by denoting $\mathbf 1-\mathbf{u^*}/\mathbf k:=\left(1-u_1^*/k_1,\cdots,1-u_n^*/k_n\right)$.)
\end{lemma}
\begin{proof}
  Model \eqref{m1} can be rewritten as the abstract evolution equation \eqref{abstract} in Appendix \ref{AppendixC}.
Then, by Remark \ref{linearstability}, a steady state of \eqref{m1} (or \eqref{abstract}) is asymptotically stable (resp., unstable)
if it is linearly stable (resp., linearly unstable). Accordingly, we consider the eigenvalue problem associated with the semi-trivial steady state $(\mathbf u^*,\mathbf 0)$:
\begin{equation}\label{mm3}
	\begin{cases}
d_i (\psi_{i})_{xx}+r_i \left(1-\ds\frac{2u^*_i}{k_i}\right)\psi_i-\dfrac{r_iu_i^*}{k_i}\phi_i=\mu\psi_i, \;\;\;\;x\in \Omega_i,\;\; i=1,2,\cdots,n, \\		\ds \hat d_i (\phi_i)_{xx}+r_i \left(1-\dfrac{u^*_i}{k_i}\right)\phi_i=\mu \phi_i, \;\; x\in \Omega_i,\;\; i=1,2,\cdots,n, \\
		\psi_{i+1}(x_i^{+})=p_i\psi_i(x_i^{-}),\;\;d_{i+1}(\psi_{i+1})_x(x_i^{+})=d_i(\psi_i)_x(x_i^{-}),\\
		\phi_{i+1}(x_i^{+})=\hat p_i\phi_i(x_i^{-}),\;\;\hat d_{i+1}(\phi_{i+1})_{x}(x_i^{+})=\hat d_i (\phi_{i})_{x}(x_i^{-}),\\
		(\phi_{1})_{x}(0^{+})=(\phi_{n})_{x}(L^{-})=(\psi_1)_x(0^{+})=(\psi_n)_x(L^{-})=0.
	\end{cases}
\end{equation}
We first show that
\begin{equation}\label{min}
	\operatorname{Re}\mu\le \max\{\lambda_1(\mathbf {\hat p},\mathbf {\hat d},\mathbf 1-\mathbf{u^*}/\mathbf k),\lambda_1(\mathbf { p},\mathbf {d},\mathbf 1-2\mathbf{u^*}/\mathbf k)\}.
\end{equation}
Here, $\lambda_1 (\mathbf {\hat p},\mathbf {\hat d},\mathbf 1-\mathbf{u^*}/\mathbf k)$ is the principal eigenvalue of \eqref{m3}, and
$\lambda_1(\mathbf { p},\mathbf {d},\mathbf 1-2\mathbf{u^*}/\mathbf k)$ corresponds to that of
\begin{equation}\label{m32}
	\begin{cases}
		\ds d_i (\psi_i)_{xx}+r_i \left(1-\dfrac{2u^*_i}{k_i}\right)\psi_i=\lambda \psi_i, \;\; x\in \Omega_i,\;\; i=1,2,\cdots,n, \\
		\psi_{i+1}(x_i^{+})= p_i\psi_i(x_i^{-}),\;\;d_{i+1}(\psi_{i+1})_{x}(x_i^{+})= d_i (\psi_{i})_{x}(x_i^{-}),\\
		(\psi_{1})_{x}(0^{+})=(\psi_{n})_{x}(L^{-})=0.
	\end{cases}
\end{equation}
The existence of the principal eigenvalue for \eqref{m32} is also a consequence of  Proposition \ref{prin-exist} in Appendix \ref{AppendixC}. Furthermore, this proposition shows that $\operatorname{Re}\mu\le \lambda(\mathbf {\hat p},\mathbf {\hat d},\mathbf 1-\mathbf{u^*}/\mathbf k)$ when $\Phi=(\phi_1,\cdots,\phi_n) \not \equiv \mathbf0$. If $\Phi \equiv \mathbf 0$, then $\Psi= (\psi_1,\cdots,\psi_n)\not \equiv \mathbf 0$, and we obtain that $\operatorname{Re} \mu\le \lambda(\mathbf { p},\mathbf {d},\mathbf 1-2\mathbf{u^*}/\mathbf k)$. Therefore,
\eqref{min} holds.

Now we prove that $\lambda_1(\mathbf { p},\mathbf {d},\mathbf 1-2\mathbf{u^*}/\mathbf k)<0$. 	Multiplying the first equation of \eqref{m32} with $(\prod_{\ell=i}^{n}p_\ell)u_i^*$, integrating the result over $\Omega_i$ and summing them  over all $i$, we have
\begin{equation}\label{2l11}
	\begin{aligned}
		&\lambda_1\left(\mathbf { p},\mathbf {d},\mathbf 1-\frac{2\mathbf u^*}{\mathbf k}\right) \left(\sum_{i=1}^{n}(\prod_{\ell=i}^{n}p_\ell)\int_{x_{i-1}}^{x_i}{\psi_iu_i^*}{\rm d}x\right) \\
		=&\sum_{i=1}^{n}{d_i}(\prod_{\ell=i}^{n}p_\ell)\int_{x_{i-1}}^{x_i}{(\psi_{i})_{xx}u_i^*}{\rm d}x+\sum_{i=1}^{n}(\prod_{\ell=i}^{n}p_\ell)\int_{x_{i-1}}^{x_i}{r_i \psi_i u^*_i \left(1-\dfrac{2u^*_i}{k_i}\right)}{\rm d}x.\\
	\end{aligned}
\end{equation}
Similarly, 	we deduce from \eqref{m2a} that
\begin{equation}\label{2l21}
	0=\sum_{i=1}^{n}d_i(\prod_{\ell=i}^{n}p_\ell)\int_{x_{i-1}}^{x_i}{(u^*_{i})_{xx}\psi_i}{\rm d}x+\sum_{i=1}^{n}(\prod_{\ell=i}^{n}p_\ell)\int_{x_{i-1}}^{x_i}{r_i \psi_i u^*_i \left(1-\dfrac{u^*_i}{k_i}\right)}{\rm d}x.
\end{equation}
Subtracting \eqref{2l21} from \eqref{2l11} and  applying integration by parts yields
\begin{equation*}
\lambda_1(\mathbf { p},\mathbf {d},\mathbf 1-2\mathbf{u^*}/\mathbf k) \left(\sum_{i=1}^{n}(\prod_{\ell=i}^{n}p_\ell)\int_{x_{i-1}}^{x_i}{\psi_iu_i^*}{\rm d}x\right)=-\sum_{i=1}^{n}(\prod_{\ell=i}^{n}p_\ell)\int_{x_{i-1}}^{x_i}{\dfrac{r_iu_i^*}{k_i}\psi_i}{\rm d}x<0.
\end{equation*}
This implies that
$\lambda_1(\mathbf { p},\mathbf {d},\mathbf 1-2\mathbf{u^*}/\mathbf k)<0$. The desired result follows from this and the fact that $\lambda_1(\mathbf {\hat p},\mathbf {\hat d},\mathbf 1-\mathbf{u^*}/\mathbf k)$ is also an eigenvalue of \eqref{mm3}.
\end{proof}

\begin{lemma}\label{lemma2.2}
	Let  $\lambda_1 (\mathbf {\hat p},\mathbf {\hat d},\mathbf 1-\mathbf{u^*}/\mathbf k)$ be the principal eigenvalue of \eqref{m3} with  corresponding eigenfunction $\Phi=(\phi_1(x),\cdots,\phi_n(x))$, where $\phi_i>0$ in $\overline\Omega_i$ for each $i=1,\cdots,n$. Then the following identity holds:
	\begin{equation}\label{id}
		\begin{aligned}
			&\lambda_1 (\mathbf {\hat p},\mathbf {\hat d},\mathbf 1-\mathbf{u^*}/\mathbf k) \left(\sum_{i=1}^{n}(\prod_{\ell=i}^{n}p_\ell)\int_{x_{i-1}}^{x_i}{\phi_iu_i^*}{\rm d}x\right) \\
			=&\sum_{i=1}^{n-1}(\prod_{\ell=i+1}^{n}p_\ell)d_{i}(\hat{p}_i-p_i)(u^*_{i})_{x}(x_i^-)\phi_i(x_i^-) +\sum_{i=1}^{n}(\prod_{\ell=i}^{n}p_\ell)(d_i-\hat d_i)\int_{x_{i-1}}^{x_i}(\phi_{i})_{x}(u^*_{i})_{x}{\rm d}x.
		\end{aligned}
	\end{equation}
	Here we define $p_n=1$.		
\end{lemma}
\begin{proof}
	Multiplying \eqref{m3-a} with $(\prod_{\ell=i}^{n}p_\ell)u_i^*$, integrating the result over $\Omega_i$ and summing them  over all $i$, we have
	\begin{equation}\label{2l1}
		\begin{aligned}
			&\lambda_1 (\mathbf {\hat p},\mathbf {\hat d},\mathbf 1-\mathbf{u^*}/\mathbf k) \left(\sum_{i=1}^{n}(\prod_{\ell=i}^{n}p_\ell)\int_{x_{i-1}}^{x_i}{\phi_iu_i^*}{\rm d}x\right) \\
			=&\sum_{i=1}^{n}\hat{d_i}(\prod_{\ell=i}^{n}p_\ell)\int_{x_{i-1}}^{x_i}{(\phi_{i})_{xx}u_i^*}{\rm d}x+\sum_{i=1}^{n}(\prod_{\ell=i}^{n}p_\ell)\int_{x_{i-1}}^{x_i}{r_i \phi_i u^*_i \left(1-\dfrac{u^*_i}{k_i}\right)}{\rm d}x.\\
		\end{aligned}
	\end{equation}
	Similarly, 	we deduce from \eqref{m2a} that
	\begin{equation}\label{2l2}
		0=\sum_{i=1}^{n}d_i(\prod_{\ell=i}^{n}p_\ell)\int_{x_{i-1}}^{x_i}{(u^*_{i})_{xx}\phi_i}{\rm d}x+\sum_{i=1}^{n}(\prod_{\ell=i}^{n}p_\ell)\int_{x_{i-1}}^{x_i}{r_i \phi_i u^*_i \left(1-\dfrac{u^*_i}{k_i}\right)}{\rm d}x.
	\end{equation}
	Subtracting \eqref{2l2} from \eqref{2l1} and integrating by parts, we can obtain the desired result from \eqref{m2b}-\eqref{m2c} and \eqref{m3-b}-\eqref{m3-c}.
\end{proof}

The following result is a direct consequence of Lemma \ref{lemma2.2}.
\begin{lemma}\label{theorem2.3}
	Suppose that $\mathbf d=\hat{\mathbf d}$, and let $\mathbf{\overline k}$ be defined in \eqref{ok}. Then the following statements hold for model \eqref{m1}:
	\begin{enumerate}
		\item [{$\rm (i)$}]  If $\mathbf p\gg \mathbf{\overline k}$, the semi-trivial steady state $(\mathbf u^*,\mathbf 0)$ is unstable for $\mathbf p\gg\hat{\mathbf p }$ and stable for $\hat{\mathbf p }\gg \mathbf p $;
		\item [{$\rm (ii)$}]  If $\mathbf{\overline k}\gg \mathbf p$, the semi-trivial steady state $(\mathbf u^*,\mathbf 0)$ is unstable for $\hat{\mathbf p }\gg \mathbf p$ and stable for $\mathbf p\gg\hat{\mathbf p }$.
	\end{enumerate}
\end{lemma}
\begin{proof}
	Since $\mathbf d=\hat{\mathbf d}$, it follows from Lemma \ref{lemma2.2} that
	\begin{equation*}
		\begin{split}
			&\lambda_1 (\mathbf {\hat p},\mathbf {\hat d},\mathbf 1-\mathbf{u^*}/\mathbf k) \left(\sum_{i=1}^{n}(\prod_{\ell=i}^{n}p_\ell)\int_{x_{i-1}}^{x_i}{\phi_iu_i^*}{\rm d}x\right) \\ &=\sum_{i=1}^{n-1}(\prod_{\ell=i+1}^{n}p_\ell)d_{i}(\hat{p}_i-p_i)(u^*_{i})_{x}(x_i^-)\phi_i(x_i^-).
		\end{split}
	\end{equation*}
	Here $p_n=1$ and $(\phi_1(x),\cdots,\phi_n(x))$ is the eigenfunction corresponding to $\lambda_1 (\mathbf {\hat p},\mathbf {\hat d},\mathbf 1-\mathbf{u^*}/\mathbf k)$, where $\phi_i>0$ in $\overline\Omega_i$ for all $i=1,\cdots,n$.
	Thus, the desired result follows from this and Lemma \ref{lemma2.1-n}.
\end{proof}

\begin{lemma}\label{lemma2.3}
	Suppose that $\lambda_1 (\mathbf {\hat p},\mathbf {\hat d},\mathbf 1-\mathbf{u^*}/\mathbf k)=0$,  where $\lambda_1 (\mathbf {\hat p},\mathbf {\hat d},\mathbf 1-\mathbf{u^*}/\mathbf k)$ is the principal eigenvalue of \eqref{m3} and $\Phi=(\phi_1(x),\cdots,\phi_n(x))$ is the corresponding eigenfunction with $\phi_i(x)>0$ for $x\in \overline\Omega_i$ and $i=1,\cdots,n$. Let $x_{\ell_*-1}\le a< x_{\ell_*}$ and $x_{\ell^*-1}< b \le x_{\ell^*}$ for some $\ell_*$ and $\ell^*$ with
	$1\le \ell_*< \ell^*\le n$.
	Then the following two identities hold:
	\begin{equation}\label{id-1}
		\begin{aligned}
			0=&\sum_{i=\ell_*}^{\ell^*-1}(\prod_{\ell=i+1}^{\ell^*}p_\ell)d_{i}(\hat{p}_i-p_i)(u^*_{i})_{x}(x_i^-)\phi_i(x_i^-)+
			\hat {d}_{\ell^*}p_{\ell^*}(\phi_{\ell^*})_x(b^-)u^*_{\ell^*}(b^-)\\ &-\hat {d}_{\ell_*}(\prod_{\ell=\ell_*}^{\ell^*}p_\ell)(\phi_{\ell_*})_x(a^+)u^*_{\ell_*}(a^+) +\sum_{i=\ell_*}^{\ell^*}(\prod_{\ell=i}^{\ell^*}p_\ell)(d_i-\hat d_i)\int_{\Omega_i\cap(a,b)}(\phi_{i})_{x}(u^*_{i})_{x}{\rm d}x\\
			&-{d}_{\ell^*}p_{\ell^*}(u^*_{\ell^*})_x(b^-)\phi_{_{\ell^*}}(b^-) +{d}_{\ell_*}(\prod_{\ell=\ell_*}^{\ell^*}p_\ell)(u^*_{\ell_*})_x(a^+)\phi_{_{\ell_*}}(a^+),
		\end{aligned}
	\end{equation}
	and
	\begin{equation}\label{uv4-1}
		\begin{aligned}
			0=&\sum_{i=\ell_*}^{\ell^*-1} \dfrac{\hat{d}_i(k_i\hat{p}_i-k_{i+1})(\phi_{i})_{x}(x_i^-)}{\hat{p}_i\phi_i(x_i^-)}+\dfrac{\hat{d}_{\ell^*}k_{\ell^*}(\phi_{\ell^*})_{x}(b^-)}{\phi_{\ell^*}(b^-)}-\dfrac{\hat{d}_{\ell_*}k_{\ell_*}(\phi_{\ell_*})_{x}(a^+)}{\phi_{\ell_*}(a^+)}\\
			&+\sum_{i=\ell_*}^{\ell^*}\int_{\Omega_i\cap(a,b)}\left[\dfrac{\hat{d}_ik_i(\phi_{i})_{x}^2}{\phi_i^2}+\dfrac{r_i(k_i-u_i^*)^2}{k_i}  \right]{\rm d}x-d_{\ell^*}(u^*_{\ell^*})_x(b^-)+d_{\ell_*}(u^*_{\ell_*})_x(a^+).\\
		\end{aligned}
	\end{equation}
	Here we define $p_n=1$.
\end{lemma}
\begin{proof}
	Letting $\lambda_1 (\mathbf {\hat p},\mathbf {\hat d},\mathbf 1-\mathbf{u^*}/\mathbf k)=0$ and replacing $(1,n)$ by $(l_*,l^*)$ in the proof of Lemma \ref{lemma2.2}, we can obtain that \eqref{id-1} holds.
	Now we show that \eqref{uv4-1} holds, and the proof is inspired by \cite{36-Maciel-2020}. Dividing \eqref{m3-a} by $\phi_i/k_i$, integrating the result over $\Omega_i\cap (a,b)$ and summing them from $i=\ell_*$ to $i=\ell^*$, we see from integration by parts and \eqref{m3-b}  that
	\begin{equation}\label{uv5}
		\begin{aligned}
			0=&\sum_{i=\ell_*}^{\ell^*-1} \dfrac{\hat{d}_i(k_i\hat{p}_i-k_{i+1})(\phi_{i})_{x}(x_i^-)}{\hat{p}_i\phi_i(x_i^-)}+\dfrac{\hat{d}_{\ell^*}k_{\ell^*}(\phi_{\ell^*})_{x}(b^-)}{\phi_{\ell^*}(b^-)}-\dfrac{\hat{d}_{\ell_*}k_{\ell_*}(\phi_{\ell_*})_{x}(a^+)}{\phi_{\ell_*}(a^+)}\\			&+\sum_{i=\ell_*}^{\ell^*}\int_{\Omega_i\cap(a,b)}\left[\dfrac{\hat{d}_ik_i(\phi_{i})_{x}^2}{\phi_i^2}+r_ik_i\left(1-\dfrac{u^*_i}{k_i}\right)  \right]{\rm d}x.
		\end{aligned}
	\end{equation}
	In addition, integrating \eqref{m2a} over $\Omega_i\cap(a,b)$ and summing them from $i=\ell_*$ to $i=\ell^*$ gives
	\begin{equation}\label{uv3}
		\sum_{i=\ell_*}^{\ell^*}\int_{\Omega_i\cap(a,b)} r_iu_i^*\left(1-\dfrac{u_i^*}{k_i}\right){\rm d}x+d_{\ell^*}(u_{\ell^*})_x(b^-)-d_{\ell_*}(u_{\ell_*})_x(a^+)=0.
	\end{equation}
	Subtracting \eqref{uv3} from \eqref{uv5}, we see that \eqref{uv4-1} holds.
\end{proof}
Then we obtain two results on the signs of $\lambda_1 (\mathbf {\hat p},\mathbf {\hat d},\mathbf 1-\mathbf{u^*}/\mathbf k)$. The first one is as follows.
\begin{lemma}\label{lemma2.5}
	Let $\mathbf{\overline k}$ be defined in \eqref{ok} and let $\lambda_1 (\mathbf {\hat p},\mathbf {\hat d},\mathbf 1-\mathbf{u^*}/\mathbf k)$ be the principal eigenvalue of \eqref{m3}. Then the following two statements hold:
	\begin{enumerate}
		\item [{$\rm (i)$}]  If $\mathbf p\gg \mathbf{\overline k}$, then $\lambda_1 (\mathbf {\hat p},\mathbf {\hat d},\mathbf 1-\mathbf{u^*}/\mathbf k)>0$ for $(\mathbf {\hat p},\mathbf {\hat d})\in \mathcal L_1$ and $\lambda_1 (\mathbf {\hat p},\mathbf {\hat d},\mathbf 1-\mathbf{u^*}/\mathbf k)<0$ for $(\mathbf {\hat p},\mathbf {\hat d})\in \mathcal L_2$, where
		\begin{equation}\label{l1l2}
			\mathcal L_1:=\{(\mathbf {\hat p},\mathbf {\hat d}):\mathbf p\gg\mathbf {\hat p},\;\mathbf d\ge \mathbf{\hat d}\},\;\;\mathcal L_2:=\{(\mathbf {\hat p},\mathbf {\hat d}):\mathbf {\hat p}\gg \mathbf p,\;\mathbf{\hat d}\ge \mathbf d\};
		\end{equation}
		\item [{$\rm (ii)$}]  If $\mathbf{\overline k}\gg \mathbf p$, then $\lambda_1 (\mathbf {\hat p},\mathbf {\hat d},\mathbf 1-\mathbf{u^*}/\mathbf k)>0$ for $(\mathbf {\hat p},\mathbf {\hat d})\in \mathcal S_1$ and $\lambda_1 (\mathbf {\hat p},\mathbf {\hat d},\mathbf 1-\mathbf{u^*}/\mathbf k)<0$ for $(\mathbf {\hat p},\mathbf {\hat d})\in \mathcal S_2$, where
		\begin{equation}\label{s1s2}
			\mathcal S_1:=\{(\mathbf {\hat p},\mathbf {\hat d}):\mathbf {\hat p}\gg\mathbf p,\; \mathbf d\ge\mathbf{\hat d}\},\;\;
			\mathcal S_2:=\{(\mathbf {\hat p},\mathbf {\hat d}): \mathbf p\gg\mathbf {\hat p},\;\mathbf{\hat d}\ge \mathbf d\}.
		\end{equation}
	\end{enumerate}
\end{lemma}
\begin{proof}
	We only prove $\rm (i)$, and $\rm (ii)$ can be treated similarly. By Lemma \ref{theorem2.3},
	we see that $\lambda_1 (\mathbf {\hat p},\mathbf { \hat d},\mathbf 1-\mathbf{u^*}/\mathbf k)>0$ (resp. $<0$) for $\mathbf p\gg\hat{\mathbf p } $ (resp. $\hat{\mathbf p }\gg \mathbf p$) if $\mathbf { \hat d}=\mathbf {  d}$. 
	Note from Proposition \ref{prin-exist} in Appendix \ref{AppendixC} that $\lambda_1 (\mathbf {\hat p},\mathbf {\hat d},\mathbf 1-\mathbf{u^*}/\mathbf k)$ is continuous with respect to $\mathbf {\hat d}$.
	Noting that $\mathcal L_i$ $(i=1,2)$ are convex, it suffices to show that $\lambda_1 (\mathbf {\hat p},\mathbf {\hat d},\mathbf 1-\mathbf{u^*}/\mathbf k) \ne 0$ for $(\mathbf {\hat p},\mathbf {\hat d})\in \mathcal L_1\cup\mathcal L_2$.
	Suppose to the contrary that $\lambda_1 (\mathbf {\hat p},\mathbf {\hat d},\mathbf 1-\mathbf{u^*}/\mathbf k) = 0$ for some $(\mathbf {\hat p},\mathbf {\hat d})\in \mathcal L_1\cup\mathcal L_2$ and denote the corresponding eigenfunction by $\mathbf \phi=( \phi_1(x),\cdots,\phi_n(x))$ with $\phi_i(x)>0$ for $x\in\overline\Omega_i$ and $i=1,\cdots,n$.
	Then we will obtain a contradiction for each of the following two cases: (i$_{1}$) $(\mathbf {\hat p},\mathbf {\hat d})\in \mathcal L_1$; (i$_{2}$) $(\mathbf {\hat p},\mathbf {\hat d})\in \mathcal L_2$.
	For simplicity, we only obtain a contradiction for 	(i$_1$), and (i$_{2}$) can be treated similarly.	
	
	By Lemma \ref{lemma2.1-n}, $u^*_1(0^+)<k_1$. Then we deduce (by \eqref{m3-a}) that $(\phi_1)_{xx}(0^+)<0$. This combined with $(\phi_1)_{x}(0^+)=0$ implies that
	$(\phi_1)_x(x)<0$ for $x\in(0,\epsilon)$ with $0<\epsilon\ll1$.
	Since $(\phi_1)_{x}(L^-)=0$, it follows that
	\begin{equation*}
		y_1=\inf \{y>0:\exists 1\le \ell\le n\;\; \text{s.t.}\;\; y\in (x_{\ell-1},x_\ell] \;\;\text{and} \;\; (\phi_\ell)_x(y^-)=0\}
	\end{equation*}
	is well-defined with $y_1\in (x_{\ell_1-1},x_{\ell_1}]$ for some $\ell_1$ with $1\le \ell_1\le n$.  Then
	\begin{equation}\label{phiell1} (\phi_{\ell_1})_x(y_1^-)=0\;\;\text{and}\;\;(\phi_i)_x(x)<0\;\;\text{for}\;\;x\in\Omega_i\cap(0,y_1)\;\;\text{with}\;\;i=1,\cdots,\ell_1.
	\end{equation}
	Substituting $a=0$, $b=y_1$, $\ell_*=1$ and $\ell^*=\ell_1$ into \eqref{id-1}, we obtain that
	\begin{equation}\label{1ell1}
		\begin{aligned}
			0=&\sum_{i=1}^{\ell_1-1}(\prod_{\ell=i+1}^{\ell_1}p_\ell)d_{i}(\hat{p}_i-p_i)(u^*_{i})_{x}(x_i^-)\phi_i(x_i^-)+\hat {d}_{\ell_1}p_{\ell_1}(\phi_{_{\ell_1}})_x(y_1^-)u^*_{\ell_1}(y_1^-) \\ &+\sum_{i=1}^{\ell_1}(\prod_{\ell=i}^{\ell_1}p_\ell)(d_i-\hat d_i)\int_{\Omega_i\cap(0,y_1)}(\phi_{i})_{x}(u^*_{i})_{x}{\rm d}x-{d}_{\ell_1}p_{\ell_1}(u^*_{\ell_1})_x(y_1^-)\phi_{_{\ell_1}}(y_1^-)>0,
		\end{aligned}
	\end{equation}
	where we have used $(\mathbf {\hat p},\mathbf {\hat d})\in \mathcal L_1$, \eqref{phiell1} and Lemma \ref{lemma2.1-n} in the last step. This leads to a contradiction and completes the proof.
\end{proof}
In the following, we show that the sign of $\lambda_1 (\mathbf {\hat p},\mathbf {\hat d},\mathbf 1-\mathbf{u^*}/\mathbf k)$ can be determined without additional assumption on $\mathbf d,\mathbf {\hat d}$ if $\mathbf p$ and $\mathbf {\hat p}$ are opposite sides of $\mathbf{\overline k}$ with respect to the order defined in \eqref{order1}.
\begin{lemma}\label{lemma2.5-s}
	Let $\mathbf{\overline k}$ be defined in \eqref{ok} and let $\lambda_1 (\mathbf {\hat p},\mathbf {\hat d},\mathbf 1-\mathbf{u^*}/\mathbf k)$ be the principal eigenvalue of \eqref{m3}. Then the following statements hold:
	\begin{enumerate}
		\item [{$\rm (i)$}]  If $\mathbf p\gg \mathbf{\overline k}$, then $\lambda_1 (\mathbf {\hat p},\mathbf {\hat d},\mathbf 1-\mathbf{u^*}/\mathbf k)> 0$ for $(\mathbf {\hat p},\mathbf {\hat d})\in \mathcal L_3$, where
		\begin{equation}\label{l3}
			\mathcal L_3:=\{(\mathbf {\hat p},\mathbf {\hat d}):\mathbf{\overline{k}}\gg\mathbf {\hat p}\};
		\end{equation}
		\item [{$\rm (ii)$}]  If $\mathbf{\overline k}\gg\mathbf p$, then $\lambda_1 (\mathbf {\hat p},\mathbf {\hat d},\mathbf
		1-\mathbf{u^*}/\mathbf k) >0$ for $(\mathbf {\hat p},\mathbf {\hat d})\in \mathcal S_3$, where
		\begin{equation}\label{s3}
			\mathcal S_3:=\{(\mathbf {\hat p},\mathbf {\hat d}):\mathbf {\hat p}\gg\mathbf{\overline{k}}\}.
		\end{equation}
	\end{enumerate}
\end{lemma}
\begin{proof}
	We only prove $\rm (i)$, and $\rm (ii)$ can be treated similarly.
	Similar to Lemma \ref{lemma2.5}, it suffices to show that $\lambda_1 (\mathbf {\hat p},\mathbf {\hat d},\mathbf 1-\mathbf{u^*}/\mathbf k) \ne 0$ for $(\mathbf {\hat p},\mathbf {\hat d})\in \mathcal L_3$.
	Suppose to the contrary that $\lambda_1 (\mathbf {\hat p},\mathbf {\hat d},\mathbf 1-\mathbf{u^*}/\mathbf k) = 0$ for some $(\mathbf {\hat p},\mathbf {\hat d})\in \mathcal L_3$ and denote the corresponding eigenfunction by $\mathbf \phi=( \phi_1,\cdots,\phi_n)$ with $\phi_i(x)>0$ for $x\in\overline\Omega_i$ and $i=1,\cdots,n$.
	
	For $x\in(0,x_1]$, integrating \eqref{m3-a} over $(0,x)$ and noticing that $\lambda_1 (\mathbf {\hat p},\mathbf {\hat d},\mathbf 1-\mathbf{u^*}/\mathbf k) = 0$, we obtain that
	\begin{equation*}
		\hat{d}_1(\phi_{1})_{x}(x^-) =-\int_{0}^{x}{r_1\phi_{1}\left(1-\dfrac{u^*_1}{k_1}\right)}{\rm d}x<0,
	\end{equation*}
	where we have used Lemma \ref{lemma2.1-n} (i) in the last step. This combined with \eqref{m3-c} implies that
	\begin{equation*}
		y_0=\inf \{y:\exists \;2\le \ell\le n \text{ s.t. } y\in (x_{\ell-1},x_\ell]\;\;\text{and}\;\;(\phi_{\ell})_x(y^-)=0\}
	\end{equation*}
	is well-defined with $x_{\ell_0-1}<y_0\le x_{\ell_0}$ for some $\ell_0$ with $2\le \ell_0\le n$. Thus,
	\begin{equation}\label{lm4-6}
		(\phi_{\ell_0})_x(y_0^-)=0,\;\;(\phi_{\ell})_x(x^-)<0\;\;\text{for}\;\;x\in(x_{\ell-1}, x_{\ell}]\cap(0,y_0),\;\;\ell=1,\cdots,\ell_0.
	\end{equation}
	Noticing that $\lambda_1 (\mathbf {\hat p},\mathbf {\hat d},\mathbf 1-\mathbf{u^*}/\mathbf k) = 0$ and substituting $a=0$, $b=y_0$, $\ell_*=1$ and $\ell^*=\ell_0$ into \eqref{uv4-1}, we have
	\begin{equation*}
		\begin{aligned}
			0=&\sum_{i=1}^{\ell_0-1} \dfrac{\hat{d}_i(k_i\hat{p}_i-k_{i+1})(\phi_{i})_{x}(x_i^-)}{\hat{p}_i\phi_i(x_i^-)}+d_{1}(u^*_{1})_x(0^+)\\			&+\sum_{i=1}^{\ell_0}\int_{\Omega_i\cap(0,y_0)}\left[\dfrac{\hat{d}_ik_i(\phi_{i})_{x}^2}{\phi_i^2}+\dfrac{r_i(k_i-u_i^*)^2}{k_i}  \right]{\rm d}x-d_{\ell_0}(u^*_{\ell_0})_x(y_0^-)\\
			&+\dfrac{\hat{d}_{\ell_0}k_{\ell_0}(\phi_{\ell_0})_{x}(y_0^-)}{\phi_{\ell_0}(y_0^-)}-\dfrac{\hat{d_{1}}k_{1}(\phi_{1})_{x}(0^+)}{\phi_{1}(0^+)}>0,
		\end{aligned}
	\end{equation*}
	where we have used \eqref{m2c}, \eqref{m3-c}, \eqref{lm4-6}, $\mathbf{ \overline{k}}\gg \mathbf{\hat p}$ and Lemma \ref{lemma2.1-n} in the last step. This leads to a contradiction and  completes the proof.
\end{proof}	
The following result is a direct consequence of Lemmas \ref{lemma2.5} and \ref{lemma2.5-s}.
\begin{theorem}\label{theorem2.8}
	Let $\mathbf{\overline k}$ be defined in \eqref{ok}. Then the following statements  hold for model \eqref{m1}:
	\begin{enumerate}
		\item [{$\rm (i)$}] If $\mathbf p\gg \mathbf{\overline{k}}$, then the semi-trivial steady state $(\mathbf{u^*},\mathbf 0)$ is unstable for $(\mathbf {\hat p},\mathbf {\hat d})\in \mathcal L_1\cup\mathcal L_3 $ and stable for $(\mathbf {\hat p},\mathbf {\hat d})\in \mathcal L_2$, where $\mathcal L_1,\mathcal L_2$ are defined in \eqref{l1l2} and $\mathcal L_3$ is defined in \eqref{l3};
		\item [{$\rm (ii)$}] If $\mathbf{\overline{k}}\gg\mathbf p$, then the semi-trivial steady state $(\mathbf{u^*},\mathbf 0)$ is unstable for $(\mathbf {\hat p},\mathbf {\hat d})\in \mathcal S_1\cup\mathcal S_3 $ and stable for $(\mathbf {\hat p},\mathbf {\hat d})\in \mathcal S_2$, where $\mathcal S_1,\mathcal S_2$ are defined in \eqref{s1s2} and $\mathcal S_3$ is defined in \eqref{s3}.
	\end{enumerate}
\end{theorem}

Note that the mutant species $\mathbf v=(v_1,\cdots,v_n)$ can (resp. cannot) invade when rare if the semi-trivial steady state $(\mathbf u^*,\mathbf 0)$ is unstable (resp. stable).  Then, it follows from Theorem \ref{theorem2.8} that, for fixed $(\mathbf d,\mathbf q)$ with $\mathbf p\gg \mathbf{\overline{k}}$, the mutant species $\mathbf v$ can invade when rare if $(\mathbf {\hat p},\mathbf {\hat d})\in \mathcal L_1\cup\mathcal L_3 $; and for fixed $(\mathbf d,\mathbf q)$ with $\mathbf{\overline{k}}\gg\mathbf p$, the mutant species $\mathbf v$ can invade when rare if $(\mathbf {\hat p},\mathbf {\hat d})\in\mathcal S_1\cup\mathcal S_3 $. Finally, we summarize the results on invasion analysis in the following: (see Tables \ref{tablepk} and \ref{tablekp})

\begin{table}[htbt]
	\centering
	\caption{Summary of invasion analysis for $\mathbf p\gg \mathbf{\overline{k}}$. Here `` $1$ '' (resp. `` $0$ ") represents that species $\mathbf v$ can (resp. cannot) invade when rare.}
	\label{tablepk}
	\vspace{5pt}
	\begin{tabular}{ccc}
		\toprule[1pt] 
		Parameter condition for $\hat{\mathbf p}$ & Parameter condition for $\hat{\mathbf d}$ & Results\\
		\midrule[1pt]
		$\hat{\mathbf p}\gg\mathbf p$ &$\hat{\mathbf d}\ge\mathbf d$ &$0$\\
		$\mathbf p\gg\hat{\mathbf p}$   &$\mathbf d\ge \hat{\mathbf d}$ &$1$ \\
		$\mathbf{ \overline{k}}\gg\mathbf{\hat{p}}$  &none &$1$ \\
		\bottomrule[1pt]
	\end{tabular}
\end{table}
\begin{table}[htbt]
	\centering
	\caption{Summary of invasion analysis for $\mathbf{\overline{k}}\gg \mathbf p$. Here `` $1$ '' (resp. `` $0$ ") represents that species $\mathbf v$ can (resp. cannot) invade when rare.}
	\label{tablekp}
	\vspace{5pt}
	\begin{tabular}{ccc}
		\toprule[1pt] 
		Parameter condition for $\hat{\mathbf p}$ & Parameter condition for $\hat{\mathbf d}$ & Results\\
		\midrule[1pt]
		$\hat{\mathbf p}\gg\mathbf p$ &$\mathbf d \ge \hat{\mathbf d}$ &$1$\\
		$\mathbf p\gg\hat{\mathbf p}$   &$\hat{\mathbf d}\ge\mathbf d$ &$0$ \\
		$\mathbf{\hat{p}}\gg\mathbf{ \overline{k}}$  &none &$1$ \\
		\bottomrule[1pt]
	\end{tabular}
\end{table}

In the following, we further develop the concepts of adaptive dynamics introduced earlier, and show that our result in the two-patch scenario is related to the concept of convergent stable strategy (CSS). For the case $n=2$,
let $p_1$ and $\hat{p}_1$ be the strategies for the resident species $\mathbf{u}$ and mutant species $\mathbf{v}$, respectively, treating all other parameters as identical to isolate their effects.
The invasion fitness is given by the principal eigenvalue $\lambda_1(\hat{p}_1, \mathbf{\hat{d}}, \mathbf{1} - \mathbf{u^*}/\mathbf{k})$. Since   $\mathbf{u^*}$ depends on the resident's strategy  $p_1$, we denote $\lambda_1(\hat{p}_1, \mathbf{\hat{d}}, \mathbf{1} - \mathbf{u^*}/\mathbf{k})$ by  $\lambda_1(p_1, \hat{p}_1)$ in the subsequent analysis.

	\begin{definition}\label{51}\cite{57-Cantrell-2010,lam2022,55-SmithPrice}
		Suppose that $n=2$.
 A strategy $p_1^*$  is called a local evolutionarily stable strategy (ESS) if there exists $\delta>0$ such that $\lambda_1(p_1^*, \hat{p}_1)<0$ for all $\hat{p}_1\in (p_1^*-\delta, p_1^*+\delta)\setminus\{p_1^*\}$. It is called a global ESS if the above condition holds for $\delta=\infty$.
	\end{definition}
	\begin{definition}\label{52}\cite{56-Apaloo,57-Cantrell-2010,lam2022}
		Suppose that $n=2$. The strategy $\hat{p}_1^*$  is called a local neighborhood invader strategy (NIS) if
		there exists $\delta>0$ such that $\lambda_1(p_1, \hat{p}^*_1)>0$ for all $ p_1\in (\hat{p}_1^*-\delta, \hat{p}_1^*+\delta)\setminus\{\hat{p}_1^*\}$. It is called a global NIS if the above condition holds for $\delta=\infty$.
	\end{definition}
	\begin{definition}\label{53}\cite{57-Cantrell-2010,58-LamBMB}
	Suppose that $n=2$.  A strategy $p_1^*$ is called a local convergent stable strategy (CSS) if there exists $\delta>0$ such that
		\begin{equation*}
		\lambda_1(p_1, \hat{p}_1)=\begin{cases}
				<0\;\;\text{if}\;\; p_1^*\le p_1<\hat{p}_1<p_1^*+\delta\;\text{or}\;p_1^*-\delta< \hat{p}_1<p_1\le p_1^*,\\
				>0\;\;\text{if}\;\; p_1^*\le \hat{p}_1<p_1<p_1^*+\delta\;\text{or}\;p_1^*-\delta< p_1<\hat{p}_1\le p_1^*.
			\end{cases}
		\end{equation*}
It is called a global CSS if the above condition holds for $\delta=\infty$.		
	\end{definition}

	\begin{remark}
		For the two-patch case, it was shown in \cite{36-Maciel-2020} that the IFD strategy $\overline{\mathbf{k}}=k_2/k_1$ is a global ESS and NIS. 
		Our result in Theorem \ref{theorem2.8} (or Lemma \ref{theorem2.3}) implies that the IFD strategy is a global CSS.
	\end{remark}

\section{Global dynamics}
In the subsequent analysis, we investigate the global dynamics of of model \eqref{m1} using the theory of monotone dynamical systems. Subsect. \ref{ss1} is devoted to the nonexistence of positive steady state for model \eqref{m1}, while the main result is stated in Subsect. \ref{ss2}. We note that partial proofs in this section are inspired by the method developed in \cite{36-Maciel-2020,51-Zhou-2016}.

\subsection{Nonexistence of positive steady state}\label{ss1}
Denote by $(\mathbf u,\mathbf v)$ the positive steady state of \eqref{m1} if it exists, where
\begin{equation*}
	\mathbf u=(u_1(x),\cdots,u_n(x))\gg\mathbf0\;\;\text{and}\;\;\mathbf v=(v_1(x),\cdots,v_n(x))\gg\mathbf0.
\end{equation*}
Then $(\mathbf u,\mathbf v)$  is a positive solution of
\begin{subequations}\label{m4}
	\begin{numcases}
		\ds d_i (u_{i})_{xx}+r_i u_i\left(1-\dfrac{u_i+v_i}{k_i}\right)=0, \;\;x\in \Omega_i,\;\; i=1,2,\cdots,n, \label{m4-a}\\
		\ds \hat{d}_i (v_{i})_{xx}+r_i v_i\left(1-\dfrac{u_i+v_i}{k_i}\right)=0,\;\;\; x\in \Omega_i,\;\; i=1,2,\cdots,n,\label{m4-b}\\
		u_{i+1}(x_i^{+})=p_i u_i(x_i^{-}),\;d_{i+1}( u_{i+1})_{x}(x_i^{+})=d_i ( u_{i})_{x}(x_i^{-}),\;\;\label{m4-c}\\
		v_{i+1}(x_i^{+})=\hat{p}_iv_i(x_i^{-}),\;\hat{d}_{i+1}( v_{i+1})_{x}(x_i^{+})=\hat{d}_i ( v_{i})_{x}(x_i^{-}),\label{m4-d}\\
		( u_1)_x(0^{+})=(u_n)_x(L^{-})=( v_1)_x(0^{+})=( v_n)_x(L^{-})=0.\label{m4-e}
	\end{numcases}
\end{subequations}

In the following, we obtain some \textit{a priori} estimates on the positive steady state $(\mathbf u,\mathbf v)$. The first one is for the case $n=2$.
\begin{lemma}\label{lemma3.1}
	Suppose that $n=2$. Let $\mathbf u,\mathbf v\in C^2(\overline{\Omega}_1)\times C^2(\overline{\Omega}_2)$ be the positive solution of \eqref{m4} and let $\mathbf{\overline k}$ be defined in \eqref{ok}. Then the following statements hold:
	\begin{enumerate}
		\item [{$\rm (i)$}]  If $\mathbf p,\mathbf{\hat p}\gg\mathbf{ \overline{k}}$, then
		$(u_{1})_{x}(x_1^{-}),(u_{2})_{x}(x_1^{+}), (v_{1})_{x}(x_1^{-}),(v_{2})_{x}(x_1^{+})<0$;
		\item [{$\rm (ii)$}]  If $\mathbf{ \overline{k}}\gg \mathbf p,\mathbf{\hat p} $, then $(u_{1})_{x}(x_1^{-}),(u_{2})_{x}(x_1^{+}), (v_{1})_{x}(x_1^{-}),(v_{2})_{x}(x_1^{+})>0$.
	\end{enumerate}
\end{lemma}
\begin{proof}
	We only prove (i), since (ii) can be proved similarly.
	Using arguments similar to those in the proof of \eqref{uv4-1}, we obtain that
	\begin{equation}\label{uv2}
		\begin{aligned}
			0=&\dfrac{d_1(k_1p_1-k_{2})(u_{1})_{x}(x_1^-)}{p_1u_1(x_1^-)}+
			\sum_{i=1}^{2}\int_{x_{i-1}}^{x_i}\left[\dfrac{d_ik_i(u_{i})_{x}^2}{u_i^2}+\dfrac{r_i(k_i-u_i-v_i)^2}{k_i}  \right]{\rm d}x.
		\end{aligned}
	\end{equation}
	Suppose to the contrary that $(u_{1})_{x}(x_1^{-})\ge0$. This combined with $\mathbf p\gg\mathbf{\overline{k}}$ implies that $u_i(x)+v_i(x)= k_i$ for $x \in\overline\Omega_i$ with $i=1,2$. Then, by \eqref{m4-c}-\eqref{m4-d},
	\begin{equation*} k_2=u_2(x_1^+)+v_2(x_1^+)=p_1u_1(x_1^-)+\hat{p}_1v_1(x_1^-)>k_2,
	\end{equation*}
	where we have used $\mathbf p,\mathbf {\hat p}\gg\mathbf{\overline{k}}$ in the last step. This leads to a contradiction, and consequently, $(u_{1})_{x}(x_1^-)< 0$. By \eqref{m4-c} again, we have $(u_{2})_{x}(x_1^+)< 0$. Similarly, we can also obtain that $(v_{1})_{x}(x_1^-),(v_{2})_{x}(x_1^+)< 0$. This completes the proof.
\end{proof}

The second one gives some identities for later use.
\begin{lemma}\label{lemma3.2}
		Let $\mathbf u,\mathbf v\in C^2(\overline{\Omega}_1)\times\cdots\times C^2(\overline{\Omega}_n)$ be the positive solution of \eqref{m4}.
	\begin{enumerate}
		\item [$\rm (i)$] Suppose that $x_{i-1}\le a<b\le x_{i}$ with $i=1,\cdots, n$. Then
		\begin{equation}\label{3m}
			\begin{split}
				\int_{a}^{b} r_i\left(k_i-u_i-v_i\right) {\rm d}x&= -\dfrac{d_ik_i(u_{i})_{x}(b^-)}{u_i(b^-)}+\dfrac{d_ik_i(u_{i})_{x}(a^+)}{u_i(a^+)}-d_ik_i\int_{a}^{b}\dfrac{(u_{i})_{x}^2}{u_i^2}{\rm d}x\\
				&=-\dfrac{\hat{d}_ik_i(v_{i})_{x}(b^-)}{v_i(b^-)}+\dfrac{\hat{d}_ik_i(v_{i})_{x}(a^+)}{v_i(a^+)} -\hat{d}_ik_i\int_{a}^{b}\dfrac{(v_{i})_{x}^2}{v_i^2}{\rm d}x.
			\end{split} 		
		\end{equation}
		Moreover, if $[d_i(u_{i})_{x}+\hat d_i(v_{i})_{x}](a^+)\ge (\text{resp.} >) [d_i(u_{i})_{x}+\hat d_i(v_{i})_{x}](b^-)$, then \begin{equation}\label{intrk}
			\int_{a}^{b} r_i\left(k_i-u_i-v_i\right) {\rm d}x\ge(\text{resp.} >)  0.
		\end{equation}	
		\item [$\rm (ii)$] Suppose that $x_{\ell_*-1}\le a< x_{\ell_*}$ and $x_{\ell^*-1}< b \le x_{\ell^*}$ for some $\ell_*$ and $\ell^*$ with
		$1\le \ell_*< \ell^*\le n$. Then
		\begin{equation}\label{3m5}
			\begin{aligned}
				0=&\sum_{i=\ell_*}^{\ell^*-1}(\prod_{\ell=i+1}^{\ell^*}p_\ell)d_{i}(\hat{p}_i-p_i)(u_{i})_{x}(x_i^-)v_i(x_i^-)+
				\hat {d}_{\ell^*}p_{\ell^*}(v_{\ell^*})_x(b^-)u_{\ell^*}(b^-)\\ &-\hat {d}_{\ell_*}(\prod_{\ell=\ell_*}^{\ell^*}p_\ell)(v_{\ell_*})_x(a^+)u_{\ell_*}(a^+) +{d}_{\ell_*}(\prod_{\ell=\ell_*}^{\ell^*}p_\ell)(u_{\ell_*})_x(a^+)v_{_{\ell_*}}(a^+)\\
				&-{d}_{\ell^*}p_{\ell^*}(u_{\ell^*})_x(b^-)v_{_{\ell^*}}(b^-) +\sum_{i=\ell_*}^{\ell^*}(\prod_{\ell=i}^{\ell^*}p_\ell)(d_i-\hat d_i)\int_{\Omega_i\cap(a,b)}(v_{i})_{x}(u_{i})_{x}{\rm d}x.
			\end{aligned}
		\end{equation}
		Here we define $p_n=1$.
	\end{enumerate}
\end{lemma}
\begin{proof}
	Using arguments similar to those in the proof of \eqref{id-1}, we can obtain that (ii) holds.
	Dividing \eqref{m4-a}  and \eqref{m4-b} by $u_i/k_i$ and  $v_i/k_i$, respectively, and integrating by parts, we obtain that \eqref{3m} holds.
	The proof for the remaining part of  (i) can be found in \cite[Lemma 4.3]{45-Ge-2024}
	(see also \cite[Theorem 5.1]{36-Maciel-2020}).
\end{proof}
The third one is on the signs of $(u_i)_x$ and $(v_i)_x$. 	
\begin{lemma}\label{noccur}
	 Let $\mathbf u,\mathbf v\in C^2(\overline{\Omega}_1)\times\cdots\times C^2(\overline{\Omega}_n)$ be the positive solution of \eqref{m4} and suppose that $x_{i-1}\le a<b\le x_{i}$ with $i=1,\cdots, n$. Then the following two cases cannot occur:
	\begin{enumerate}
		\item [{\rm (i)}] $(u_i)_x,(v_i)_x\ge0$  with $(u_i)_x\not\equiv0$ or $(v_i)_x\not\equiv0$ in $(a,b)$, and $$\max\{(u_i)_{xx}( a^+),(v_i)_{xx}( a^+)\}\ge 0, \;\;\min\{(u_i)_{xx}(b^-),(v_i)_{xx}(b^-)\}\le 0;$$
		\item [{\rm (ii)}] $(u_i)_x,(v_i)_x\le0$ with $(u_i)_x\not\equiv0$ or $(v_i)_x\not\equiv0$ in $(a,b)$, and $$\min\{(u_i)_{xx}( a^+),(v_i)_{xx}( a^+)\}\le 0,\;\;\max\{(u_i)_{xx}(b^-),(v_i)_{xx}(b^-)\}\ge 0.$$
	\end{enumerate}
\end{lemma}
\begin{proof}
	We only show that (i) cannot occur, and (ii) can be treated similarly. Suppose to the contrary that (i) occurs.
	Since
	$$\max\{(u_i)_{xx}( a^+),(v_i)_{xx}(a^+)\}\ge 0 \;\;\text{and}\;\;
	\min\{(u_i)_{xx}(b^-),(v_i)_{xx}(b^-)\}\le 0,$$ it follows from \eqref{m4-a} and \eqref{m4-b} that
	\begin{equation*}
		u_i(a^+)+v_i(a^+)\ge k_i\;\;\text{and}\;\;u_i(b^-)+v_i(b^-)\le k_i,
	\end{equation*}
	which contradicts the fact that $(u_i)_x,(v_i)_x\ge0$ with $(u_i)_x\not\equiv0$ or $(v_i)_x\not\equiv0$ in $(a, b)$. Therefore, (i) can not occur.
\end{proof}

The last one is as follows.
\begin{lemma}\label{add3.6}
	 Let $\mathbf u,\mathbf v\in C^2(\overline{\Omega}_1)\times\cdots\times C^2(\overline{\Omega}_n)$ be the positive solution of \eqref{m4}. If $(u_1)_x(x_1^-),(v_1)_x(x_1^-)\le0$, then $(u_{1})_{x}, (v_{1})_{x}\le 0$  in $\Omega_1$.
\end{lemma}
\begin{proof}
	Since $(u_1)_x(x_1^-),(v_1)_x(x_1^-)\le0$, we divide the proof into the following two cases:
	\begin{equation*}
		{\rm(a_1)}\;(u_1)_x(x_1^-)(v_1)_x(x_1^-)=0;\;\;{\rm(a_2)}\;(u_1)_x(x_1^-),\;(v_1)_x(x_1^-)<0.
	\end{equation*}
	
	For case (a$_1$), since $0=[d_1(u_1)_x+\hat d_1 (v_1)_x](0^+)\ge[d_1(u_1)_x+\hat d_1 (v_1)_x](x_1^-)$, it follows from Lemma \ref{lemma3.2} (i) that
	\begin{equation}\label{riine}
		\int_{0}^{x_1} r_1\left(k_1-u_1-v_1\right)\ge0.
	\end{equation}
	In view of  $(u_1)_x(x_1^-)(v_1)_x(x_1^-)=0$, we see from \eqref{3m} and \eqref{riine} that
	$(u_1)_x=0$ in $\Omega_1$ or $(v_1)_x=0$ in $\Omega_1$. This combined with \eqref{m4-a}-\eqref{m4-b} implies that
	$(u_1)_x,(v_1)_x=0$ in $\Omega_1$, and the desired result holds.
	
	For case (a$_2$), since $(u_1)_x(0^+)=(v_1)_x(0^+)=0$, it follows  that $$y_1:=\text{sup}\{x\in [0,x_1)|(u_{1})_{x}(x^+)=0\}\;\text{and}\; y_2:=\text{sup}\{x\in [0,x_1)|(v_{1})_{x}(x^+)=0\}$$
	are well-defined with $0\le y_1, y_2<x_1$. Define $y_0:=\max\{y_1,y_2\}$. Clearly, the desired result holds for $y_0=0$. If
	$y_0>0$, $(u_1)_x(y_0^-),(v_1)_x(y_0^-)\le0$ with $(u_1)_x(y_0^-)(v_1)_x(y_0^-)=0$. Then using similar arguments as in the proof of case (a$_1$), we see that $(u_1)_x,(v_1)_x=0$ in $[0,y_0]$, and the desired result holds.
\end{proof}

Now we show the nonexistence of positive steady state for model \eqref{m1}.
\begin{theorem}\label{theorem3.6}
	Let $\mathbf{\overline k}$ be defined in \eqref{ok}. Then model \eqref{m1} has no positive steady state if one of the following conditions holds:
	\begin{enumerate}
		\item [{\rm (i)}] $\mathbf p \gg\mathbf{\overline{k}}$ and $(\mathbf {\hat p},\mathbf {\hat d})\in \mathcal L_1^*\cup \mathcal L_2$, where $\mathcal L_2$ is defined in \eqref{l1l2}, and
		\begin{equation}\label{l1l2*}
			\mathcal L_1^*:=\{(\mathbf{\hat{p}},\mathbf{\hat{d}}):\mathbf p\gg\mathbf {\hat p}\gg  \mathbf{ \overline{k}},\;\mathbf d\ge\mathbf{\hat d}\};
		\end{equation}
		\item [{\rm (ii)}] $\mathbf{\overline{k}}\gg\mathbf p $ and $(\mathbf {\hat p},\mathbf {\hat d})\in \mathcal S_1^*\cup \mathcal S_2$, where $\mathcal S_2$ is defined in \eqref{s1s2}, and	
		\begin{equation}\label{s1s2*}
			\mathcal S_1^*:=\{(\mathbf{\hat{p}},\mathbf{\hat{d}}):\mathbf{ \overline{k}}\gg\mathbf {\hat p}\gg\mathbf p,\;\mathbf d\ge\mathbf{\hat d}\}.
		\end{equation}
	\end{enumerate}
\end{theorem}
\begin{proof}
	We only prove (i), and (ii) can be handled similarly.
	Since  the nonlinear terms of model \eqref{m1} are symmetric, it suffices to
	consider the case $\mathbf p \gg\mathbf{\overline{k}}$ and $(\mathbf {\hat p},\mathbf {\hat d})\in \mathcal L_1^*$.
	The rest of the proof is based on the following three claims.
	
	\emph{Claim 1:}	 If $(u_{i})_{x},(v_{i})_{x}\le 0$ in $\Omega_i$  and  $(u_{i})_{x}(x_{i}^-),(v_{i})_{x}(x_i^-)\le 0$ for  $i=1,\cdots,n-1$, then $(u_{n})_{x},(v_{n})_{x}\le 0$ in $\Omega_n$.
	
	\noindent\emph{Proof of Claim:} Since $(u_n)_x(L^-)=(v_n)_x(L^-)=0$, it follows that $z_1,z_2$ are well-defined with $x_{n-1}\le z_1,z_2\le L$, where
	\begin{equation*}
		\begin{split}
			&z_1:=\text{inf}\{x\in[x_{n-1},x_n]:\exists\; n-1\le i\le n \;\text{s.t.}\;x\in(x_{i-1},x_{i}]\;\text{and}\; (u_{i})_{x}(y^-)=0\},\\
			&z_2:=\text{inf}\{x\in[x_{n-1},x_n]:\exists \;n-1\le i\le n \;\text{s.t.}\;x\in(x_{i-1},x_{i}]\;\text{and}\; (v_{i})_{x}(y^-)=0\},
		\end{split}
	\end{equation*}
	and for fixed $i=1,2$, $z_i\in (x_{\ell_i-1},x_{\ell_i}]$, where
	$\ell_i=n$ if $z_i>x_{n-1}$ and $\ell_i=n-1$ if $z_i=x_{n-1}$.
	
	We first claim that $z_1\le z_2$. Suppose not, then $z_1> z_2$. Thus, we have
	\begin{equation}\label{ncu2}
		(u_{\ell_2})_x(z_2^-)<0\;\;\text{and}\;\;(u_{i})_{x}, (v_i)_x\le 0\;\;\text{in}\;\;\Omega_i\cap(0,z_2)\;\;\text{for}\;\;i=1,\cdots,\ell_2.
	\end{equation}
	Substituting  $a=0$, $b=z_2$, $\ell_*=1$ and $\ell^*=\ell_2$ into \eqref{3m5} yields
	\begin{equation}\label{0z2}
		\begin{aligned}
			0=&\sum_{i=1}^{\ell_2-1}(\prod_{\ell=i+1}^{\ell_2}p_\ell)d_{i}(\hat{p}_i-p_i)(u_{i})_{x}(x_i^-)v_i(x_i^-) +\hat d_{\ell_2}p_{\ell_2}u_{\ell_2}(z_2^-)(v_{\ell_2})_{x}(z_2^-)\\
			&+\sum_{i=1}^{\ell_2}(\prod_{\ell=i}^{\ell_2}p_\ell)(d_i-\hat d_i)\int_{\Omega_i\cap(0,z_2)}(v_{i})_{x}(u_{i})_{x}{\rm d}x-d_{\ell_2}p_{\ell_2}(u_{\ell_2})_{x}(z_2^-)v_{\ell_2}(z_2^-)>0,
		\end{aligned}
	\end{equation}
	where we have used \eqref{ncu2}, $(\mathbf {\hat p},\mathbf {\hat d})\in \mathcal L_1^*$ and the conditions of Claim 1 in the last step.  This leads to a contradiction, and  the claim $z_1\le z_2$ holds.
	
	Then we claim $(u_{n})_x(z_2^-)\le 0$ if $z_{2}\in(x_{n-1},L]$. Clearly, $(u_{n})_x(L^-)=0$, and we only need to consider the case $z_2\in(x_{n-1}, L)$. To the contrary, we have
	\begin{equation*}
		(u_n)_x(z_2)=(u_{n})_x(z_2^-)>0=(u_n)_x(L^-),\;\;(v_n)_x(z_2)=(v_{n})_x(z_2^-)=(v_n)_x(L^-)=0.
	\end{equation*}
	This together with Lemma \ref{lemma3.2} (i) yields
	\begin{equation*}
		0<\int_{z_2}^{L} r_n\left(k_n-u_n-v_n\right) {\rm d}x=-\hat{d}_nk_n\int_{z_2}^{L}\dfrac{(v_{n})_{x}^2}{v_n^2}{\rm d}x\le 0,
	\end{equation*}
	which is a contradiction. Therefore,  the claim holds, i.e., $(u_{n})_x(z_2^-)\le 0$ if $z_{2}\in(x_{n-1},L]$.
	
	Now we claim that $z_1=z_2$. Suppose not, then $x_{n-1}\le z_1<z_2\le L$. Thanks to $(u_n)_x(z_2^-)\le 0$, we will obtain a contradiction for each of the
	following two cases:
	\begin{equation*}
		{\rm(a_1)}\;z_1<z_2,\; (u_n)_x(z_2^-)<0;\;\;{\rm(a_2)}\;z_1<z_2,\; (u_n)_x(z_2^-)=0.
	\end{equation*}	
	For case (a$_1$), noticing that $d_n(u_n)_x(z_1^+)=d_{\ell_1}(u_{\ell _1})_x(z_1^-)=0$, where $\ell_1=n$ if $z_1>x_{n-1}$ and $\ell_1=n-1$ if $z_1=x_{n-1}$, we see that
	$z_3=\text{sup}\{x\in [z_1,z_2):(u_{n})_{x}(x^+)=0\}$ is well-defined with $z_1\le z_3<z_2$. Thus,
	$(u_n)_{xx}(z_3^+)\le 0$, $(u_n)_{x}(z_3^+)=0 $ and $(u_n)_{x} <0$ in $(z_3,z_2)$.
	This combined the definition of $z_2$ implies that Lemma \ref{noccur} (ii) occurs with $a=z_3$ and $b=z_2$, which is a contradiction.

	For case (a$_2$), $(u_n)_x(z_2^-)=0$. Set $\mathcal M:=\{x\in (z_1,z_2):(u_n)_x(x)\ge0\}$. If $\mathcal M\subsetneqq (z_1,z_2)$, then there exists $z_4\in (z_1,z_2)$ such that $(u_n)_x(z_4)<0$. Noting that $(u_n)_x(z_1^+)=(u_n)_x(z_2^-)=0$, we have
	$$z_5:=\text{sup}\{x\in [z_1,z_4)|(u_{n})_{x}(x^+)=0\}\;\;\text{and}\;\;
	z_6:=\text{inf}\{x\in (z_4,z_2]|(u_{n})_{x}(x^-)=0\}$$
	are well-defined with $z_1\le z_5<z_4$ and $z_4<z_6\le z_2$. Thus,
	$(u_n)_{xx}(z^+_5)\le 0$, $(u_n)_{xx}(z_6^-)\ge 0$ and $(u_n)_{x} <0$ in $(z_5,z_6)$.
	This together with \eqref{ncu2} implies Lemma \ref{noccur} (ii) occurs with $a=z_5$ and $b=z_6$, which is a contradiction.
	
	If $\mathcal M= (z_1,z_2)$, we have $(u_n)_x\ge0$ in $(z_1,z_2)$.
	Then one of the following two subcases occur:
	\begin{equation*}
		\begin{aligned}
			&{\rm(b_1)}\; \text{there exists}\;z_*\in(z_1,z_2)\;\text{such that}\;\;d_n(u_n)_x(z_*)+\hat{d}_n(v_n)_x(z_*)\ge0;\\
			&{\rm(b_2)}\;d_n(u_n)_x+\hat{d}_n(v_n)_x<0\;\;\text{in}\;\;(z_1,z_2).
		\end{aligned}
	\end{equation*}
	For subcase ${\rm(b_1)}$, noticing that $(u_n)_x(z^-_2)=(v_{n})_x(z_2^-)=0$, we deduce from Lemma \ref{lemma3.2} (i) that
	\begin{equation}\label{intrk-1}
		\int_{z_*}^{z_2} r_n\left(k_n-u_n-v_n\right) {\rm d}x\ge 0.
	\end{equation}
	In addition, since $z_2>x_{n-1}$, we see from the definition of $z_2$ that
	\begin{equation}\label{ncu2-2}
		(v_n)_x<0\;\;\text{in}\;\;[z_*,z_2)\;\;\text{and}\;\;(v_n)_x(z_2^-)=0.
	\end{equation}
	Then it follows  from \eqref{3m} that
	\begin{equation*}
		\begin{split}
			\int_{z_*}^{z_2} r_n\left(k_n-u_n-v_n\right) {\rm d}x
			=&-\dfrac{\hat{d}_nk_n(v_{n})_{x}(z_2^-)}{v_n(z_2^-)}+\dfrac{\hat{d}_nk_n(v_{n})_{x}(z_*)}{v_2(z_*)} \\ &-\hat{d}_nk_n\int_{z_*}^{z_2}\dfrac{(v_{n})_{x}^2}{v_n^2}{\rm d}x<0,
		\end{split}
	\end{equation*}
	where we have used \eqref{ncu2-2} in the last step.
	This contradicts \eqref{intrk-1}.
	
	For subcase ${\rm(b_2)}$, we have
	\begin{equation}\label{z1z2}
		\hat{d}_n\left[(u_n)_x+(v_n)_x\right]\le d_n(u_n)_x+\hat{d}_n(v_n)_x<0 \;\;\text{in} \;\; (z_1,z_2),
	\end{equation}
	where we have used $\hat{d}_n\le d_n$ and $\mathcal M= (z_1,z_2)$ in first step. By the definition of $z_2$, $(v_{n})_{xx}(z_2^-)\ge0$, which together with \eqref{m4-b} yields $u_n(z_2^-)+v_n(z_2^-)\ge k_n$. Then it follows from \eqref{z1z2} that $u_n(x)+v_n(x)> k_n$ for $x\in (z_1,z_2)$, and consequently, we see from  \eqref{m4-a} that $(u_{n})_{xx}>0$ in $(z_1,z_2)$, which contradicts the fact that $(u_n)_x(z_1^+)=(u_n)_x(z_2^-)=0$.
	Therefore, we obtain a contradiction for each of the cases (a$_1$)-(a$_2$), and the claim $z_1=z_2$ holds.
	
	If $z_1=z_2=L$, then by the definitions of $z_1$ and $z_2$, $(u_{n})_{x}(x),(v_{n})_{x}(x)\le 0$ for $x\in\Omega_n$. Thus, Claim 1 holds for this case. If $z_1=z_2<L$, we have $$(u_n)_x(z_2^+)=(u_n)_x(L^-)=(v_n)_x(z^+_2)=(v_n)_x(L^-)=0.$$
	Then, by Lemma \ref{lemma3.2} (i) again,
	\begin{equation*}
		0\le \int_{z_2}^{L} r_n\left(k_n-u_n-v_n\right) {\rm d}x=-{d_2}k_n\int_{z_2}^{L}\dfrac{(u_{n})_{x}^2}{u_n^2}{\rm d}x=-\hat{d}_2k_n\int_{z_2}^{L}\dfrac{(v_{n})_{x}^2}{v_n^2}{\rm d}x\le 0,
	\end{equation*}
	which yields $(u_n)_x,(v_n)_x=0$ in $(z_2,L)$. This combined with the definitions of $z_1$ and $z_2$ implies that Claim 1 also holds for this case.
	
	\emph{Claim 2:}	We claim that model \eqref{m1} has no positive steady state for the case $n=2$.
	
	\noindent\emph{Proof of Claim:} Suppose to the contrary that model (\ref{m1}) has a positive steady state $(\mathbf u,\mathbf v)$ for the case $n=2$.
	It follows from Lemmas \ref{lemma3.1} and \ref{add3.6} and Claim 1 that
	\begin{equation}\label{uvle0}
		(u_1)_x(x_1^-),(v_1)_x(x_1^-)<0\;\;\text{and}\;\;(u_i)_x,(v_i)_x\le0\;\;\text{in}\;\;\Omega_i\;\;\text{for each}\;\;i=1,2.
	\end{equation}
	Plugging  $a=0$ and $b=L$ into \eqref{3m5} yields
	\begin{equation}\label{0z2-main}
		\begin{aligned}
			0=&d_{1}(\hat{p}_1-p_1)(u_{1})_{x}(x_1^-)v_1(x_1^-) +\sum_{i=1}^{2}(\prod_{\ell=i}^{2}p_\ell)(d_i-\hat{d}_i)\int_{\Omega_i}(v_{i})_{x}(u_{i})_{x}{\rm d}x>0,
		\end{aligned}
	\end{equation}
	where we have used \eqref{uvle0}, $ \mathbf p\gg\mathbf{\hat p}$ and $\mathbf d\ge\mathbf{\hat d}$ in the last step. This leads to a contradiction and completes the proof for Claim 2.
	
	\emph{Claim 3:}	We claim that model \eqref{m1} has no positive steady state for the case $n=m$ with $m\ge3$ if it has no positive steady state for $n=2,\cdots,m-1$.
	
	\noindent\emph{Proof of Claim:} Suppose to the contrary that model (\ref{m1}) has a positive steady state $(\mathbf u,\mathbf v)$ for the case $n=m$. To obtain a contradiction, we first show that the following two cases cannot occur:
	\begin{enumerate}
		\item [(c$_1$)] there exists $2\le\ell_0\le m-2$ such that $(u_{\ell_0})_{x}(x_{\ell_0}^-)(v_{\ell_0})_{x}(x_{\ell_0}^-)\le 0$ if $m\ge4$;
		\item [(c$_2$)] there exists $2\le\ell_0\le m-1$ and $x^*\in\Omega_{\ell_0}$ such that $(u_{\ell_0})_{x}(x^*)(v_{\ell_0})_{x}(x^*)\le0$.
	\end{enumerate}
	Suppose to the contrary that (c$_1$) occurs. Without loss of generality, we assume that
	\begin{equation}\label{uvell0}
		(u_{\ell_0})_{x}(x_{\ell_0}^-)\le 0\;\;\text{and}\;\; (v_{\ell_0})_{x}(x_{\ell_0}^-)\ge 0.
	\end{equation}
	Construct the following auxiliary system:
	\begin{equation}\label{uvt}
		\begin{cases}
			(w_i)_t=d_i (w_{i})_{xx}+r_i w_i\left(1-\ds\frac{w_i+\xi_i}{k_i}\right), \;\;t>0,\;\; x\in \Omega_{i},\;\; i=1,\cdots,\ell_0, \\
			(\xi_i)_t=\hat{d}_i (\xi_{i})_{xx}+r_i \xi_i\left(1-\ds\frac{w_i+\xi_i}{k_i}\right), \;\;t>0,\;\; x\in \Omega_{i},\;\; i=1,\cdots,\ell_0, \\
			w_{i+1}(x_i^{+},t)=p_i w_i(x_i^{-},t),\;d_{i+1}( w_{i+1})_{x}(x_i^{+},t)=d_i ( w_{i})_{x}(x_i^{-},t),\;\;t>0,\;\;\\
			\xi_{i+1}(x_i^{+},t)=\hat p_i\xi_i(x_i^{-},t),\;\hat d_{i+1}( \xi_{i+1})_{x}(x_i^{+},t)=\hat d_i ( \xi_{i})_{x}(x_i^{-},t),\;\;t>0,\\
			(w_{1})_{x}(0^{+},t)=(w_{\ell_0})_{x}(x_{\ell_0}^{-},t)=
			(\xi_{1})_{x}(0^{+},t)=(\xi_{\ell_0})_{x}(x_{\ell_0}^{-},t)=0,\;\;t>0.
		\end{cases}
	\end{equation}
	Then system \eqref{uvt} generates a monotone dynamical system, which is order-preserving with respect to the order ``$\succeq$" defined in \eqref{order} (see the Appendix \ref{appA}).
	Define $$([\mathbf u]_{\ell_0},[\mathbf v]_{\ell_0})=(u_{1}(x),\cdots,u_{\ell_0}(x),v_{1}(x),\cdots,v_{\ell_0}(x)).$$
	It follows from \eqref{uvell0} and Proposition \ref{supersub} that
	$([\mathbf u]_{\ell_0},[\mathbf v]_{\ell_0})$ is  a sub-equilibrium of \eqref{uvt}.
	
	Let $(\mathbf w^*,\mathbf 0)$ be the semi-trivial steady state of \eqref{uvt}, where
	$$\mathbf w^*=( w^*_1(x),\cdots, w^*_{\ell_0}(x))$$
	with $w^*_i(x)>0$ for $x\in\overline\Omega_i$ and $i=1,\cdots,\ell_0$.
	It follows from Lemma \ref{lemma2.5} (i) that $\tilde \la_1>0$, where $\tilde \la_1$ is the  principal eigenvalue of the following eigenvalue problem:
	\begin{equation*}
		\begin{cases}
			\hat{d}_i (\psi_{i})_{xx}+r_i \psi_i\left(1-\dfrac{w^*_i}{k_i}\right)=\tilde\lambda_1 \psi_i,\;\; x\in {\Omega}_{i},\;\; i=1,\cdots,\ell_0,\\
			\psi_{i+1}(x_i^{+})=\hat{p}_i\psi_i(x_i^{-}),\;\;\hat{d}_{i+1}(\psi_{i+1})_{x}(x_i^{+})=\hat d_i (\psi_{i})_{x}(x_i^{-}),\\
			(\psi_{1})_{x}(0^{+})=(\psi_{\ell_0})_{x}(x_{\ell_0}^-)=0.
		\end{cases}
	\end{equation*}
	Let $\mathbf{\psi}=(\psi_{1}(x),\cdots,\psi_{\ell_0}(x))$ be the eigenfunction corresponding to $\tilde \la_1$ with $\psi_i(x)>0$ for $x\in\overline\Omega_i$ and $i=1,\cdots,\ell_0$. Choose $\epsilon_1,\epsilon_2>0$ so that
	\begin{equation}\label{epsilon-n}
		\min_{x\in\overline\Omega_i}v_i(x)>\epsilon_2\max_{x\in\overline\Omega_i}\psi_i(x)\;\;\text{and}\;\; \tilde\lambda_1>\dfrac{r_i}{k_i}\left(\epsilon_1\max_{x\in\overline\Omega_i}w_i^*+\epsilon_2\max_{x\in\overline\Omega_i}\psi_i\right)\;\; \text{for}\;\;i=1,\cdots,\ell_0.
	\end{equation}
	A direct computation yields, for $i=1,\cdots,\ell_0$,
	\begin{equation*}
		\begin{split}
			&(1+\epsilon_1)d_i (w^*_{i})_{xx}+(1+\epsilon_1)r_i w^*_i\left(1-\dfrac{w^*_i+\epsilon_1 w_i^*+\epsilon_2\psi_i}{k_i}\right)\\
			=&-\dfrac{(1+\epsilon_1)r_iw_i^*(\epsilon_1w_i^*+\epsilon_2\psi_i)}{k_i}<0,
		\end{split}
	\end{equation*}
	and
	\begin{equation*}
		\epsilon_2\hat{d}_i (\psi_{i})_{xx}+\epsilon_2r_i \psi_i\left(1-\dfrac{(1+\epsilon_1)w^*_i+\epsilon_2\psi_i}{k_i}\right)=\epsilon_2\psi_i\left(\tilde\lambda_1-\dfrac{ r_i(\epsilon_1w_i^*+\epsilon_2\psi_i)}{k_i}\right)>0,
	\end{equation*}
	which implies that $((1+\epsilon_1)\mathbf{w^*}, \epsilon_2\mathbf{\psi})$ is a super-equilibrium of \eqref{uvt}.
	
	By the first inequality of \eqref{uvell0},  $[\mathbf u]_{\ell_0}=(u_1,\cdots,u_{\ell_0})$ is a lower solution of
	\begin{equation}\label{1wphi-n}
		\begin{cases}
			\ds d_i (w_{i})_{xx}+r_i w_i\left(1-\dfrac{w_i}{k_i}\right)=0, \;\;x\in \Omega_{i},\;\; i=1,\cdots,\ell_0, \\
			w_{i+1}(x_i^{+})=p_i w_i(x_i^{-}),\;\;d_{i+1}(w_{i+1})_x(x_i^{+})=d_i(w_i)_x(x_i^{-}),\\
			(w_{1})_{x}(0^{+})=(w_{\ell_0})_{x}(x_{\ell_0}^-)=0.
		\end{cases}
	\end{equation}
	This implies $u_i\le w_i^*$ in $\overline \Omega_i$, and consequently, $u_i< (1+\epsilon_1)w_i^*$ in $\overline \Omega_i$ for $i=1,\cdots,\ell_0$. This together with \eqref{epsilon-n} yields $((1+\epsilon_1)\mathbf{w^*}, \epsilon_2\mathbf{\psi})\succeq(\neq) ([\mathbf u]_{\ell_0},[\mathbf v]_{\ell_0})$.	Then, it follows  from Theorem \ref{pss} that the auxiliary system \eqref{uvt} admits a positive steady state, which contradicts the fact that
	model \eqref{m1} has no positive steady state for $n=2,\cdots,m-1$. Thus,
	(c$_1$) cannot occur, and case (c$_2$) can be handled similarly.
	
	Since  (c$_1$) and (c$_2$) cannot occur, it follows that
	one of the following two cases can occur:
	\begin{enumerate}
		\item [(d$_1$)] $(u_{i})_{x},(v_{i})_{x}> 0$ in $\Omega_i$ for each $2\le i\le m-1$ and $(u_{i})_{x}(x_{i}^-),(v_{i})_{x}(x_i^-)> 0$ for each $i=2,\cdots,m-2$;
		\item [(d$_2$)] $(u_{i})_{x},(v_{i})_{x}< 0$ in $\Omega_i$ for each $2\le i\le m-1$ and $(u_{i})_{x}(x_{i}^-),(v_{i})_{x}(x_i^-)< 0$ for each $i=2,\cdots,m-2$.
	\end{enumerate}
	Now we obtain a contradiction for each of cases (d$_1$) and  (d$_2$).	
	
	For case (d$_1$), we see from \eqref{m4-e} that $y_*,y^*,z_*,z^*$ are well-defined with $0\le y_*,z_*\le x_1\le x_{m-1}\le y^*,z^*\le L$, where
	\begin{equation*}
		\begin{aligned}
			&y_*=\sup \{y:\exists\;1\le \ell\le 2 \;\text{s.t.}\;x\in[x_{\ell-1},x_{\ell})\;\text{and}\; (u_{\ell})_{x}(y^+)=0\},\\
			&y^*=\inf  \{y:\exists\;  m-1\le \ell\le m \;\text{s.t.}\;x\in(x_{\ell-1},x_{\ell}]\;\text{and}\; (u_{\ell})_{x}(y^-)=0\},\\
			&z_*=\sup \{z:\exists \;1\le \ell\le 2 \;\text{s.t.}\;x\in[x_{\ell-1},x_{\ell})\;\text{and}\; (v_{\ell})_{x}(z^+)=0\},\\
			&z^*=\inf \{z:\exists \; m-1\le \ell\le m \;\text{s.t.}\;x\in(x_{\ell-1},x_{\ell}]\;\text{and}\; (v_{\ell})_{x}(z^-)=0\}.\\
		\end{aligned}
	\end{equation*}
	Let
	$\underline x=\max\{y_*,z_*\}$ and $\overline x=\min\{y^*,z^*\}$,
	and there exist $\underline\ell$ and $\overline \ell$ with $1\le \underline\ell\le 2$ and $m-1\le \overline\ell\le m$ such that $\underline x\in[x_{\underline \ell-1}, x_{\underline \ell})$ and $\overline x\in(x_{\overline \ell-1}, x_{\overline \ell}]$. Then,
	by \eqref{m4-a}-\eqref{m4-b},
	$u_{\underline\ell}(\underline x^+)+v_{\underline\ell}(\underline x^+)\ge k_{\underline\ell}$ and $u_{\overline \ell }({\overline x}^-)+v_{\overline \ell }({\overline x}^-)\le k_{\overline \ell}$. This combined with  ${\rm(d_1)}$ and $\mathbf p\gg\mathbf {\hat p}\gg\mathbf{\overline{k}}$ implies that
	\begin{equation*}
		\begin{aligned}
			k_{\overline \ell }&\ge u_{\overline \ell }({\overline x }^-)+v_{\overline \ell }({\overline x }^-)>u_{\overline \ell }(x_{\overline \ell-1}^+)+v_{\overline \ell}(x_{\overline \ell-1}^+)>\hat{p}_{\overline \ell-1}\left(u_{\overline \ell-1}(x_{\overline \ell-1}^-)+v_{\overline \ell-1}(x_{\overline \ell-1}^-)\right)\\
			&>\cdots >(\prod_{\underline \ell}^{\overline\ell-1}\hat{p}_i)\left(u_{\underline \ell}(x_{\underline\ell}^-)+v_{\underline \ell}(x_{\underline \ell}^-)\right)>(\prod_{\underline \ell}^{\overline\ell-1}\hat{p}_i)\left(u_{\underline \ell}(\underline x^+)+v_{\underline \ell}(\underline x^+)\right)>k_{\overline \ell},
		\end{aligned}
	\end{equation*}
	which is a contradiction.
	
	For case (d$_2$), we see from Lemma \ref{add3.6} and Claim 1 that
	$(u_{i})_{x},(v_{i})_{x}\le0$ in $\Omega_i$ for each $1\le i\le m$ and $(u_{i})_{x}(x_{i}^-),(v_{i})_{x}(x_i^-)\le 0$ for each $i=1,\cdots,m-1$. Then plugging  $a=0$, $b=L$, $\ell_*=1$ and $\ell^*=m$ into \eqref{3m5} yields
	\begin{equation}\label{0z2-main2}
		\begin{split}
			0=&\sum_{i=1}^{m-1}(\prod_{\ell=i+1}^{m}p_\ell)d_{i}(\hat{p}_i-p_i)(u_{i})_{x}(x_i^-)v_i(x_i^-)\\
			&+\sum_{i=1}^{m}(\prod_{\ell=i}^{m}p_\ell)(d_i-\hat{d}_i)\int_{\Omega_i}(v_{i})_{x}(u_{i})_{x}{\rm d}x\ge0,
		\end{split}
	\end{equation}
	where we have used $\mathbf p\gg\mathbf{\hat p}$ and $\mathbf d\ge\mathbf{\hat d}$ in the last step. Exchanging the positions of $u_i$ and $v_i$ in \eqref{0z2-main2}, we can also obtain that
	\begin{equation}\label{0z2-main3}
		\begin{split} 0=&\sum_{i=1}^{m-1}(\prod_{\ell=i+1}^{m}\hat p_\ell)\hat d_{i}(p_i-\hat{p}_i)(v_{i})_{x}(x_i^-)u_i(x_i^-)\\
			&+\sum_{i=1}^{m}(\prod_{\ell=i}^{m}\hat{ p}_\ell)(\hat{d}_i- d_i)\int_{\Omega_i}(v_{i})_{x}(u_{i})_{x}{\rm d}x\le0,
		\end{split}
	\end{equation}
	where we have used $\mathbf p\gg\mathbf{\hat p}$ and $\mathbf d\ge\mathbf{\hat d}$ in the last step.
	Then it follows from \eqref{0z2-main2} and \eqref{0z2-main3} that
	$(u_{i})_{x}(x_i^-)=(v_{i})_{x}(x_i^-)=0$ for each $i=1,\cdots,m-1$, which implies that $u_i+v_i=k_i$ in $\Omega_i$ for each $i=1,\cdots,m$.
	Then
	\begin{equation}
		\begin{split}
			k_m=&u_m(x_{m-1}^+)+v_m(x_{m-1}^+)>\hat{p}_{m-1}[u_{m-1}(x_{m-1}^-)+ v_{m-1}(x_{m-1}^-)]\\
			>&\cdots>(\prod_{\ell=1}^{m-1}\hat{p}_\ell)[u_{1}(x_{1}^-)+ v_{1}(x_{1}^-)]>k_m,
		\end{split}
	\end{equation}
	where we have used \eqref{m4-d} and $\mathbf p\gg\mathbf{\hat p}$ in the second step and $\mathbf{\hat p}\gg\mathbf{\overline k}$ in the last step. This also leads to a contradiction and completes the proof for Claim 3.
	
	Then the desired result follows from Claims 2-3 and induction.
\end{proof}

\subsection{Main result}\label{ss2}
Now we present our main result on the global dynamics of model \eqref{m1}.
\begin{theorem}\label{theorem3.8}
	Let $\mathbf{\overline k}$ be defined in \eqref{ok}. Then the following statements hold for model \eqref{m1}:
	\begin{enumerate}
		\item [{$\rm (i)$}] If $\mathbf p\gg \mathbf{\overline{k}}$, then the semi-trivial steady state $(\mathbf{0},\mathbf{v^*})$ is globally asymptotically stable for $(\mathbf {\hat p},\mathbf {\hat d})\in \mathcal L_1^*$ and $(\mathbf{u^*},\mathbf{0})$ is globally asymptotically stable for $(\mathbf {\hat p},\mathbf {\hat d})\in \mathcal L_2$, where $\mathcal L_1^*$ and
		$\mathcal L_2$ are defined in \eqref{l1l2*} and \eqref{l1l2}, respectively.
		Moreover, model \eqref{m1} admits at least one stable positive steady state for $(\mathbf {\hat p},\mathbf {\hat d})\in \mathcal L_3$, where $\mathcal L_3$ is defined in \eqref{l3}.
		\item [{$\rm (ii)$}] If $\mathbf{\overline{k}}\gg\mathbf p$, then the semi-trivial steady state $(\mathbf{0},\mathbf{v^*})$ is globally asymptotically stable for $(\mathbf {\hat p},\mathbf {\hat d})\in \mathcal S_1^*$ and $(\mathbf{u^*},\mathbf{0})$ is globally asymptotically stable for $(\mathbf {\hat p},\mathbf {\hat d})\in \mathcal S_2$, where $S_1^*$ and $S_2$ are defined in \eqref{s1s2*} and \eqref{s1s2}, respectively.
		Moreover, model \eqref{m1} admits at least one stable positive steady state for $(\mathbf {\hat p},\mathbf {\hat d})\in \mathcal S_3$, where $\mathcal S_3$ is defined in \eqref{s3}.
	\end{enumerate}
\end{theorem}
\begin{proof}
We only prove (i), since (ii) can be treated similarly.
 We can rewrite model \eqref{m1} as the abstract evolution equation \eqref{abstract} in Appendix \ref{AppendixC}.
Theorem \ref{C4} implies that the semiflow $T_t$ generated by \eqref{abstract} is strongly order-preserving with respect to $\ge_k$ (defined in \eqref{kine}) and order compact on the fractional power spaces $X^{\alpha}_u \times X^{\alpha}_v$ (defined in Remark \ref{rcompact}) for all $t > 0$. Consequently, the theory of monotone dynamical systems \cite{41-Hess-1991,46-Hsu-1996,47-Lam-2016,48-Smith-1995} yields the global dynamics of \eqref{m1}.
Specifically,
Theorem \ref{theorem3.6} implies that model \eqref{m1} has no positive steady state for  $(\mathbf {\hat p},\mathbf {\hat d})\in \mathcal L_1^*\cup\mathcal L_2$. Moreover, Theorem \ref{theorem2.8} shows that $(\mathbf{u^*},\mathbf 0)$ is unstable (resp. stable) for  $(\mathbf {\hat p},\mathbf {\hat d})\in \mathcal L_1^*$ (resp. $(\mathbf {\hat p},\mathbf {\hat d})\in \mathcal L_2$). Since the nonlinear terms of model \eqref{m1} are symmetric, we can also obtain that $(\mathbf 0,\mathbf{v^*})$ is stable (resp. unstable) for  $(\mathbf {\hat p},\mathbf {\hat d})\in \mathcal L_1^*$ (resp.  $(\mathbf {\hat p},\mathbf {\hat d})\in \mathcal L_2$). Then, by the monotone dynamical system theory \cite{41-Hess-1991,46-Hsu-1996,47-Lam-2016,48-Smith-1995}, we obtain $(\mathbf 0, \mathbf{v^*})$ (resp. $(\mathbf{u^*},\mathbf 0)$) is globally asymptotically stable for $(\mathbf {\hat p},\mathbf {\hat d})\in \mathcal L_1^*$ (resp. $(\mathbf {\hat p},\mathbf {\hat d})\in \mathcal L_2$). By Theorem \ref{theorem2.8}, we see that $(\mathbf{u^*},\mathbf 0)$ is unstable for $(\mathbf {\hat p}, \mathbf {\hat d})\in \mathcal L_3$. Note that the nonlinear terms of model \eqref{m1} are symmetric. We can also obtain that $(\mathbf{0},\mathbf {v^*})$ is unstable. Then, by the monotone dynamical system theory again, \eqref{m1} admits at least one stable positive steady state.
\end{proof}

Finally, we summarize the global dynamics of model \eqref{m1} in the following two tables: (see Tables \ref{tablepk-2} and \ref{tablekp-2}).
\begin{table}[htbt]
	\centering
	\caption{Summary of the global dynamics for $\mathbf p\gg \mathbf{\overline{k}}$. }
	\label{tablepk-2}
	\vspace{5pt}
	\begin{tabular}{ccc}
		\toprule[1pt]
		Parameter condition for $\hat{\mathbf p}$ & Parameter condition for $\hat{\mathbf d}$ & Results\\
		\midrule[1pt]
		$\hat{\mathbf p}\gg\mathbf p$ &$\hat{\mathbf d}\ge\mathbf d$ &$(\mathbf{u^*},\mathbf{0})$ g.a.s.\\
		$\mathbf p\gg\hat{\mathbf p}\gg\mathbf{ \overline{k}}$   &$\mathbf d\ge \hat{\mathbf d}$ &$(\mathbf{0},\mathbf{v^*})$ g.a.s. \\
		$\mathbf{ \overline{k}}\gg\mathbf{\hat{p}}$  &none &coexistence\\
		\bottomrule[1pt]
	\end{tabular}
\end{table}
\begin{table}[htbt]
	\centering
	\caption{Summary of the global dynamics for $\mathbf{\overline{k}}\gg\mathbf p$.}
	\label{tablekp-2}
	\vspace{5pt}
	\begin{tabular}{ccc}
		\toprule[1pt]
		Parameter condition for $\hat{\mathbf p}$ & Parameter condition for $\hat{\mathbf d}$ & Results\\
		\midrule[1pt]
		$\mathbf{ \overline{k}}\gg\hat{\mathbf p}\gg\mathbf p$ &$\mathbf d \ge \hat{\mathbf d}$ &$(\mathbf{0},\mathbf{v^*})$ g.a.s.\\
		$\mathbf p\gg\hat{\mathbf p}$   &$\hat{\mathbf d}\ge\mathbf d$ &$(\mathbf{u^*},\mathbf{0})$ g.a.s. \\
		$\mathbf{\hat{p}}\gg\mathbf{ \overline{k}}$  &none &coexistence\\
		\bottomrule[1pt]
	\end{tabular}
\end{table}
\section{Discussion}

We study a reaction-diffusion model for two-species competition in a fragmented habitat of $n$ adjacent patches.  The model incorporates habitat preference of individuals, which induces density discontinuities at interfaces. This model differs from established models, such as coupled ordinary differential equations (patch models) and classical reaction-diffusion models, which typically do not account for such discontinuities.

A central question in the study of two-species competition concerns the mechanisms that drive competitive exclusion or enable coexistence. Within a two-patch reaction-diffusion framework incorporating edge behavior, \cite{36-Maciel-2020} demonstrated that when habitat preference is viewed as a strategy, the ideal free distribution (IFD) strategy is a global evolutionarily stable strategy (ESS).  Our results extend this by proving that the IFD strategy is a global convergent stable strategy (CSS). Specifically, we show that a species with a strategy closer to the IFD can invade and outcompete the resident whose strategy is farther away, provided both strategies lie on the same side of the IFD. Conversely, two species can coexist if their strategies lie on opposite sides of the IFD. Moreover, this result is extended to an $n$-patch reaction-diffusion model incorporating edge behavior.

 These findings  are consistent with results from both coupled ordinary differential equations (patch models) and classical reaction-diffusion models, in which dispersal is typically treated as a strategy. Indeed, as demonstrated by \cite{CantrellIFD,LouIFD} and \cite{AverillIFD,57-Cantrell-2010}, the IFD strategy is also a global ESS in the frameworks of patch models and classical reaction-diffusion models, respectively. Furthermore, competitive outcomes (exclusion or coexistence) are also determined by whether species' strategies reside on the same or opposite sides of the IFD, see \cite{57-Cantrell-2010,Liusapm}.
 Therefore, across all three classes of models, the ideal free distribution (IFD) emerges as a universal principle that governs whether competition results in exclusion or coexistence.

Our analysis relies crucially on the monotonicity of the species steady state for the corresponding single-species model \eqref{m2}. According to Lemma \ref{lemma2.1-n}, this monotonicity in solutions to system \eqref{m2} is guaranteed if either $\mathbf{p} \gg \overline{\mathbf{k}}$ or $\overline{\mathbf{k}} \gg \mathbf{p}$ holds. Biologically, this condition requires that the discontinuities in population density at all interfaces are either uniformly larger or uniformly smaller than the corresponding discontinuities in the carrying capacity \cite{33-Zaker-BMB-2019}. A violation of this condition may lead to non-monotonic steady states, as demonstrated in Proposition \ref{example}. Nevertheless, the non-monotonicity of steady states poses analytical challenges that remain to be explored in future work.

\vspace{10pt}
\begin{appendices}
\titleformat{\section}
{\normalfont\Large\bfseries}{Appendix \thesection}{1em}{}
\section{Semi-trivial steady state}\label{appA}
In this part,  we consider the positive solutions of the following system:
\begin{equation}\label{apm2}
	\begin{cases}
		\ds d_i (u_{i})_{xx}+r_i u_i\left(1-\dfrac{u_i}{k_i}\right)=0, \;\;x\in \Omega_i,\;\; i=1,2,\cdots,n, \\
		u_{i+1}(x_i^{+})=p_iu_i(x_i^{-}),\;\;d_{i+1}(u_{i+1})_{x}(x_i^{+})=d_i (u_{i})_{x}(x_i^{-}),\;\;i=1,\cdots,n-1,\\
		(u_{1})_{x}(0^{+})=(u_{n})_{x}(L^{-})=0,
	\end{cases}
\end{equation}
where $d_i,r_i,k_i$ are positive for each $i=1,\cdots,n$ and have the same meanings as model \eqref{m1}.
Here the solution $\mathbf u=(u_1,\cdots,u_n)$ of \eqref{apm2} is positive means that $\mathbf u\gg\mathbf0$ (i.e., $u_i>0$ in $\overline\Omega_i$ for each $i=1,\cdots,n$).
\begin{theorem}\label{apt1}
Model \eqref{apm2} admits a unique positive solution $\mathbf u^*\in C^2(\overline{\Omega}_1)\times \cdots \times C^2(\overline{\Omega}_n)$.
\end{theorem}
\begin{proof}
	Inspired by \cite{38-Maciel-2015, 33-Zaker-BMB-2019},
	we define
	\begin{equation*}
		\xi=\begin{cases}
			&x,\;\; x\in \Omega_1,\\
			&\xi_{i-1}+\prod_{j=1}^{i-1}p_j(x-x_{i-1}),\;\; x\in \Omega_i,\;\; i=2,\cdots,n,
		\end{cases}
	\end{equation*}
	where
	\begin{equation*}
		\xi_i=\begin{cases}
			&0,\;\; i=0,\\
			&x_1,\;\; i=1,\\
			&x_1+\sum_{\ell=1}^{i-1}\prod_{j=1}^{\ell}p_j(x_{\ell+1}-x_{\ell}), \;\; i=2,\cdots,n.
		\end{cases}
	\end{equation*}
	Clearly, $x\in \Omega_i$ if and only if $\xi\in \widetilde{\Omega}_i$ with $\widetilde{\Omega}_i=(\xi_{i-1},\xi_i)$. Then using the following scaling
	\begin{equation*}
		w_i=\begin{cases}
			&u_1,\;\; i=1,\\
			&\dfrac{u_i}{\prod_{j=1}^{i-1}p_j},\;\; i=2,\cdots,n,
		\end{cases}
	\end{equation*}
	we convert model \eqref{apm2} into the following equivalent form:
	\begin{equation}\label{apm3}
		\begin{cases}
			\ds D_i (w_{i})_{\xi\xi}+r_i w_i\left(1-\dfrac{w_i}{\widetilde{k}_i}\right)=0,  \;\;\xi\in \widetilde \Omega_i,\;\; i=1,2,\cdots,n,\\
			w_{i+1}(\xi_i^{+})=w_i(\xi_i^{-}),\;\;D_{i+1}(w_{i+1})_{\xi}(\xi_i^{+})=D_i (w_{i})_{\xi}(\xi_i^{-}),\;\;i=1,\cdots,n-1,\\
			(w_{1})_{\xi}(\xi_0^{+})=(w_{n})_{\xi}(\xi_n^{-})=0,
		\end{cases}
	\end{equation}
	where
	\begin{equation*}
		D_i=\begin{cases}
			d_1,\;\; i=1,\\
			d_i\left(\prod_{j=1}^{i-1}p_j\right)^2,\;\; i=2,\cdots,n,
		\end{cases}
		\;\;\text{and}\;\;
		\widetilde{k}_i=\begin{cases}
			k_1,\;\; i=1,\\
			\dfrac{k_i}{\prod_{j=1}^{i-1}p_j},\;\; i=2,\cdots,n.
		\end{cases}
	\end{equation*}
	Note that system \eqref{apm3} has continuous density and flux at interface. Then it follows from \cite[Theorem A.15]{42-Jin-2019} and the proof of \cite[Theorem 3.3]{42-Jin-2019} that system \eqref{apm3} admits a unique positive solution
	$\mathbf w^*\in \widetilde{X}$, where
	\begin{equation*}
		\widetilde{X}:=\{\mathbf w=(w_1,\cdots,w_n)\in C([0,\xi_n]):w_i\in C^2([\xi_{i-1},\xi_i])\}.
	\end{equation*}
	Consequently, system \eqref{apm2} admits a unique positive solution $\mathbf u^*\in \mathbf u^*\in C^2(\overline{\Omega}_1)\times \cdots \times C^2(\overline{\Omega}_n)$.		
	This completes the proof.
\end{proof}

We now present some properties of $\mathbf u^*$, which were originally derived in \cite[Theorem 1]{33-Zaker-BMB-2019} for the case $n=2$. Here we generalize these properties to the case where $n$ is finite but arbitrary.
\begin{lemma}\label{lemma2.1-n}
	 Let $\mathbf u\in C^2(\overline{\Omega}_1)\times\cdots\times C^2(\overline{\Omega}_n)$ be the unique positive solution of \eqref{m2}, and let $\mathbf{\overline k}$ be defined in \eqref{ok}. Then the following statements hold:
	\begin{enumerate}
		\item [$\rm{(i)}$] If $\mathbf p\gg\mathbf{\overline k}$, then $u^*_1(0^{+})< k_1$, $u^*_n(L^{-})> k_n$,
		$(u_i^*)_x<0$ in $\Omega_i$ for $i=1,\cdots,n$, and
		\begin{equation}\label{interf1}
			(u^*_{i+1})_x(x_i^{+}),\;(u^*_{i})_x(x_i^{-})<0\;\;\text{for}\;\;i=1,\cdots,n-1;
		\end{equation}
		\item [$\rm {(ii)}$] If $\mathbf{\overline k}\gg\mathbf p$, then $u^*_1(0^{+})> k_1$, $u^*_n(L^{-})< k_n$,
		$(u_i^*)_x>0$ in $\Omega_i$ for $i=1,\cdots,n$, and
		\begin{equation}\label{interf1-2}
			(u^*_{i+1})_x(x_i^{+}),\;(u^*_{i})_x(x_i^{-})>0\;\;\text{for}\;\;i=1,\cdots,n-1.
		\end{equation}
	\end{enumerate}
\end{lemma}
\begin{proof}
	We only prove (i), and part (ii) can be addressed similarly. The proof relies on the following three claims.
	
	\emph{Claim 1:}  We claim that
	$(u^*_{i})_x\le0$ in $\Omega_i$ for $i=1,\cdots,n$.

	\noindent\emph{Proof of Claim:} If the claim is not true, then there exist $1\le \ell_0\le n$ and $y_0\in \Omega_{\ell_0}$ such that
	$(u^*_{\ell_0})_x(y_0)>0.$
	It follows from \eqref{m2c} that $y_1,y_2$  are well-defined with $y_1<y_2$, where
	\begin{equation*}
		\begin{aligned}
			&y_1=\sup \{y:\exists\; 1\le \ell\le \ell_0\;\; \text{s.t.}\;\; y\in [x_{\ell-1},x_\ell) \;\;\text{and} \;\; (u^*_{\ell})_x(y^+)=0\},\\
			&y_2=\inf \{y:\exists\; \ell_0\le \ell\le n \text{ s.t. } y\in (x_{\ell-1},x_\ell]\;\;\text{and}\;\;(u^*_{\ell})_x(y^-)=0\},\\
		\end{aligned}
	\end{equation*}
	and there exist $\ell_1$ and $\ell_2$ with $1\le \ell_1\le \ell_0$ and $\ell_0\le \ell_2\le n$ such that $y_1\in [x_{\ell_1-1},x_{\ell_1})$ and $y_2\in (x_{\ell_2-1},x_{\ell_2}]$.  By the definitions of $y_1$ and $y_2$,
	\begin{equation}\label{iva1-1}
		(u^*_{\ell_1})_{xx}(y^+_1)\ge 0,\;\;(u^*_{\ell_2})_{xx}(y^-_2)\le 0,\;\;(u^*_{\ell_1})_x(y_1^+)=(u^*_{\ell_2})_x(y_2^-)=0,
	\end{equation}
	and
	\begin{subequations}\label{iva1}
		\begin{align}
			&(u^*_{i})_x>0\;\;\text{in}\;\;\Omega_i\cap(y_1,y_2) \;\;\text{for}\;\;i=\ell_1,\cdots,\ell_2,\;\;\label{iva1-a}\\
			&(u^*_{i+1})_x(x_i^{+}),\;(u^*_i)_x(x_i^{-})>0 \;\;\text{for}\;\;i=\ell_1,\cdots,\ell_2-1.\label{iva1-b}
		\end{align}
	\end{subequations}
	It follows from \eqref{m2a} and \eqref{iva1-1} that
	\begin{equation}\label{iva1-2}
		u^*_{\ell_1}(y_1^+)\ge k_{\ell_1}\;\;\text{and}\;\;u^*_{\ell_2}(y_2^-)\le k_{\ell_2},
	\end{equation}
	which contradicts \eqref{iva1-a} if $\ell_1=\ell_2$.
	If $\ell_1<\ell_2$, we deduce from  \eqref{iva1}-\eqref{iva1-2} that
	$k_{\ell_2}>u^*_{\ell_2}(x_{\ell_2-1}^+)$ and $u^*_{\ell_1}(x_{\ell_1}^-)>k_{\ell_1}$, and consequently,
	\begin{equation*}
		k_{\ell_2}>u^*_{\ell_2}(x_{\ell_2-1}^+)=p_{\ell_2-1}u^*_{\ell_2-1}(x_{\ell_2-1}^-)>\cdots >\left(\prod_{\ell=\ell_1}^{\ell_2-1}p_{\ell}\right)u^*_{\ell_1}(x_{\ell_1}^-)>k_{\ell_2},
	\end{equation*}
	where we have used \eqref{m2b} and $\mathbf p\gg\mathbf{\overline k}$ in the second and last steps, respectively. This also leads to a contradiction. Thus, Claim 1 holds.
	
	
	\emph{Claim 2:}	We claim that (i) holds for the case $n=2$.
	
	\noindent\emph{Proof of Claim:} We first show that \eqref{interf1} holds for $n=2$. Suppose not,
	then we see from \eqref{m2b} and Claim 1 that
	$d_2(u^*_{2})_{x}(x_1^+)=d_1(u^*_{1})_{x}(x_1^-)=0$,
	which implies that $u^*_i=k_i$ in $\Omega_i$ for each $i=1,2$. This combined with $\mathbf p\gg\mathbf{\overline{k}}$ yields $k_2=u_2(x_1^+)=p_1u_1(x_1^-)>k_2$, which is a contradiction. Therefore, \eqref{interf1} holds with $n=2$.
	Noticing that $(u^*_1)_x(0^{+})=0$, by Claim 1 again, we have $(u^*_1)_{xx}(0^{+})\le 0$, which yields $u_1^*(0^+)\le k_1$. If $u_1^*(0^+)= k_1$, then
	$u_1^*$ satisfies
	\begin{equation}\label{1l2}
		\begin{cases}
			\ds d_1 (u^*_1)_{ x x}+r_1 u^*_1\left(1-\frac{u^*_1}{k_1}\right)=0, & 0<x<x_1,\\
			(u^*_{1})_{x}(0^+)=0,\;\;u^*_1(0^+)=k_1.
		\end{cases}
	\end{equation}
	The uniqueness of the solution of \eqref{1l2} implies that $u_1^*\equiv k_1$, which contradicts \eqref{interf1} with $i=1$.
	Thus, $u^*_1(0^+)<k_1$, and $u^*_2(L^-)>k_2$ can be proved similarly.
	Finally, we show
	that $(u_i^*)_x<0$ in $\Omega_i$ for each $i=1,2$. Suppose not, then there exist  $1\le \tilde i\le 2$ and $\tilde x\in\Omega_{\tilde i}$ such that  $(u_{\tilde i}^*)_x(\tilde x)\ge0$. Without loss of generality, we assume that $\tilde i=1$.
	Then it follows from  Claim 1 that $(u_1^*)_x(\tilde x)=0$, and consequently, $u^*_1(x)=k_1$ for $x \in [0,\tilde x]$, which contradicts $u^*_1(0^+)<k_1$.
	This proves the claim.
	
	\emph{Claim 3:}	We claim that (i) holds for $n=m$ with $m\ge 3$ if it holds for $n=2,\cdots,m-1$.
	
	\noindent\emph{Proof of Claim:} 	
	We first show that
	\begin{equation}\label{1lm-1}
		(u_i^*)_x<0\;\;\text{in}\;\;\Omega_i\;\;
		\text{for}\;\;i=2,\cdots,m-1.
	\end{equation}
	Assume to the contrary that \eqref{1lm-1} does not hold. In view of Claim 1, there exist $2\le \ell_*\le m-1$ and $x_*\in \Omega_{\ell_*}$ such that $(u^*_{\ell_*})_{x} (x_*)=0$. Then \eqref{m2} can be rewritten
	as the following two subsystems:
	\begin{equation}\label{m2-sub1}
		\begin{cases}
			\ds d_i (u^*_{i})_{xx}+r_i u^*_i\left(1-\frac{u^*_i}{k_i}\right)=0, \;\;\;\;x\in \tilde\Omega_i,\;\; i=1,2,\cdots,\ell_*,\\
			u^*_{i+1}(x_i^{+})=p_iu^*_i(x_i^{-}),\;\;d_{i+1}(u^*_{i+1})_x(x_i^{+})=d_i(u^*_i)_x(x_i^{-}),\\
			(u^*_1)_x(0^{+})=(u^*_{\ell_*})_x(x_*^{-})=0,
		\end{cases}
	\end{equation}
	and
	\begin{equation}\label{m2-sub2}
		\begin{cases}{}
			\ds d_i (u^*_{i})_{xx}+r_i u^*_i\left(1-\frac{u^*_i}{k_i}\right)=0, \;\;\;\;x\in \hat\Omega_i,\;\; i=\ell_*,\cdots,m, \\
			u^*_{i+1}(x_i^{+})=p_iu^*_i(x_i^{-}),\;\;d_{i+1}(u^*_{i+1})_x(x_i^{+})=d_i(u^*_i)_x(x_i^{-}),\\
			(u^*_{\ell_*})_x(x_*^{+})=(u^*_m)_x(L^{-})=0,
		\end{cases}
	\end{equation}
	where
	\begin{equation*}
		\tilde\Omega_i=\begin{cases}
			\Omega_i,\;\;&i=1,\cdots,\ell_*-1,\\
			(x_{\ell_*-1},x_*),\;\;&i=\ell_*,
		\end{cases}
	\end{equation*}
	and
	\begin{equation*}		
		\hat\Omega_i=\begin{cases}
			(x_*,x_{\ell_*}),\;\;&i=\ell_*,\\
			\Omega_i,\;\;&i=\ell_*+1,\cdots,m.
		\end{cases}
	\end{equation*}
	
	Since $\mathbf p\gg\mathbf{\overline{k}}$, it follows that $[\mathbf p]_{\ell_*-1}\gg[\mathbf{\overline{k}}]_{\ell_*-1}$ and $[\mathbf p]^{\sharp}_{m-\ell_*}\gg[\mathbf{\overline{k}}]^{\sharp}_{m-\ell_*}$, where $[\cdot]_{\ell_*-1}$ and $[\cdot]^{\sharp}_{m-\ell_*}$ are defined in \eqref{truct}.
	Note that (i) holds for $n=2,\cdots,m-1$, and subsystems \eqref{m2-sub1} and \eqref{m2-sub2} consist of $\ell_*$ and $m-\ell_*+1$ patches, respectively, where $2\le\ell_*,m-\ell_*+1\le m-1$.
	Then $u^*_{\ell_*}(x^-_*)> k_{\ell_*}$ and $u^*_{\ell_*}(x^+_*)< k_{\ell_*}$, which is a contradiction. Thus, \eqref{1lm-1} holds.
	
	Then we show that \eqref{interf1} holds for $n=m$.
	If it is not true, by Claim 1 again, there exists $1\le \tilde \ell\le m-1$ such that $(u_{\tilde \ell}^*)_x(x_{\tilde \ell}^-)=0$.
	Then we can also rewrite \eqref{m2} into the following two subsystems:
	\begin{subequations}\label{m2-sub3}
		\begin{numcases}{}
			\ds d_i (u^*_{i})_{xx}+r_i u^*_i\left(1-\frac{u^*_i}{k_i}\right)=0, \;\;\;\;x\in \Omega_i,\;\; i=1,\cdots,\tilde \ell, \label{m2a-sub3}\\
			u^*_{i+1}(x_i^{+})=p_iu^*_i(x_i^{-}),\;\;d_{i+1}(u^*_{i+1})_x(x_i^{+})=d_i(u^*_i)_x(x_i^{-}),\label{m2b-sub3}\\
			(u^*_1)_x(0^{+})=(u^*_{\tilde \ell})_x(x_{\tilde \ell}^{-})=0,\label{m2c-sub3}
		\end{numcases}
	\end{subequations}
	and
	\begin{subequations}\label{m2-sub4}
		\begin{numcases}{}
			\ds d_i (u^*_{i})_{xx}+r_i u^*_i\left(1-\frac{u^*_i}{k_i}\right)=0, \;\;\;\;x\in \Omega_i,\;\; i=\tilde \ell+1,\cdots,m, \label{m2a-sub4}\\
			u^*_{i+1}(x_i^{+})=p_iu^*_i(x_i^{-}),\;\;d_{i+1}(u^*_{i+1})_x(x_i^{+})=d_i(u^*_i)_x(x_i^{-}),\label{m2b-sub4}\\
			(u^*_{\tilde \ell+1})_x(x_{\tilde \ell}^{+})=(u^*_m)_x(L^{-})=0.\label{m2c-sub4}
		\end{numcases}
	\end{subequations}
	Here we remark that, if  $\tilde \ell=1$, \eqref{m2-sub3} holds except for \eqref{m2b-sub3},  and if $\tilde \ell=m-1$, \eqref{m2-sub4} holds except for \eqref{m2b-sub4}.

	We first obtain a contradiction for the case $2\le \tilde \ell\le m-2$. Since $\mathbf p\gg\mathbf{\overline{k}}$, it follows that
	$[\mathbf p]_{\tilde\ell-1}\gg[\mathbf{\overline{k}}]_{\tilde\ell-1}$ and $[\mathbf p]^{\sharp}_{m-\tilde\ell-1}\gg[\mathbf{\overline{k}}]^{\sharp}_{m-\tilde\ell-1}.$
	Note that (i) holds for $n=2,\cdots,m-1$, and subsystems \eqref{m2-sub3} and \eqref{m2-sub4} consist of $\tilde \ell$ and $m-\tilde \ell$ patches, respectively, where $2\le \tilde \ell,m-\tilde \ell\le m-2$. Thus, $u^*_{\tilde\ell}(x_{\tilde\ell}^{-})> k_{\tilde\ell}$ and $u^*_{\tilde\ell+1}(x_{\tilde\ell}^{+})< k_{\tilde\ell+1}$, and consequently,
	\begin{equation}\label{tildeell}
		\frac{ u^*_{\tilde \ell+1}(x_{\tilde \ell}^{+})}{u^*_{\tilde \ell}(x_{\tilde \ell}^{-})}=p_{\tilde \ell}<\frac{k_{\tilde\ell+1}}{k_{\tilde\ell}},
	\end{equation}
	which contradicts $\mathbf p\gg\mathbf{\overline k}$.
	If $\tilde \ell=1$, then $u_{\tilde \ell}^*= k_{\tilde \ell}$ in $\Omega_{\tilde \ell}$.
	Note that $[\mathbf p]^{\sharp}_{m-\tilde\ell-1}\gg[\mathbf{\overline{k}}]^{\sharp}_{m-\tilde\ell-1}$ and subsystems \eqref{m2-sub4} consist of $m-1$ patches. Then
	$u^*_{\tilde\ell+1}(x_{\tilde\ell}^{+})< k_{\tilde\ell+1}$, which implies that \eqref{tildeell} also holds. This leads to a contradiction. Using similar arguments, we can obtain a contradiction for the case
	$\tilde \ell=m-1$. Therefore, \eqref{interf1} holds for $n=m$.
	Finally, using arguments similar to those in the proof of Claim 2, we can obtain that $u_1^*(0^+)<k_1$, $u_m^*(L^-)>k_m$ and
	$(u^*_i)_x<0$ in $\Omega_i$ for $i=1,m$. This completes the proof of the claim.
	
	Therefore, (i) holds by Claims 2-3 and induction.
\end{proof}

For the two-patch case, the steady state $\mathbf{u^*}$ is necessarily monotonic. However, this monotonicity can be lost for $n \geq 3$. While Lemma \ref{lemma2.1-n} provides sufficient conditions for monotonicity, the steady state may otherwise be non-monotonic, as illustrated by the concrete example below.
\begin{proposition}\label{example}
	Suppose that $n=3$ and let $\mathbf u^*$ be the unique positive solution of \eqref{m2}. If
	$k_1=k_3:=k$ and $p_1p_2=1$ with $p_1\neq k_2/k$ (or equivalently, $p_2\neq k/k_2$), then $\mathbf u^*$ is non-monotonic.
\end{proposition}
\begin{proof}
	Without loss of generality, we assume that $p_1 > k_2 / k$. Then it follows from $p_1 p_2 = 1$ that $p_2 < k / k_2$. We claim that there exists $x_*\in (x_1, x_2)$ such that $(u^*_2)_x(x_*) = 0$. If it is not true, then  $(u_2^*)_x>0$ or $(u_2^*)_x<0$ in $(x_1,x_2)$. We only derive a contradiction for the case $(u_2^*)_x>0$ in $(x_1,x_2)$, as the other case is analogous.
	
	By continuity of $(u_2^*)_x$ on $[x_1,x_2]$, we have $(u_2^*)_x(x_1^+), (u_2^*)_x(x_2^-)\ge0$. We now divide the arguments into two cases:
		\begin{equation*}
		{\rm(a_1)}\;(u_2^*)_x(x_1^+), (u_2^*)_x(x_2^-)>0;\;\;{\rm(a_2)}\;(u_2^*)_x(x_1^+)=0\;\;\text{or} \;\;(u_2^*)_x(x_2^-)=0.
	\end{equation*}
	For case $(a_1)$, we see from \eqref{m2c} that
	\begin{equation}\label{y1y2}
		y_1:=\sup \{y\in [0,x_1): (u^*_{1})_x(y^+)=0\}\;\;\text{and}\;\;y_2:=\inf \{y\in (x_2,L]: (u^*_{3})_x(y^-)=0\}
	\end{equation}
	are well-defined with $0\le y_1<x_1<x_2< y_2\le L$. Thus, we have
	\begin{equation}\label{dis2}
	(u^*_{1})_{xx}(y_1^+)\ge 0,\;\;(u^*_{3})_{xx}(y_2^-)\le 0
	\end{equation}
	and
	\begin{equation}\label{dis2-1}
			(u^*_{i})_x>0\;\;\text{in}\;\;\Omega_i\cap(y_1,y_2) \;\;\text{for}\;\;i=1,2,3.
	\end{equation}
	It follows from \eqref{m2a} and \eqref{dis2} that $u^*_1(y_1^+)\ge k_1=k$ and $u^*_3(y_2^-)\le k_3= k$. This combined with \eqref{dis2-1} implies that
	\begin{equation}\label{dis4}
		k> u^*_3(x_2^+)=p_2u^*_2(x_2^-)>p_2u^*_2(x_1^+)=p_1p_2u^*_1(x_1^-)> k,
	\end{equation}
	where we have used \eqref{m2b} in the second and fourth steps. This leads to a contradiction.

	Now we derive a contradiction for case $(a_2)$.  For simplicity, we only consider the case that $(u_2^*)_x(x_1^+)=0$ and $(u_2^*)_x(x_2^-)>0$, as other cases can be treated similarly. Then $\mathbf u^*$ satisfies
	\begin{equation}\label{dis1}
		\begin{cases}
			\ds d_1 (u^*_1)_{ x x}+r_1 u^*_1\left(1-\frac{u^*_1}{k}\right)=0, & 0<x<x_1,\\
			(u^*_{1})_{x}(0)=	(u^*_{1})_{x}(x_1^-)=0.
		\end{cases}
	\end{equation}
	The uniqueness of the solution to \eqref{dis1} implies $u^*_1\equiv k$ in $[0,x_1]$. Thus, by an argument similar to that  in \eqref{dis4}, we obtain
		\begin{equation*}
		k\ge u^*_3(y_2^-)> u^*_3(x_2^+)=p_2u^*_2(x_2^-)>p_2u^*_2(x_1^+)=u^*_1(x_1^-)=k,
	\end{equation*}
where $y_2$ is defined as in \eqref{y1y2}. This also leads to a contradiction. Therefore, the claim holds.
	
	Then we can rewrite \eqref{m2} into the following two subsystems:
	\begin{equation*}
		\begin{cases}
			\ds d_i (u^*_{i})_{xx}+r_i u^*_i\left(1-\frac{u^*_i}{k_i}\right)=0, \;\;\;\;x\in \tilde\Omega_i,\;\; i=1,2,\\
			u^*_{2}(x_1^{+})=p_1u^*_1(x_1^{-}),\;\;d_{2}(u^*_{2})_x(x_1^{+})=d_1(u^*_1)_x(x_1^{-}),\\
			(u^*_1)_x(0^+)=(u^*_{2})_x(x_*^-)=0,
		\end{cases}
	\end{equation*}
	and
	\begin{equation*}
		\begin{cases}
			\ds d_i (u^*_{i})_{xx}+r_i u^*_i\left(1-\frac{u^*_i}{k_i}\right)=0, \;\;\;\;x\in \hat\Omega_i,\;\; i=2,3,\\
			u^*_{3}(x_2^{+})=p_2u^*_2(x_2^{-}),\;\;d_{3}(u^*_{3})_x(x_2^{+})=d_2(u^*_2)_x(x_2^{-}),\\
			(u^*_2)_x(x_*^+)=(u^*_{3})_x(L^-)=0,
		\end{cases}
	\end{equation*}
	where
	\begin{equation*}
		\tilde\Omega_i=\begin{cases}
			\Omega_1,\;\;&i=1,\\
			(x_{1},x_*),\;\;&i=2,
		\end{cases}
	\end{equation*}
	and
	\begin{equation*}		
		\hat\Omega_i=\begin{cases}
			(x_*,x_{2}),\;\;&i=2,\\
			\Omega_3,\;\;&i=3.
		\end{cases}
	\end{equation*}
Observing that $p_1 > k_2 / k$ and $p_2 < k / k_2$, by applying Lemma \ref{lemma2.1-n} to each subsystem, we see $(u^*_i)_x < 0$ in $\tilde{\Omega}_i$ for $i = 1, 2$, and $(u^*_i)_x > 0$ in $\hat{\Omega}_i$ for $i = 2, 3$. Consequently, $\mathbf{u}^*$ is non-monotonic on $[0,L]$.
\end{proof}

\section{Comparison principle}
	Throughout this section, we define an order
in $\left(C(\overline \Omega_1)\times\cdots\times C(\overline \Omega_n)\right)^2$: $({\mathbf u}^1,\mathbf{v}^1)\succeq (\mathbf {u}^2,\mathbf{v}^2)$
if and only if
\begin{equation}\label{order}
	u^1_i(x)\ge u^2_i(x) \;\;\text{and}\;\; v^1_i(x)\le v^2_i(x)\;\;\text{for all}\;\;x\in\overline\Omega_i\;\;\text{and}\;\;i=1,\cdots,n,
\end{equation}
Then we show that system \eqref{m1} preserves the order ``$\succeq$''.

\begin{theorem}\label{appcp}
	Let $\overline{\mathbf u},\underline{\mathbf v},\overline{\mathbf v},\underline{\mathbf u}\in Y^+$,
	where
\begin{equation*}
\begin{aligned}
	Y^+ := \{
	&(u_1, \cdots, u_n): u_i \in C^{2,1}\left (\overline{\Omega}_i \times (0, T) \right) \cap C\left( \overline{\Omega}_i \times [0, T) \right) \\
		&u_i \ge 0 \text{ in } \overline{\Omega}_i \times [0, T)\;\;\text{for all}\;\;i=1,\cdots,n\}
	\end{aligned}
\end{equation*}
	with $T \in(0,+\infty]$.
	Suppose that $(\overline{\mathbf u},\underline{\mathbf v})$ and $(\underline{\mathbf u},\overline{\mathbf v})$ satisfy, for $i=1,\cdots,n-1$,
	\begin{equation}\label{B2}
		\begin{split}
			&u_{i+1}(x_i^{+},t)=p_i u_i(x_i^{-},t),\;d_{i+1}( u_{i+1})_{x}(x_i^{+},t)=d_i ( u_{i})_{x}(x_i^{-},t),\;t>0,\\
			&v_{i+1}(x_i^{+},t)=\hat p_iv_i(x_i^{-},t),\;\hat d_{i+1}( v_{i+1})_{x}(x_i^{+},t)=\hat d_i ( v_{i})_{x}(x_i^{-},t),\;t>0,\\
		\end{split}
	\end{equation}
	and for $i=1,\cdots,n$,
	\begin{equation}\label{compr}
		\begin{cases}
			(\overline{u}_i)_t-d_i( \overline{u}_i)_{xx}-f_i(\overline{u}_i, \underline{v}_i)\ge0\ge	(\underline{u}_i)_t-d_i(\underline{u}_i)_{xx}-f_i(\underline{u}_i, \overline{v}_i),\;x\in \Omega_i,\;t>0,\\
			(\overline{v}_i)_t-\hat d_i( \overline{v}_i)_{xx}-\hat{f}_i(\underline{u}_i,\overline{v}_i)\ge0\ge(\underline{v}_i)_t-\hat d_i( \underline{v}_i)_{xx}-\hat{f}_i(\overline{u}_i,\underline{v}_i),\;x\in \Omega_i,\;t>0,\\ (\overline{u}_1)_x(0^{+},t)\le0\le(\underline{u}_1)_x(0^{+},t),\;(\overline{u}_{n})_x(L^{-},t)\ge0\ge(\underline{u}_{n})_x(L^{-},t),\;\;\;t>0,\\ (\overline{v}_1)_x(0^{+},t)\le0\le(\underline{v}_1)_x(0^{+},t),\;(\overline{v}_{n})_x(L^{-},t)\ge0\ge(\underline{v}_{n})_x(L^{-},t),\;\;\;t>0,
		\end{cases}
	\end{equation}
	where
	\begin{equation}\label{fhatf}
		f_i(u_i, v_i)=r_iu_i\left(1-\dfrac{u_i+v_i}{k_i}\right),\;\;\hat{f}_i(u_i, v_i)=r_iv_i\left(1-\dfrac{u_i+v_i}{k_i}\right).
	\end{equation}
	If $(\overline{\mathbf u}(\cdot,0),\underline{\mathbf v}(\cdot,0))\succeq (\underline{\mathbf u}(\cdot,0),\overline{\mathbf v}(\cdot,0))$, then $(\overline{\mathbf u}(\cdot,t),\underline{\mathbf v}(\cdot,t))\succeq (\underline{\mathbf u}(\cdot,t),\overline{\mathbf v}(\cdot,t))$ for any $t\in (0,T)$.
\end{theorem}
\begin{proof}
	Fix any $T_0 \in(0, T)$, we denote
	\begin{equation*}
		M:=\max_{1\le i\le n} \left(\|\overline{u}_i\|_{\infty},\|\underline{u}_i\|_{\infty},\|\overline{v}_i\|_{\infty},\|\underline{v}_i\|_{\infty}\right),
	\end{equation*}
	where $\|w_i\|_{\infty}=\max_{(x,t)\in\overline{\Omega}_i\times[0,T_0]}\lvert w_i(x,t)\rvert$ for $w_i=\overline{u}_i,\;\underline{u}_i,\;\overline{v}_i,\;\underline{u}_i$,
	and let
	\begin{equation*}
		\mu:=\max_{1\le i\le n}\left\{\max_{[0,M]\times[0,M]}\left(\Big\lvert\dfrac{\partial f_i}{\partial u_i}\Big\rvert+\Big\lvert\dfrac{\partial f_i}{\partial v_i}\Big\rvert+\Big\lvert\dfrac{\partial \hat{f}_i}{\partial u_i}\Big\rvert+\Big\lvert\dfrac{\partial \hat{f}_i}{\partial v_i}\Big\rvert\right)\right\}.
	\end{equation*}
	Inspired by \cite{44-Zhang-2024}, we define, for $i=1,\cdots,n$,
	\begin{equation*}
		\begin{split}
			&\theta_i(x, t):=(\overline{u}_i(x, t)-\underline{u}_i(x, t)) e^{-\mu t}+\epsilon(t+1),\\
			&\rho_i(x, t):=(\overline{v}_i(x, t)-\underline{v}_i(x, t)) e^{-\mu t}+\epsilon(t+1). \\
		\end{split}
	\end{equation*}
	Then it  follows from the mean value theorem that there exist
	\begin{equation*}
		\begin{aligned}
			&a_i(x,t):=\dfrac{\partial f_i}{\partial u_i}\left(\eta_i(x,t),\zeta_i(x,t)\right), \;\;b_i(x,t):=\dfrac{\partial f_i}{\partial v_i}\left(\eta_i(x,t),\zeta_i(x,t)\right)\le0,\\
			&\hat{a}_i(x,t):=\dfrac{\partial \hat{f}_i}{\partial u_i}(\hat{\eta}_i(x,t),\hat{\zeta}_i(x,t))\le0,\;\; \hat{b}_i(x,t):=\dfrac{\partial \hat{f}_i}{\partial v_i}\left(\hat{\eta}_i(x,t),\hat{\eta}_i(x,t)\right),\\
		\end{aligned}
	\end{equation*}
	where
	$0\le \eta_i(x,t),\;\zeta_i(x,t),\;\hat{\eta}_i(x,t),\hat{\zeta}_i(x,t)\le M$ for $(x,t)\in\Omega_i\times(0,T_0]$,
	such that
	\begin{equation}\label{n1n2}
		\begin{aligned} \mathcal{L}_i\left({\theta}_i,\rho_i\right):=&({\theta}_i)_t-d_i({\theta}_i)_{xx}+\left(\mu-a_i(x,t)\right){\theta}_i+b_i(x,t){\rho}_i \\
			=&e^{-\mu t}\left[(\overline{u}_i-\underline{u}_i)_t -d_i(\overline{u}_i-\underline{u}_i)_{xx}-a_i(x,t)(\overline{u}_i-\underline{u}_i)+b_i(x,t)\left(\overline{v}_i-\underline{v}_i\right)\right]\\
			&+\epsilon+\left(\mu-a_i(x,t)+b_i(x,t)\right)\epsilon(t+1)>0\;\;\text{in}\;\; \Omega_i\times\left(0, T_0\right],\\ \hat{\mathcal{L}}_i\left({\theta}_i,\rho_i\right):=&({\rho}_i)_t-\hat{d}_i({\rho}_i)_{xx}+\left(\mu-\hat{b}_i(x,t)\right){\rho}_i+\hat{a}_i(x,t){\theta}_i \\
			=&e^{-\mu t}\left((\overline{v}_i-\underline{v}_i)_t -\hat{d}_i(\overline{v}_i-\underline{v}_i)_{xx}-\hat{b}_i(x,t)(\overline{v}_i-\underline{v}_i)+\hat{a}_i(x,t)\left(\overline{u}_i-\underline{u}_i\right)\right)\\
			&+\epsilon+\left(\mu-\hat{b}_i(x,t)+\hat{a}_i(x,t)\right)\epsilon(t+1)>0\;\;\text{in}\;\; \Omega_i\times\left(0, T_0\right].\\
		\end{aligned}
	\end{equation}
	Moreover,
	\begin{subequations}\label{apwthe}
		\begin{align}
			&d_i({\theta}_i)_{x}(x_i^{-},t)= d_{i+1}({\theta}_{i+1})_{x}(x_i^{+},t),\;\hat{d}_i({\rho}_i)_{x}(x_i^{-},t)=\hat{d}_{i+1}({\rho}_{i+1})_{x}(x_i^{+},t),\label{apwthe-a}\\
			&({\theta}_1)_x(0^+,t), ({\rho}_1)_x(0^+,t)\le0,\;	({\theta}_{n})_x(L^-,t), ({\rho}_{n})_x(L^-,t)\ge0,\label{apwthe-b}\\
			&{\theta}_i(\cdot,0),\;{\rho}_i(\cdot,0)> 0\;\;\text{in}\;\;\overline{\Omega}_i.\label{apwthe-c}
		\end{align}
	\end{subequations}

We first prove that ${\theta}_i,{\rho}_i>0$ in $(x,t) \in\overline{\Omega}_i\times[0, T_0]$ for $i=1,\cdots,n$.
Suppose, for the sake of contradiction, that this is not true. Then, by \eqref{apwthe-c}, at least one of the following quantities is well-defined and lies in the interval $(0, T_0]$:
\begin{equation*}
\begin{split}
&t_1:=\inf\{t\in[0,T_0]: \exists 1\le i_0\le n\;\text{and}\;y_0\in\overline{\Omega}_{i_0}\;\;\text{s.t.}\;\;{\theta}_{i_0}(y_0, t)=0\},\\
&t_2:=\inf\{t\in[0,T_0]: \exists 1\le i_0\le n\;\text{and}\;y_0\in\overline{\Omega}_{i_0}\;\;\text{s.t.}\;\;{\rho}_{i_0}(y_0, t)=0\}.
\end{split}
\end{equation*}
Without loss of generality, we assume that both $t_1$ and $t_2$ exist satisfying $t_1\le t_2$. By the definition of $t_1$ and $t_2$, we have
\begin{equation}\label{initial}
\begin{split}
&{\theta}_{i_0}(x, t)>0\;\;\text{for}\;\;x\in\Omega_i,\; t\in[0,t_1),\\
&{\rho}_{i_0}(x, t)\ge0\;\;\text{for}\;\;x\in\Omega_i,\; t\in[0,t_1].
\end{split}
\end{equation}
Now we claim that
\begin{equation}\label{suptheta}
{\theta}_{i_0}(x, t_1)>0\;\;\text{for}\;\;x\in\Omega_{i_0}.
\end{equation}
If it is not true, then there exists $y_0\in \Omega_{i_0}$ such that ${\theta}_{i_0}(y_0, t_1)=0$. Then
	$({\theta}_{i_0})_t|_{(y_0,t_1)} \le 0$, $({\theta}_{i_0})_x|_{(y_0,t_1)}=0$ and $({\theta}_{i_0})_{x x}|_{(y_0,t_1)}\ge 0$.
	Thus,
	\begin{equation*}
		\begin{aligned} \mathcal{L}_{i_0}\left({\theta}_{i_0},{\rho}_{i_0}\right)|_{(y_0,t_1)}=&({\theta}_{i_0})_t|_{(y_0,t_1)}-d_{i_0}({\theta}_{i_0})_{xx}|_{(y_0,t_1)}\\
			&+\left(\mu-a_{i_0}(y_0,t_1)\right){\theta}_{i_0}(y_0,t_1)+b_{i_0}(y_0,t_1){\rho}_{i_0}(y_0,t_1)\le 0,
		\end{aligned}
	\end{equation*}
	which contradicts \eqref{n1n2}. Therefore, the claim is true and \eqref{suptheta} holds.
Then, by the definition of $t_1$ again, we see that ${\theta}_{i_0}(x_{i_0-1}^+, t_1)=0$ or ${\theta}_{i_0}(x_{i_0}^-, t_1)=0$.
 Now we obtain a contradiction for the case ${\theta}_{i_0}(x_{i_0}^-, t_1)=0$, and the other case can be treated similarly.
 If $i_0=n$, it follows from \eqref{initial}-\eqref{suptheta} and the Hopf's lemma that $({\theta}_{i_0})_x(x_{i_0}^{-},t_0)<0$, which contradicts \eqref{apwthe-b}. If $i_0<n$, by \eqref{initial}-\eqref{suptheta} and the Hopf's lemma again, we have $({\theta}_{i_0})_x(x_{i_0}^{-},t_0)<0$ and $({\theta}_{i_0+1})_x(x_{i_0}^{+},t_0)>0$, which contradicts \eqref{apwthe-a}.
	
Thus, ${\theta}_i,{\rho}_i>0$ in $\overline{\Omega}_i\times[0, T_0]$ for $i=1,\cdots,n$.
	Since $\epsilon>0$  and $T_0 \in(0, T)$ are arbitrary, we conclude that $\overline{u}_i\ge\underline{u}_i$ and $\underline{v}_i\le\overline{v}_i$ in $\overline{\Omega}_i\times(0, T)$ for $i=1,\cdots,n$. This completes the proof.
\end{proof}
\begin{remark}\label{strict}
	We remark that if condition $(\overline{\mathbf u}(\cdot,0),\underline{\mathbf v}(\cdot,0))\succeq(\underline{\mathbf u}(\cdot,0),\overline{\mathbf v}(\cdot,0))$ in Theorem \ref{appcp} is replaced by $(\overline{\mathbf u}(\cdot,0),\underline{\mathbf v}(\cdot,0))\succeq (\not\equiv)(\underline{\mathbf u}(\cdot,0),\overline{\mathbf v}(\cdot,0))$, one can obtain that $\overline{\mathbf u}(x,t)\gg\underline{\mathbf u}(x,t)$ and $\overline{\mathbf v}(x,t)\gg \underline{\mathbf v}(x,t)$ for all $x\in[0,L]$ and $t\in (0,T]$. The proof is also based on the comparison principle for classical reaction-diffusion equations, and here we omit the proof for simplicity.
\end{remark}
By Theorem \ref{appcp} and the definitions of sub-equilibrium and super-equilibrium
(see \cite[Chapter I]{41-Hess-1991}), we can obtain the following result.
\begin{proposition}\label{supersub}
	Let  $\overline{\mathbf{w}}, \overline{\mathbf{z}}, \underline{\mathbf{w}}, \underline{\mathbf{z}}\in X^+$ with
	\begin{equation}\label{c2app}
		X^+:=\{\mathbf u=(u_1,\cdots,u_n):u_i\in C^2(\overline{\Omega}_i)\;\text{and}\;u_i\ge0\;\text{in}\;\overline{\Omega}_i\;\;\text{for all}\;\;i=1,\cdots,n\}.
	\end{equation}
	Assume that $(\overline{\mathbf{w}}, \underline{\mathbf{z}})$ and $(\underline{\mathbf{w}}, \overline{\mathbf{z}})$  satisfy \eqref{m4-c}-\eqref{m4-d}, and for $i=1,\cdots,n$,
	\begin{equation*}
		\begin{cases}
			-d_i( \overline{w}_i)_{xx}-f_i(\overline{w}_i, \underline{z}_i)\ge0\ge	-d_i(\underline{w}_i)_{xx}-f_i(\underline{w}_i, \overline{z}_i),\;\; x\in \Omega_i,\\
			-\hat d_i( \overline{z}_i)_{xx}-\hat{f}_i(\underline{w}_i,\overline{z}_i)\ge0\ge-\hat d_i( \underline{z}_i)_{xx}-\hat{f}_i(\overline{w}_i,\underline{z}_i),\; \;x\in \Omega_i,\\ (\overline{w}_1)_x(0^{+})\le0\le(\underline{w}_1)_x(0^{+}),\;\;(\overline{w}_{n})_x(L^{-})\ge0\ge(\underline{w}_{n})_x(L^{-}),\\ (\overline{z}_1)_x(0^{+})\le0\le(\underline{z}_1)_x(0^{+}),\;\;(\overline{z}_{n})_x(L^{-})\ge0\ge(\underline{z}_{n})_x(L^{-}).
		\end{cases}
	\end{equation*}
	Then $(\overline{\mathbf{w}}, \underline{\mathbf{z}})$ (resp. $(\underline{\mathbf{w}}, \overline{\mathbf{z}})$) is a super-equilibrium (resp. sub-equilibrium) of system \eqref{m1}.
\end{proposition}
Then, by \cite[Lemma 1.1]{41-Hess-1991}, we have the following result.
\begin{theorem}\label{pss}
	Let $(\overline{\mathbf{w}}, \underline{\mathbf{z}})$  be a super-equilibrium of system \eqref{m1}, and  let $(\underline{\mathbf{w}}, \overline{\mathbf{z}})$ be a sub-equilibrium of system \eqref{m1} with
	$\overline{\mathbf{w}}, \overline{\mathbf{z}}, \underline{\mathbf{w}}, \underline{\mathbf{z}}\in X^+$ (defined in \eqref{c2app}). If $(\overline{\mathbf{w}}, \underline{\mathbf{z}})\succeq (\underline{\mathbf{w}}, \overline{\mathbf{z}})$, then system \eqref{m1} admits a steady state $(\mathbf u,\mathbf v)$ satisfying $(\overline{\mathbf{w}}, \underline{\mathbf{z}})\succeq (\mathbf u,\mathbf v)\succeq(\underline{\mathbf{w}}, \overline{\mathbf{z}})$.
\end{theorem}

We end this section with a comparison principle for an  elliptic equation.
\begin{theorem}\label{parabolic}
	Assume that, for each $i=1,\cdots,n$, the function $c_i(x)$ is positive and bounded. If $(u_1,\cdots,u_n)\in C^{2}( \overline{\Omega}_1)\times \cdots\times C^{2}( \overline{\Omega}_n)$ satisfies
	\begin{subequations}\label{bm3}
	\begin{empheq}[left=\empheqlbrace]{align}
		&\ds -d_i (u_i)_{xx}+c_i(x) u_i= f_i(x),\;\; x\in \Omega_i,\;\; i=1,2,\cdots,n, \label{bm3-a} \\
		&u_{i+1}(x_i^{+})= p_iu_i(x_i^{-}), \quad d_{i+1}(u_{i+1})_{x}(x_i^{+})= d_i (u_{i})_{x}(x_i^{-}), \label{bm3-b} \\
		&(u_{1})_{x}(0^{+})\le0\le(u_{n})_{x}(L^{-}),\label{bm3-c}
	\end{empheq}
\end{subequations}
where $f_i\ge 0$ in $\overline{\Omega}_i$ for all $i=1,\cdots,n$, and there exists $1\le i^*\le n$ such that $f_{i^*}\not\equiv 0$,
then $u_i>0$ in $\overline{\Omega}_i$ for all $i=1,\cdots,n$.
\end{theorem}
	\begin{proof}
The proof follows the technique of  \cite[Proposition 3.3]{36-Maciel-2020}, which we generalize to the $n$-patch case.
		Let  $m:=\min_{1\le i \le n,\, x\in \overline{\Omega}_i}u_i(x)$. We first prove that $m\ge0$. Suppose, for contradiction, that $m<0$.
We claim that
\begin{equation}\label{interior}
u_i>m \text{  in  } \Omega_i \text{  for each  }i=1,\cdots,n.
\end{equation}
If this claim is false, then there exists $1\le i_0\le n$ and $y_0\in{\Omega}_{i_0}$ such that ${u}_{i_0}(y_0)=m<0$. Then $(u_{i_0})_{xx}(y_0) \geq 0$, which implies that
	\begin{equation*}
		-d_{i_0} (u_{i_0})_{xx}(y_0) + c_{i_0}(y_0) u_{i_0}(y_0) \leq c_{i_0}(y_0) u_{i_0}(y_0) < 0,
		\end{equation*}
contradicting \eqref{bm3-a}. Therefore, \eqref{interior} holds, and consequently, there exists $1\le i_1\le n $ such that
$u_{i_1}(x_{i_1-1}^-)=m$ or $u_{i_1}(x_{i_1}^+)=m$. Without loss of generality, we assume that $u_{i_1}(x_{i_1}^-)=m$.
If $i_1=n$, it follows from the Hopf's Lemma and \eqref{interior} that
$(u_n)_x(x_{n}^-)<0$, which contradicts \eqref{bm3-c}. If $1\le i_1<n$, by the Hopf's Lemma and \eqref{interior} again, we have $(u_{i_1})_x(x_{i_1}^-)<0$ and $(u_{i_1+1})_x(x_{i_1}^+)>0$, which contradicts \eqref{bm3-b}. Therefore, $m\ge0$.
		
Now we prove that $m>0$. First, we claim that
\begin{equation}\label{claims1}
u_{i^*}>0 \text{ in }\Omega_{i^*}.
 \end{equation}
 If this is false, the strong maximum principle implies that $u_{i^*}\equiv0$ in $\Omega_{i^*}$. This contradicts \eqref{bm3-a} since $f_{i^*}\not\equiv0$. Therefore, \eqref{claims1} holds.
 Next, we show that $u_{i^*}(x_{i^*}^-)>0$, and the proof for $u_{i^*}(x_{i^*-1}^+)>0$ is similar. Suppose to the contrary that $u_{i^*}(x_{i^*}^-)=0$.
If $i^*=n$, it follows from  the Hopf's Lemma and  \eqref{claims1} that
$(u_{i^*})_x(x_{i^*}^-)<0$, which  contradicts \eqref{bm3-c}. If $1\le i^*<n$, the Hopf's Lemma and \eqref{interior} again gives $(u_{i^*})_x(x_{i^*}^-)<0$.
Since $u_{i^*+1}\ge0$ in $\overline \Omega_{i^*+1}$, we have $(u_{i^*+1})_x(x_{i^*}^+)\ge0$, which contradicts \eqref{bm3-b}. Therefore, $u_{i^*}>0$ in $\overline \Omega_{i^*}$.
Finally, we show that  $u_{i}>0$ in $\overline \Omega_{i}$ for all $i\ne i^*$. Without loss of generality, we assume that $1\le i^*<n$. The strong maximum principle gives $u_{i^*+1}>0$ in $\Omega_{i^*+1}$. Using arguments similar to those in the proof of $u_{i^*}(x_{i^*}^-)>0$, we see that $u_{i^*+1}(x_{i^*+1}^-)>0$, and consequently, $u_{i^*+1}>0$ in $\overline \Omega_{i^*+1}$.
The result for all remaining $i$ follows by induction.
\end{proof}

\section{Set-up of the problem}\label{AppendixC}
	
Following \cite{44-Zhang-2024}, we define a classical solution to model \eqref{m1} as follows:
	\begin{definition}\label{defc1}
		For $T \in (0, +\infty]$, we say that a function $$(\mathbf u,\mathbf v)=(u_1,\cdots, u_n, v_1,\cdots,v_n):  [0,L]\times[0, T)  \rightarrow \mathbb{R}^n\times\mathbb{R}^n$$ is a classical solution of  model \eqref{m1} in $ [0,L]\times[0, T)$ if:
		\begin{enumerate}[itemsep=1.25pt, parsep=1.25pt, topsep=1.25pt]
			\item [(i)]  $u_i, v_i$ is of class $C^{2,1}( \overline{\Omega}_i\times (0, T) )\cap C( \overline{\Omega}_i\times [0, T) ) $ for each $i=1,\cdots,n$;
			\item [(ii)]  All identities in \eqref{m1} are satisfied pointwise for $0 < t < T$.
		\end{enumerate}
	\end{definition}
To show the existence of a classical solution of system \eqref{m1}, we first analyze the abstract evolution equation based on semigroup theory, inspired by \cite{44-Zhang-2024,36-Maciel-2020}.  Since their models differ slightly from ours, the associated function spaces and operators must be adapted accordingly. For completeness, we now detail these necessary modifications.
We begin by defining the function spaces
	\begin{equation}\label{spacex}
		X_{u} ,X_{v} := L^{2}(\Omega_1) \times \cdots \times L^{2}(\Omega_n) \;\; \text{and}\;\; X := X_{u} \times X_{v},
	\end{equation}
equipped with the following inner products:
\begin{equation*}
    \begin{aligned}
        &\langle \mathbf{u}, \hat{\mathbf{u}} \rangle_{X_{u}} = \sum_{i=1}^{n} \frac{1}{\prod_{j=1}^{i-1} p_j} \langle u_{i}, \hat{u}_{i} \rangle_{L^{2}(\Omega_i)} && \text{for } \mathbf{u}, \hat{\mathbf{u}} \in X_{u}, \\
        &\langle \mathbf{v}, \hat{\mathbf{v}} \rangle_{X_{v}} = \sum_{i=1}^{n} \frac{1}{\prod_{j=1}^{i-1} \hat{p}_j} \langle v_{i}, \hat{v}_{i} \rangle_{L^{2}}(\Omega_i) && \text{for } \mathbf{v}, \hat{\mathbf{v}}  \in X_{v}, \\
        &\langle \mathbf{w}, \mathbf{z} \rangle_{X} = \langle \mathbf{u}, \hat{\mathbf{u}} \rangle_{X_{u}} + \langle \mathbf{v}, \hat{\mathbf{v}} \rangle_{X_{v}} && \text{for } \mathbf{w}=(\mathbf{u}, \mathbf{v}),\ \mathbf{z}=(\hat{\mathbf{u}}, \hat{\mathbf{v}}) \in X.
    \end{aligned}
\end{equation*}
Theses inner products induce the norms $\| \cdot\|_{X_u}=\sqrt{\langle \cdot, \cdot \rangle_{X_u}}$, $\| \cdot\|_{X_v}=\sqrt{\langle \cdot, \cdot \rangle_{X_v}}$ and  $\| \cdot\|_{X}=\sqrt{\langle \cdot, \cdot \rangle_{X}}$. Equipped with these norms,  $X_u$, $X_v$ and $X$ become Hilbert spaces, respectively.
	
Then define the function spaces
	\begin{equation}\label{hu1}
	\begin{aligned}
		&\mathcal{H}^1_{u} :=\{\mathbf u\in H^{1}(\Omega_1) \times \cdots \times H^{1}(\Omega_n): u_{i+1}(x_i^{+})=p_i u_i(x_i^{-}),i=1,\cdots,n-1\},\\
		&\mathcal{H}^1_{v} :=\{\mathbf v\in H^{1}(\Omega_1) \times \cdots \times H^{1}(\Omega_n): v_{i+1}(x_i^{+})=\hat{p}_i v_i(x_i^{-}),i=1,\cdots,n-1\},\\
&\mathcal{H}^1 = \mathcal{H}^1_{u} \times 	\mathcal{H}^1_{v},
	\end{aligned}
	\end{equation}
equipped with the following norms:
\begin{equation*}
\begin{split}
	&\| \mathbf u \|_{\mathcal{H}_u^1}= \sum_{i=1}^{n}\dfrac{1}{\prod_{j=1}^{i-1}p_j}\| u_i \|_{H^{1}(\Omega_i)}\;\;\text{for}\;\;\mathbf{u}\in \mathcal{H}_u^1,\\
	&\| \mathbf v \|_{\mathcal{H}_v^1}= \sum_{i=1}^{n}\dfrac{1}{\prod_{j=1}^{i-1}\hat{p}_j}\| v_i \|_{H^{1}(\Omega_i)}\;\;\text{for}\;\;\mathbf{v}\in \mathcal{H}_v^1,\\
&\| \mathbf w \|_{\mathcal{H}^1}= \| \mathbf u\|_{\mathcal{H}_u^1}+\| \mathbf v \|_{\mathcal{H}_v^1}\;\;\text{for}\;\;\mathbf{w}=(\mathbf{u}, \mathbf{v})\in \mathcal{H}^1.\\
\end{split}
	\end{equation*}
By the Sobolev embedding theorem, function spaces $\mathcal{H}^1_{u}$ and $\mathcal{H}^1_{v}$ are closed subspaces of the product space $H^{1}(\Omega_1) \times \cdots \times H^{1}(\Omega_n)$, which implies that $\mathcal{H}^1_{u}$, $\mathcal{H}^1_{v}$ and  $\mathcal{H}^1$ are Banach spaces.

We further define, for any integer $k\ge2$, the function spaces
	\begin{equation}\label{spaceh2}
		\begin{aligned}
			&\mathcal{H}^k_{u},\, \mathcal{H}^k_{v} :=  H^{k}(\Omega_1) \times \cdots \times H^{k}(\Omega_n) \;\;\text{and}\;\; 	\mathcal{H}^k := 	\mathcal{H}^k_{u} \times 	\mathcal{H}^k_{v}
		\end{aligned}
	\end{equation}
	equipped with the following norms, respectively:
	\begin{equation*}
		\begin{split}
			&\| \mathbf u \|_{\mathcal{H}_u^k}= \sum_{i=1}^{n}\dfrac{1}{\prod_{j=1}^{i-1}p_j}\| u_i \|_{H^{k}(\Omega_i)}\;\;\text{for}\;\;\mathbf{u}\in \mathcal{H}_u^k,\\
			&\| \mathbf v \|_{\mathcal{H}_v^k}= \sum_{i=1}^{n}\dfrac{1}{\prod_{j=1}^{i-1}\hat{p}_j}\| v_i \|_{H^{k}(\Omega_i)}\;\;\text{for}\;\;\mathbf{v}\in \mathcal{H}_v^k,\\
			&\| \mathbf w \|_{\mathcal{H}^k}= \| \mathbf u\|_{\mathcal{H}_u^k}+\| \mathbf v \|_{\mathcal{H}_v^k}\;\;\text{for}\;\;\mathbf{w}=(\mathbf{u}, \mathbf{v})\in \mathcal{H}^k.\\
		\end{split}
	\end{equation*}
Clearly, function spaces $\mathcal{H}^k_{u}$, $\mathcal{H}^k_{v}$ and  $\mathcal{H}^k$ are Banach spaces for all $k\ge2$.

Finally, define the function space
	\begin{equation}\label{spaceg}
		G=\{ \mathbf u \in C^{1}(\overline{\Omega}_1) \times \cdots \times C^{1}(\overline{\Omega}_n): u_{i+1}(x_i^{+})=p_i u_i(x_i^{-}),i=1,\cdots,n-1\},
	\end{equation}
	equipped with the norm
	\begin{equation*}
		\| \mathbf u\|_{G}= \sum_{i=1}^{n}\dfrac{1}{\prod_{j=1}^{i-1}p_j}\| u_i \|_{C^{1}(\overline{\Omega}_i)}.
	\end{equation*}
Clearly, $G$ is a Banach space.

We now define an operator $\mathcal{A}$ on function space $X$ as follows:
		\begin{equation}\label{mathca}
			\mathcal{A}\mathbf w = \left
				(\mathcal{A}_u \mathbf u,
				\mathcal{A}_v \mathbf v
\right)
	\end{equation}
	where
	\begin{equation}\label{AuAv}
	\begin{aligned}
		\mathcal{A}_u \mathbf u &:= (-d_1(u_1)_{xx}+u_1, \cdots, -d_n(u_n)_{xx}+u_n),\\
		\mathcal{A}_v \mathbf v&:= (-\hat{d}_1(v_1)_{xx}+v_1, \cdots, -\hat{d}_n(v_n)_{xx}+v_n).
	\end{aligned}
	\end{equation}
The domain of $\mathcal{A}$ is given by:
	\begin{equation}\label{spaceda}
		\begin{aligned}
			\mathcal{D}(\mathcal{A}):=\{&\mathbf w=(\mathbf u,\mathbf v)\in \mathcal{H}^2\cap\mathcal{H}^1:( u_1)_x(0^{+})=(u_n)_x(L^{-})=( v_1)_x(0^{+})=( v_n)_x(L^{-})=0,  \\
			&d_{i+1}( u_{i+1})_{x}(x_i^{+})=d_i ( u_{i})_{x}(x_i^{-}),\;\hat d_{i+1}( v_{i+1})_{x}(x_i^{+})=\hat d_i ( v_{i})_{x}(x_i^{-})\}.
		\end{aligned}
	\end{equation}
By the Sobolev embedding theorem again, $\mathcal{D}(\mathcal{A})$ is a closed subspace of $\mathcal{H}^2$ and thus a Banach space.

Then we formulate model \eqref{m1} as the following abstract evolution equation:
	\begin{equation}\label{abstract}
		\dfrac{{\rm d}}{{\rm d}t}\mathbf w+\mathcal{A}\mathbf w=F(\mathbf w),
	\end{equation}
	where $\mathbf w=(\mathbf u,\mathbf v)=(u_1,\cdots,u_n, v_1,\cdots, v_n)\in X$ and
	\begin{equation}\label{auxic9}
	\begin{aligned}
		&F(\mathbf w) \\
		:=& \left( f_1(u_1, v_1) + u_1,\ \cdots,\ f_n(u_n, v_n) + u_n,\ \hat{f}_1(u_1, v_1) + v_1,\ \cdots,\ \hat{f}_n(u_n, v_n) + v_n \right)
	\end{aligned}
	\end{equation}
	with $f_i$, $\hat{f}_i$ be defined in \eqref{fhatf} for $i=1,\cdots,n$.

We now show that $-\mathcal{A}$ generates an analytic semigroup on $X$.
	\begin{lemma}\label{semigroup}
		The linear operator $ \mathcal{A} : \mathcal{D}(\mathcal{A}) \subset X \to X $ is symmetric maximal monotone, and $-\mathcal{A}$ is the infinitesimal generator of an analytic semigroup on $X$, where $X$, $\mathcal A$ and $\mathcal{D}(\mathcal{A})$ are defined in \eqref{spacex}, \eqref{mathca} and \eqref{spaceda}, respectively. Moreover,
$\mathcal{A}$ has a compact resolvent.
	\end{lemma}
	
	\begin{proof}
 Since ${\mathcal{A}}_u$ and ${\mathcal{A}}_v$ are essentially identical, it suffices to establish the properties for $\mathcal{A}_u$.  Although the proof technique parallels that of \cite[Lemma 3.3]{44-Zhang-2024} (see also \cite[Proposition 7.1]{36-Maciel-2020}), several adaptations are required. We detail these modifications below and omit the analogous steps.
		
	 Define
		\begin{equation}\label{Au}
			\begin{aligned}
				\mathcal{D}(\mathcal{A}_u) := \{ &\mathbf u \in \mathcal{H}_u^2\cap \mathcal{H}_u^1:
				d_{i+1} (u_{i+1})_{x}(x_i^{+}) = d_i (u_i)_{x}(x_i^{-}), i=1,\cdots,n-1,\\
				&(u_1)_x(0^{+}) = (u_n)_x(L^{-}) = 0 \},
			\end{aligned}
		\end{equation}
		 where $\mathcal{H}_u^1$ and $ \mathcal{H}_u^2$ are defined in \eqref{hu1} and \eqref{spaceh2}, respectively. 	For any $\mathbf u \in \mathcal{D}(\mathcal{A}_u)$,
		we use integration by parts and apply the interface and boundary conditions in \eqref{Au} to derive:
		\begin{equation}\label{con}
			\begin{aligned}
				\langle \mathcal{A}_u \mathbf u, \mathbf u \rangle_{X_u}
				&= \sum_{i=1}^{n} \frac{1}{\prod_{j=1}^{i-1} p_j} \left[
				\int_{\overline{\Omega}_i} -d_i (u_i)_{xx} u_i \, {\rm d}x
				+ \int_{\overline{\Omega}_i} (u_i)^2 \, {\rm d}x
				\right] \\
				&= \sum_{i=1}^{n} \frac{1}{\prod_{j=1}^{i-1} p_j} \left[
				\int_{\overline{\Omega}_i} d_i (u_i)_x^2 \, {\rm d}x
				+ \int_{\overline{\Omega}_i} (u_i)^2 \, {\rm d}x
				\right] \\
				&\geq \mathcal C_1 \| \mathbf u \|_{\mathcal{H}^1_u}^2 \geq 0,
			\end{aligned}
		\end{equation}
		where the constant $\mathcal C_1>0$ depends on $d_i,\,1/p_i$ ($i=1,\cdots,n$).
 Using an argument similar to that in the proof of \eqref{con}, we can prove that $\langle \mathcal{A}_u \mathbf u, \mathbf v \rangle = \langle \mathbf u, \mathcal{A}_u \mathbf v \rangle$ for all $\mathbf u, \mathbf v \in \mathcal{D}(\mathcal{A}_u)$. This implies $\mathcal{A}_u$ is symmetric .
		
		Next, we prove that for every $\lambda \geq 0$, the range $\Range(\lambda I_{X_u} + \mathcal{A}_u)=X_u$.
		That is, for any $\mathbf g=(g_1,\cdots,g_n) \in X_u$, the following problem admits a solution $\mathbf u\in\mathcal{D}(\mathcal{A}_u)$:
	\begin{subequations}\label{bv}
		\begin{empheq}[left=\empheqlbrace]{align}
			&-d_i (u_{i})_{xx} + (\lambda+1) u_i = g_i, && \text{in } \Omega_i, \label{bv-a} \\
			&u_{i+1}(x_i^{+}) = p_i u_i(x_i^{-}), && i=1,\cdots,n-1, \label{bv-d} \\
			&d_{i+1} (u_{i+1})_x(x_i^{+}) = d_i (u_i)_x(x_i^{-}), && i=1,\cdots,n-1, \label{bv-b} \\
			&(u_1)_x(0^{+}) = (u_n)_x(L^{-}) = 0. \label{bv-c}
		\end{empheq}
	\end{subequations}
		We first convert this problem into a weak formulation: find $\mathbf u \in \mathcal{H}_u^1$ such that
		\begin{equation}\label{weak}
			a(\mathbf u, \mathbf s) = \langle \mathbf g, \mathbf s \rangle_{X_u} \quad \text{for all } \mathbf s \in \mathcal{H}_u^1,
		\end{equation}
		where the bilinear form $a$ on $\mathcal{H}_u^1 \times \mathcal{H}_u^1$ is defined by
		\begin{equation*}
			a(\mathbf u, \mathbf s) = \sum_{i=1}^{n} \frac{1}{\prod_{j=1}^{i-1} p_j} \left[
			\int_{\overline{\Omega}_i} d_i (u_i)_x (s_i)_x \, \mathrm{d}x
			+ (\lambda+1) \int_{\overline{\Omega}_i} u_i s_i \, \mathrm{d}x
			\right].
		\end{equation*}
		Clearly, $a(\mathbf u,\mathbf s)$ is continuous. Following a similar approach to that for \eqref{con}, we can show that $a(\mathbf u,\mathbf s)$ is coercive. Then, by the Lax--Milgram theorem, problem \eqref{bv} admits a unique solution $ \mathbf u \in \mathcal{H}^{1}_{u} $ with
		\begin{equation}\label{hh1}
			\| \mathbf u \|_{\mathcal{H}^{1}_{u}} \leq \mathcal C_2 \| \mathbf g \|_{X_u},
		\end{equation}
where the constant $ \mathcal C_2 > 0$ depends on $d_i$, $p_i$, $1/d_i$ and $1/p_i$ ($i=1,\cdots,n$).
Combining this with the definition of $\mathcal{H}^{1}_{u}$ in \eqref{hu1}, we conclude that \eqref{bv-d} holds.

For each $i=1,\cdots,n$, we construct a test function $$\mathbf v=(v_1,\cdots,v_n)\in C^{\infty}(\overline{\Omega}_1) \times \cdots \times C^{\infty}(\overline{\Omega}_n)$$ such that $v_k\equiv 0$ for  all $1\le k\le n$ with $k\ne i$ and $v_i$ is supported in ${\Omega}_i$. Clearly, $\mathbf v\in \mathcal{H}_u^1$. Thus, \eqref{weak} implies  that $u_i\in H^2(\Omega_i)$ for all $i=1,\cdots,n$ (so that $\mathbf u \in \mathcal{H}^{2}_{u}$) and that \eqref{bv-a} holds.
Then, for each $i=1,\cdots,n-1$,  we construct a test function $\mathbf{v} = (v_1, \cdots, v_n) \in G\subset \mathcal{H}_u^1$ with the following properties: (i) $v_k \equiv 0$ for all $k$ such that $1 \le k \le i-1$ or $i+2 \le k \le n$; (ii) the component $v_i$ is supported in $(x_{i-1}, x_i]$;
(iii) The component $v_{i+1}$ is supported in $[x_i, x_{i+1})$. Here $G$ is defined in \eqref{spaceg}.
Since $\mathbf u \in \mathcal{H}^{1}_{u}\cap \mathcal{H}^{2}_{u} $, we see from \eqref{weak} that
		\begin{equation*}
		\begin{aligned}
		0=
		&-\int_{x_{i-1}}^{x_i}g_iv_i{\rm d}x-\dfrac{1}{p_i}\int_{x_i}^{x_{i+1}}g_{i+1}v_{i+1}{\rm d}x+\int_{x_{i-1}}^{x_i}\left[d_i(u_i)_x(v_i)_x+(\lambda+1) u_iv_i\right]{\rm d}x\\
		&+\dfrac{1}{p_i}\int_{x_i}^{x_{i+1}}\left[d_{i+1}(u_{i+1})_x(v_{i+1})_x+(\lambda+1) u_{i+1}v_{i+1}\right]{\rm d}x\\
		=&d_i(u_i)_x(x_i^-)v_i(x_i^-)-\dfrac{1}{p_i}d_{i+1}(u_{i+1})_x(x_i^+)v_{i+1}(x_i^+)+\int_{x_{i-1}}^{x_i}\left[-d_i(u_i)_{xx}+(\lambda+1) u_i\right]v_i{\rm d}x\\
		&-\int_{x_{i-1}}^{x_i}g_iv_i{\rm d}x+\dfrac{1}{p_i}\int_{x_i}^{x_{i+1}}\left[-d_{i+1}(u_{i+1})_{xx}+(\lambda+1) u_{i+1}-g_{i+1}\right]v_{i+1}{\rm d}x\\
		=&d_i(u_i)_x(x_i^-)v_i(x_i^-)-\dfrac{1}{p_i}d_{i+1}(u_{i+1})_x(x_i^+)v_{i+1}(x_i^+),
	   \end{aligned}
		\end{equation*}
		where we have used integration by parts in the second step and \eqref{bv-a} in the third step. Combining this with the fact that $\mathbf{v}\in\mathcal{H}_u^1$, we obtain \eqref{bv-b}. Using similar arguments, we can deduce that
$(u_1)_x(0^{+})=(u_n)_x(L^{-})=0$.
Thus, $\mathbf u \in 	\mathcal{D}(\mathcal{A}_u)$, and consequently, the operator $\mathcal{A}_u $ is maximal monotone. It is then densely defined and closed (by \cite[Proposition 7.1]{FA}). Then, by arguments similar to those in the proof of \cite[Lemma 3.3]{44-Zhang-2024}, we can deduce that $-\mathcal{A}_u $ is the infinitesimal generator of an analytic semigroup.

Furthermore, we see from \eqref{bv-a} that
\begin{equation*}
			 \| (u_i)_{xx} \|_{L_2(\Omega_i)}\le \frac{1}{d_i}\left[\| g_i \|_{L_2(\Omega_i)}+(\lambda+1)\| u_i \|_{L_2(\Omega_i)}\right],
\end{equation*}
where $\lambda\ge0$.
This combined with \eqref{hh1} implies that $ \| \mathbf u \|_{\mathcal{H}_u^{2}} \leq \mathcal C_3 \| \mathbf g \|_{X_u} $, where the constant $\mathcal C_3>0$ depends on $d_i$, $p_i$, $1/d_i$ and $1/p_i$ ($i=1,\cdots,n$).
Since $ \Range(\lambda I + \mathcal{A}_u) = X_u $ and $\lambda I + \mathcal{A}_u$ is injective (by \eqref{con}), it is invertible with a bounded inverse $(\lambda I + \mathcal{A}_u)^{-1} : X_u \to \mathcal{D}(\mathcal{A}_u)$. Since $\mathcal{D}({\mathcal{A}}_u)\xhookrightarrow{\text{c}}X_u$, we see that
$\mathcal{A}_u$ has a compact resolvent.
 This completes the proof.
	\end{proof}
	
\begin{remark}\label{rcompact}
By the proof of Lemma \ref{semigroup}, ${\mathcal{A}}_u$ (defined in \eqref{AuAv}) has a compact resolvent.
An analogous property holds for the operator ${\mathcal{A}}_v$, whose domain is given by
		\begin{equation}\label{Av}
		\begin{aligned}
			\mathcal{D}(\mathcal{A}_v) := \{ &\mathbf v \in \mathcal{H}_v^2\cap \mathcal{H}_v^1:
			\hat{d}_{i+1} (v_{i+1})_{x}(x_i^{+}) = \hat{d}_i (v_i)_{x}(x_i^{-}), i=1,\cdots,n-1,\\
			&(v_1)_x(0^{+}) = (v_n)_x(L^{-}) = 0 \},
		\end{aligned}
	\end{equation}
Because $-{\mathcal{A}}_u$ and $-{\mathcal{A}}_v$ both generate analytic semigroups, we can define their corresponding fractional power spaces $X^{\alpha}_u$ and $X^{\alpha}_v$ for $0<\alpha<1$. Moreover, since both ${\mathcal{A}}_u$ and ${\mathcal{A}}_v$ have compact resolvents, the embedding
\begin{equation}\label{eq:compact_embedding}
		X_{u}^\beta \times X_{v}^\beta \xhookrightarrow{c} X_{u}^\alpha \times X_{v}^\alpha
	\end{equation}
	is compact whenever $0<\alpha < \beta < 1$.
\end{remark}

For later applications, we establish the following embedding result for the function spaces defined above.

\begin{lemma}
	Suppose that $\alpha\in (1/4,1)$. Then the following statements hold:
	\begin{enumerate}
		\item [$\rm{(i)}$]  For $s>1/2$,
		\begin{equation}\label{A6}
		{W}^{s,2}(\Omega_1)\times\cdots\times{W}^{s,2}(\Omega_n) \xhookrightarrow{c}{C}^{r}(\Omega_1)\times\cdots\times{C}^{r}(\Omega_n),
		\end{equation}
	where $0\le r<s-1/2$; and
		\begin{equation}\label{xh}
		(X_u, \mathcal{H}_u^2)_{s,2}=(X_v, \mathcal{H}_v^2)_{s,2} ={W}^{2s,2}(\Omega_1)\times\cdots\times{W}^{2s,2}(\Omega_n),
		\end{equation}
		where $X_u, X_v, \mathcal{H}_u^2$ and $\mathcal{H}_v^2$ are defined in \eqref{spacex} and \eqref{spaceh2}.
		\item [$\rm{(ii)}$]  For  $0 < \theta < 1$,
	\begin{equation}\label{A7}
		(X_u, X_{u}^{\alpha})_{\theta,2} = (X_u, \mathcal{D}(\mathcal{A}_{u}))_{\alpha\theta,2}, \;\;(X_v, X_{v}^{\alpha})_{\theta,2} = (X_v, \mathcal{D}(\mathcal{A}_{v}))_{\alpha\theta,2}, \\
	\end{equation}
		and
		\begin{equation}\label{A9}
		 X_{u}^{\alpha},  X_{v}^{\alpha} \xhookrightarrow{c}{C}^{\nu}(\Omega_1)\times\cdots\times{C}^{\nu}(\Omega_n),
		\end{equation}
		where $0 < \nu < 2\alpha - \frac{1}{2}$, and $X_u, X_v, \mathcal{D}(\mathcal{A}_{u})$ and $\mathcal{D}(\mathcal{A}_{v})$ are defined in \eqref{spacex}, \eqref{Au} and \eqref{Av}.
	\end{enumerate}
(Here, $(S_1, S_2)_{\theta,q}$ denotes the interpolation space between function spaces $S_1$ and $S_2$ via the $K$-method; see \cite[Section 1.3.2]{Triebel}.)
\end{lemma}
\begin{proof}
The proof can be derived from \cite[Theorem A.1]{shen}, but we include it here for completeness.
The results in \eqref{A6} and \eqref{xh} follow directly from \cite[Theorem 11.5]{Amann} and \cite[Theorem 11.6]{Amann}, respectively, while \eqref{A7} is established in \cite[Section 1.15.2]{Triebel}.
	
	Now we prove \eqref{A9}. Fix an arbitrary $\nu$ satisfying $0 < \nu < 2\alpha - 1/2$. We can choose $\theta \in (0,1)$ such that $\alpha\theta \neq 1/2$ and $2\alpha\theta - 1/2 > \nu$. Then
	\begin{equation*}
	X_{u}^{\alpha} \subset (X_u, X_{u}^{\alpha})_{\theta,2} = (X_u, \mathcal{D}(\mathcal{A}_{u}))_{\alpha\theta,2} \subset (X_u, \mathcal{H}_u^2)_{\alpha\theta,2} = W^{2\alpha\theta,2}(\overline{\Omega}_1)\times\cdots\times W^{2\alpha\theta,2}(\overline{\Omega}_n),
	\end{equation*}
	where the second and last equalities follow from \eqref{A7} and \eqref{xh}, respectively.
Since $2\alpha\theta - 1/2 > \nu$, we see from \eqref{A6}  that
the embedding \eqref{A9} holds. This completes the proof.
\end{proof}

We now state the definition of a strong solution to \eqref{abstract}, following \cite[Definition 3.3.1]{HenryD}.
\begin{definition}\label{defstrong}
	A strong solution of \eqref{abstract} on the interval $(0,T)$ with initial condition $\mathbf w(0)=\mathbf w_0\in X^{\alpha}_u\times X^{\alpha}_v$ is a function $\mathbf w \colon [0,T) \to X$ satisfying the following:
\begin{enumerate}[itemsep=1.25pt, parsep=1.25pt, topsep=1.25pt]
\item [(i)]  $\mathbf{w}(t)\in \mathcal{D}(\mathcal A)$ and ${\rm d}\mathbf w(t)/{\rm d}t$ exists  for all $t\in (0,T)$;
\item [(ii)] The abstract equation \eqref{abstract} is satisfied on $(0,T)$ with $\mathbf w(0) = \mathbf w_0$;
\item [(iii)] $F(\mathbf w(t))$ is locally H\"{o}lder continuous in $t$, and $\displaystyle \int_{0}^{\rho} \|F(\mathbf w(t))\|_X {\rm d}t < \infty$ for some $\rho > 0$.
\end{enumerate}
\end{definition}

	\begin{theorem}\label{c33}
		For any given nonnegative initial value $(\mathbf u_0,\mathbf v_0)\in X^{\alpha}_u\times X^{\alpha}_v$ with $\alpha\in(3/4,1)$, model \eqref{m1} admits a unique classical solution $(\mathbf u(x,t), \mathbf v(x,t))$ (defined in Definition \ref{defc1}) that exists for $t\in[0,\infty)$ and satisfies $(\mathbf u(x,0),\mathbf v(x,0))=(\mathbf u_0,\mathbf v_0)$. Moreover, the following properties hold:
		\begin{enumerate}[itemsep=1.25pt, parsep=1.25pt, topsep=1.25pt]
		\item [$\rm{(i)}$] One  has
		 \begin{equation*}
			\begin{aligned}
				\mathbf{u}(\cdot,t) &\in C([0, \infty), X_u^{\alpha}) \cap C^1((0, \infty), X_u^{\alpha}), \\
				\mathbf{v} (\cdot,t)&\in C([0, \infty), X_v^{\alpha}) \cap C^1((0, \infty), X_v^{\alpha}),
			\end{aligned}
		\end{equation*}
where $3/4<\alpha<1$.
    	\item [$\rm{(ii)}$]

    There exist two positive constants $M_u$ and $M_v$, depending on $(\mathbf u_0,\mathbf v_0)$, $\bm k$,  $\bm p$ and  $\mathbf{\hat{ p}}$, such that for all $(x,t)\in [0,L]\times [0,\infty)$,
		 \begin{equation*}
			\mathbf 0\le \mathbf u\le \mathbf M_u\;\;\text{and}\;\;\mathbf 0\le \mathbf v\le \mathbf M_v.
		\end{equation*}
		Here  $\mathbf M_u,\mathbf M_v\in\mathbb R^n$ are defined by
		\begin{equation}\label{MuMv}
		\mathbf M_u:=(M_u, p_1M_u,\cdots, \prod_{j=1}^{n-1}p_jM_u)\;\;\text{and}\;\;\mathbf M_v:=(M_v, \hat{p}_1M_v,\cdots, \prod_{j=1}^{n-1}\hat{p}_jM_v).
		\end{equation}
		\end{enumerate}
	\end{theorem}
\begin{proof}
The existence of a classical solution is inspired by the approach in \cite[Theorems C.1 and C.2]{shen}, which studied a chemotaxis model. The details are provided below.
The uniqueness follows directly from Theorem \ref{appcp} and Remark \ref{strict}.
	
For given $3/4<\alpha<1$, the embedding \eqref{A9} implies:
\begin{equation}\label{imu}
	X_u^{\alpha}, \, X_v^{\alpha} \xhookrightarrow{c} C^1(\overline{\Omega}_1) \times \cdots \times C^1(\overline{\Omega}_n). \quad
\end{equation}
As a consequence, $F:X_u^\alpha\times X_v^\alpha\to X$, defined by \eqref{auxic9}, is locally Lipschitz continuous. By \cite[Theorem 3.3.3 and Theorem 3.5.2]{HenryD}, for  any initial value $(\mathbf{u}_0, \mathbf{v}_0) \in X^{\alpha}_u \times X^{\alpha}_v$,
 there exists $T_{\max}=T_{\max}(\mathbf{u}_0, \mathbf{v}_0) > 0$ such that \eqref{abstract} admits a unique strong solution  in the sense of Definition \ref{defstrong}, and
	 \begin{equation}\label{strongest}
\begin{split}
	& (\mathbf{u}(t), \mathbf{v}(t))\in C([0, T_{\max}), X_u^{\alpha})\times C([0, T_{\max}), X_v^{\alpha}),\\
& (\mathbf{u}(t), \mathbf{v}(t)) \in C^1((0, T_{\max}), X_u^{\alpha}) \times C^1((0, T_{\max}), X_v^{\alpha}),
\end{split}
	 \end{equation}
where $3/4<\alpha<1$.

	
Let
	\begin{equation*}
	\mathbf{u}(x,t)=(u_1,\cdots,u_n)=\mathbf u(t)(x)\;\;\text{and}\;\;\mathbf{v}(x,t):=(v_1,\cdots,v_n)=\mathbf v(t)(x)\;\;\text{for}\;\;x\in [0,L].
	\end{equation*}
	We now show that the strong solution is indeed a classical solution of the original model \eqref{m1}.
It follows from \eqref{imu} and \eqref{strongest} that
\begin{equation*}
	{u_i}(x, t), \, {v_i}(x, t) \in C^{1,0}(\overline{\Omega}_i\times [0, T_{\max})) \cap C^{1,1}(\overline{\Omega}_i \times (0, T_{\max})) \text{  for all  }i=1,\cdots,n.
\end{equation*}
Thus, $f_i({u}_i(x,t),{v}_i(x,t)),\,\hat f_i({u}_i(x,t),{v}_i(x,t))\in C^{1,1}(\overline{\Omega}_i\times(0, T_{\max}))$.
Now, for fixed $t > 0$, we rewrite \eqref{m1-1} and \eqref{m1-2} as the following form:
	\begin{equation*}
	\begin{aligned}
	&d_i( u_i)_{xx}(x,t)=u_i(x,t)+(u_i)_t(x,t)-f_i(u_i(x,t),v_i(x,t))\in C^1(\overline{\Omega}_i),\\
	&\hat{d}_i( v_i)_{xx}(x,t)=u_i(x,t)+(v_i)_t(x,t)-\hat{f}_i(u_i(x,t),v_i(x,t))\in C^1(\overline{\Omega}_i).
	\end{aligned}
	\end{equation*}
This implies that $(\mathbf{u}(\cdot,t), \mathbf{v}(\cdot,t))\in \mathcal{H}^3=\mathcal{H}_u^3\times\mathcal{H}_v^3$ for fixed $t>0$. By \eqref{A6}, we have 
$$\mathcal{H}_u^3,\mathcal{H}_v^3\xhookrightarrow{c} C^2(\overline{\Omega}_1)\times\cdots\times C^2(\overline{\Omega}_n),$$ which implies that 
	\begin{equation}\label{smooth}
	\begin{aligned}
	u_i(\cdot,\cdot), v_i(\cdot,\cdot) \in  C^{2,1}(\overline{\Omega}_i\times(  0, T_{\max})).
	\end{aligned}
	\end{equation}
Consequently, $(\mathbf{u}(x,t), \mathbf{v}(x,t))$  satisfies \eqref{m1-1}-\eqref{m1-2} pointwise. Since $(\mathbf{u}(x,t), \mathbf{v}(x,t))\in \mathcal{D}(\mathcal A)$ and given the regularity in \eqref{smooth}, it follows that $(\mathbf{u}(x,t), \mathbf{v}(x,t))$ satisfies
\eqref{m1-3}-\eqref{m1-5} pointwise. Therefore, $(\mathbf u(x,t),\mathbf v(x,t))$ is a classical solution of of the original model \eqref{m1}.
	
    Denote $\mathbf{w}(x,t):=(\mathbf u(x,t),\mathbf v(x,t))$ for simplicity. Now we show that $\{\mathbf{w}(\cdot,t)\,|\,t>0\}$ is bounded in $X_u^\alpha\times X_v^\alpha$.
    Choose two positive constants $M_u,M_v$ such that $(\mathbf M_u, \mathbf 0)\succeq( \mathbf u_0,\mathbf v_0)\succeq(\mathbf 0, \mathbf M_v)$ and ensure that the the pairs $(\mathbf M_u, \mathbf 0)$ and $(\mathbf 0, \mathbf M_v)$ (defined in \eqref{MuMv}) satisfy \eqref{B2} and \eqref{compr}. Here $\succeq$ is defined in \eqref{order}.
    Then it follows from Theorem \ref{appcp} that
    \begin{equation*}
   (\mathbf M_u, \mathbf 0)\succeq( \mathbf u(x,t),\mathbf v(x,t))\succeq(\mathbf 0, \mathbf M_v)\;\;\text{for}\;\;(x,t)\in[0,L]\times (0,T_{\max}).
    \end{equation*}
This implies that there exists a positive constant $\mathcal M_1$ such that
    \begin{equation}\label{eglobal}
	\begin{aligned}
	\|F(\mathbf{w}(x,t))\|_X\le\mathcal{M}_1 \;\;\text{for all}\;\;t>0.
	\end{aligned}
  \end{equation}
   According to \cite[Lemma 3.3.2]{HenryD}, $\mathbf{w}(\cdot, t)$ satisfies the following integral equation
   \begin{equation}\label{milde}
   	\mathbf{w}(\cdot, t) = e^{-t\mathcal{A}} \mathbf{w}_0 + \int_{0}^{t} e^{-(t-s)\mathcal{A}} F(\mathbf{w}(\cdot,s)) \, \mathrm{d}s.
   \end{equation}
 By Lemma \ref{semigroup} and its proof, we see that $\mathcal A$ has a compact resolvent, and $0$ belongs to the   resolvent set of $\mathcal A$. It follows from \cite[Chapter I, Section 1.4]{HenryD} that, for any given $0\le\beta\le1$,  the operator $\mathcal A^{\beta} e^{-t\mathcal A}$ is bounded on $X$ (i.e., $\mathcal A^{\beta} e^{-t\mathcal A}\in L(X)$), and
 \begin{equation}\label{change}
 \mathcal A^{\beta} e^{-t\mathcal A}\widetilde {\mathbf W}= e^{-t\mathcal A}\mathcal A^{\beta}\widetilde {\mathbf W}\;\;\text{for all}\;\;\widetilde {\mathbf W}\in \mathcal D(\mathcal A^\beta)=X_u^\beta\times X_v^\beta.
 \end{equation}
Moreover, there exists a positive constant $\widetilde{\mathcal M}_{\beta} $ such that
 \begin{equation}\label{bbe}
  \|\mathcal A^{\beta} e^{-t\mathcal A}\|_{L(X)}\le \widetilde{\mathcal M}_{\beta} t^{-\beta}e^{-\delta t},
  \end{equation}
  where $\|\cdot\|_{L(X)}$ is the operator norm in $L(X)$. 
Then, for all $t>0$, we see from \eqref{milde} that
     \begin{equation}\label{bound}
     \begin{aligned}
   	\|\mathbf{w}(\cdot, t)\|_{X^{\alpha}_u\times X^{\alpha}_v} &\le \|e^{-t\mathcal{A}}\|\cdot\| \mathbf{w}(\cdot,0)\|_{X^{\alpha}_u\times X^{\alpha}_v} + \int_{0}^{t} \|\mathcal{A}^{\alpha}e^{-(t-s)\mathcal{A}}\|_{L(X)}\| F(\mathbf{w}(\cdot,s))\|_{X} \, \mathrm{d}s\\
   	&\le\widetilde{\mathcal{M}}_0\| \mathbf{w}(\cdot,0)\|_{X^{\alpha}_u\times X^{\alpha}_v}+ \int_{0}^{t} \mathcal M_1\widetilde{\mathcal{M}}_\alpha(t-s)^{-\alpha}e^{-\delta (t-s)} \, \mathrm{d}s\\
   &\le\widetilde{\mathcal{M}}_0\| \mathbf{w}(\cdot,0)\|_{X^{\alpha}_u\times X^{\alpha}_v}+ \int_{0}^{\infty} \mathcal M_1\widetilde{\mathcal{M}}_\alpha(t-s)^{-\alpha}e^{-\delta (t-s)} \, \mathrm{d}s<\infty.
   	\end{aligned}
   \end{equation}
The first inequality uses \eqref{change}, while the second follows from \eqref{bbe} with $\beta = 0$ and $\beta = \alpha$, respectively.
  By \eqref{bound}, we see that $\{\mathbf{w}(\cdot,t)\,|\,t>0\}$ is bounded in $X_u^\alpha\times X_v^\alpha$.  By \eqref{imu} again, we see that, for every closed bounded set $\mathcal B$ in $X^\alpha_u\times X^\alpha_v$, the image $F(\mathcal B)$ is bounded in $X$. Then
   it follows from   \cite[Theorem 3.3.4]{HenryD} that $T_{\max}=\infty$. This completes the proof.
	\end{proof}
\begin{remark}\label{linearstability}
It follows from Lemma \ref{semigroup} and the proof of Theorem \ref{defstrong} that $-\mathcal{A}$ generates an analytic semigroup and that $F:X_u^\alpha\times X_v^\alpha\to X$, defined by \eqref{auxic9}, is locally Lipschitz continuous. Then, by \cite[Theorem 5.1.1]{HenryD}, a steady state of the abstract evolution equation \eqref{abstract} (or the original model \eqref{m1}) is asymptotically stable (resp., unstable) if
if it is linearly stable (resp., linearly unstable).
\end{remark}
 Note that $X^{\alpha}_u$ and $X^{\alpha}_v$ are ordered Banach spaces. Their positive cones, $P_{u}$ and $P_{v}$, both have nonempty interior.
 The cone $K=P_u\times (-P_v)$ induces the order $\ge_k$ as follows:
\begin{equation}\label{kine}
 (u_1,v_1)\ge_k (u_2,v_2)\;\; \text{if and only if} \;\;(u_1-u_2,v_1-v_2)\in K.
 \end{equation}
Define the continuous semiflow generated by model \eqref{abstract} as follows:
 \begin{equation}\label{ttt}
 	T_t: (\mathbf u_0,\mathbf v_0)\in X^{\alpha}_u\times X^{\alpha}_v\mapsto	 (\mathbf u(\cdot,t), \mathbf v(\cdot,t))\in X^{\alpha}_u\times X^{\alpha}_v.
 \end{equation}

\begin{theorem}\label{C4}
	Let the semiflow $T_t$ be defined as in \eqref{ttt}. Then the following statements hold:
\begin{enumerate}[itemsep=1.25pt, parsep=1.25pt, topsep=1.25pt]
\item [$\rm{(i)}$] $T_t$ is strongly order-preserving with respect to
  $\ge_k$ for all $t>0$, where $\ge_k$ is defined in \eqref{kine};
\item [$\rm{(ii)}$] $T_t$ is order compact in $X^{\alpha}_u\times X^{\alpha}_v$ for all $t>0$.
\end{enumerate}
\end{theorem}
\begin{proof}
Statement (i) is a direct consequence of Theorem \ref{appcp}, Remark \ref{strict} and Theorem \ref{c33}. Next, we prove statement (ii).
Note that $T_t$ is order compact if it maps every order interval in $P_{u}\times P_{v} \subset X^{\alpha}_u \times X^{\alpha}_v$ into a relatively compact set.
Consider the following order interval with respect to $\ge_k$:
	\begin{equation*}
		{I} := \left\{ (\mathbf{u}_0, \mathbf{v}_0) \in P_{u}\times P_{v} : (\overline{\mathbf{u}}_0, \mathbf 0) \ge_k (\mathbf{u}_0, \mathbf{v}_0) \ge_k (\mathbf 0, \overline{\mathbf{v}}_0) \right\}.
	\end{equation*}
Let $(\overline{\mathbf{u}}(\cdot,t),\mathbf 0)$, $(\mathbf 0, \overline{\mathbf{v}}(\cdot,t))$ and $(\mathbf{u}(\cdot,t), \mathbf{v}(\cdot,t))$ be the solutions of \eqref{abstract} with the initial values $(\overline{\mathbf{u}}_0,\mathbf 0)$, $(\mathbf 0, \overline{\mathbf{v}}_0)$, and $(\mathbf{u}_0, \mathbf{v}_0)$, respectively.
By Theorem \ref{c33}, these are all classical solutions to the original model \eqref{m1}. Then Theorem \ref{appcp} implies that
    \begin{equation*}
  (\overline{\mathbf{u}}(\cdot,t), \mathbf 0)\ge_k( \mathbf u(\cdot,t),\mathbf v(\cdot,t))\ge_k(\mathbf 0, \overline{\mathbf{v}}(\cdot,t))\;\;\text{for all}\;\;t>0.
    \end{equation*}
This combined with Theorem \ref{c33} (ii) implies that the set $$T_t(I)=\left\{(\mathbf u(\cdot,t),\mathbf v(\cdot,t))\,\big|\,(\mathbf u(\cdot,0),\mathbf v(\cdot,0))\in I\right\}$$ is bounded in $(C(\overline \Omega_1)\times \cdots \times C(\overline \Omega_n))^2$.
This  implies that there exists a positive constant $\mathcal M_1$ such that, for any fixed $t>0$,
    \begin{equation}\label{eglobal2}
	\begin{aligned}
\|F(\mathbf w(\cdot,t))\|_X\le \mathcal M_1 \;\;\text{whenever}\;\; \mathbf w(\cdot,0)\in I.
	\end{aligned}
  \end{equation}
Here $\mathbf{w}(\cdot,t) := (\mathbf{u}(\cdot,t), \mathbf{v}(\cdot,t))$, and  $F$ is defined by \eqref{auxic9}.

Then using arguments similar to those in the proof of Theorem \ref{c33}, we see from \eqref{eglobal2} that, for any initial value $\mathbf w(\cdot,0)\in I$ and any fixed $t > 0$,
	\begin{align*}
		\|\mathbf{w}(\cdot,t)\|_{\beta}
		&\le \|\mathcal{A}^{\beta - \alpha} e^{-t\mathcal{A}}\| \cdot \|\mathbf{w}_0\|_{\alpha} + \int_{0}^{t} \|\mathcal{A}^{\beta} e^{-(t - s)\mathcal{A}}\|_{L(X)} \cdot \|F( \mathbf{w}(\cdot,s))\|_X  \, \mathrm{d}s \\
		&\le\widetilde {\mathcal M}_{\beta-\alpha} \, t^{-(\beta - \alpha)} e^{-\delta t} \|\mathbf{w}_0\|_{\alpha} +\widetilde {\mathcal M}_{\beta}\mathcal M_1 \int_{0}^{t} (t - s)^{-\beta} e^{-\delta(t - s)}  \, \mathrm{d}s<\infty.
	\end{align*}
Here $\alpha<\beta<1$, and $\widetilde M_{\beta},\, \widetilde M_{\beta-\alpha},\,\delta$ are defined in the proof of  Theorem \ref{c33}.
	Consequently, $T_t(I)$
is bounded in $X_u^\beta\times X_v^\beta $ for any fixed $t>0$. This combined with the compact embedding \eqref{eq:compact_embedding} implies that the $T_t(I)$ is relatively compact in $X_u^\alpha\times X_v^\alpha$. This complete the proof.
\end{proof}

Finally, we consider the following eigenvalue problem
	\begin{subequations}\label{am3}
		\begin{empheq}[left=\empheqlbrace]{align}
			&\ds -d_i (\phi_i)_{xx}-c_i(x) \phi_i=\lambda\phi_i(x),\;\; \text{in } \Omega_i,\;\; i=1,2,\cdots,n, \label{am3-a} \\
			&\phi_{i+1}(x_i^{+})= p_i\phi_i(x_i^{-}), \quad d_{i+1}(\phi_{i+1})_{x}(x_i^{+})= d_i (\phi_{i})_{x}(x_i^{-}), \label{am3-b} \\
			&(\phi_{1})_{x}(0^{+})=(\phi_{n})_{x}(L^{-})=0. \label{am3-c}
		\end{empheq}
	\end{subequations}
\begin{proposition}\label{prin-exist}
	Assume that $c_i(x)\in C^1(\overline{\Omega}_i)$ for all $i=1,\cdots,n$, and let $\mathbf d=(d_1,\cdots,d_n)$. Then the eigenvalue problem \eqref{am3} admits a simple eigenvalue $\lambda_{1}(\mathbf d)$ with a corresponding positive eigenfunction $\Phi^*$, and any other eigenvalue $\lambda \neq \lambda_1(\mathbf d)$ of \eqref{am3} satisfies
	$\operatorname{Re}(\lambda) > \lambda_{1}$. Moreover, $\la_1(\mathbf d)$ is continuous with respect to $\mathbf d$.
\end{proposition}
\begin{proof}
Define $$\Lambda:=\ds\max_{1\le i\le n}\left \{\|c_i(x)\|_{C(\overline{\Omega}_i)}+1\right \}.$$ For $\mathbf g\in G$, where $G$ is defined in \eqref{spaceg}, we consider the following problem
	\begin{subequations}\label{am4}
		\begin{empheq}[left=\empheqlbrace]{align}
			&\ds -d_i (\phi_i)_{xx}+[\Lambda-c_i] \phi_i=g_i,\;\; \text{in}\;\; \Omega_i,\;\; i=1,2,\cdots,n, \label{am4-a} \\
			&\phi_{i+1}(x_i^{+})= p_i\phi_i(x_i^{-}), \quad d_{i+1}(\phi_{i+1})_{x}(x_i^{+})= d_i (\phi_{i})_{x}(x_i^{-}), \label{am4-b} \\
			&(\phi_{1})_{x}(0^{+})=(\phi_{n})_{x}(L^{-})=0. \label{am4-c}
		\end{empheq}
	\end{subequations}
Define
	\begin{equation}\label{Au-last}
		\widetilde{\mathcal{A}}_u \mathbf u := (-d_1(u_1)_{xx}+(\Lambda-c_1(x))u_1, \cdots, -d_n(u_n)_{xx}+(\Lambda-c_n(x))u_n),
	\end{equation}
and the domain of $\widetilde{\mathcal{A}}_u$ is given by:
\begin{equation}\label{tAu-last}
			\begin{aligned}
				\mathcal{D}(\widetilde{\mathcal{A}}_u) :=\, \{ &\mathbf u \in \mathcal{H}_u^2\cap \mathcal{H}_u^1:
				d_{i+1} (u_{i+1})_{x}(x_i^{+}) = d_i (u_i)_{x}(x_i^{-}),\, i=1,\cdots,n-1,\\
				&(u_1)_x(0^{+}) = (u_n)_x(L^{-}) = 0 \},
			\end{aligned}
		\end{equation}
Then using arguments similar to those in the proof of Lemma \ref{semigroup}, we conclude that $\widetilde{\mathcal{A}}_u$ is symmetric and that \eqref{am4} admits a unique solution $\Phi \in 	\mathcal{D}(\widetilde{\mathcal{A}}_u)$. Moreover,  $\Phi $ satisfies
 \begin{equation}\label{esti2}
 \| \Phi \|_{\mathcal{H}_u^{2}} \leq \mathcal{C}_1\| \mathbf g \|_{X_u}.
 \end{equation}
Here the constant $\mathcal{C}_1>0$ depends on $d_i$, $p_i$, $1/d_i$ and $1/p_i$ ($i=1,\cdots,n$).
Since $\mathcal{H}_u^2\xhookrightarrow{c} C(\overline{\Omega}_1)\times\cdots\times C(\overline{\Omega}_n)$  and  $G\xhookrightarrow {c}X_u$, it follows from  \eqref{esti2} that
\begin{equation*}
	\max_{1\le i\le n}\| \phi_i\|_{C^1(\overline{\Omega}_i)}\leq \mathcal{C}_{2} \| \mathbf g \|_{X_u}\leq \mathcal{C}_{3}\| \mathbf g \|_{G},
\end{equation*}
where the constants $\mathcal{C}_{2}, \,\mathcal{C}_{3} > 0$ depends on $d_i$, $p_i$, $1/d_i$ and $1/p_i$ ($i=1,\cdots,n$).  This combined with \eqref{am4-a} implies that $\Phi\in \mathcal{H}^3_{u}$ and
\begin{equation}\label{lastes}
	 \| \Phi \|_{\mathcal{H}_u^{3}} \le \mathcal{C}_4\| \mathbf g \|_{X_u},
\end{equation}
where the constant $\mathcal{C}_{4}>0$ depends on $d_i$, $p_i$, $1/d_i$ and $1/p_i$ ($i=1,\cdots,n$).
The embedding
$\mathcal{H}_u^3\xhookrightarrow{c} C^2(\overline{\Omega}_1)\times\cdots\times C^2(\overline{\Omega}_n)$ combined with \eqref{lastes} implies that
\begin{equation}\label{boundphi}
	\max_{1\le i\le n}\|\phi_i\|_{C^2(\overline{\Omega}_i)}\le \mathcal{C}_{5}\| \mathbf g \|_{G},
\end{equation}
where the constant $\mathcal{C}_{5}$ depends on $d_i$, $p_i$, $1/d_i$ and $1/p_i$ ($i=1,\cdots,n$).

For fixed $\mathbf d \gg\bm 0$, define the linear operator $T(\mathbf d): G \to G$ by $T(\mathbf d)\mathbf g = \Phi$, where $\Phi$ is the solution of \eqref{am4} corresponding to $\mathbf g$. Since $\Phi \in C^2(\overline{\Omega}_1)\times\cdots\times C^2(\overline{\Omega}_n) $ and  $C^2(\overline{\Omega}_1)\times\cdots\times C^2(\overline{\Omega}_n)  \xhookrightarrow{c}G$, we can deduce that $T(\mathbf d)$ is compact. Let $$\mathcal{K}:=\{\mathbf w \in G:w_i(x)\ge 0\;\;\text{in}\;\;\Omega_i\}$$ be the positive cone in $G$. This cone has nonempty interior $\mathring{\mathcal{K}}$ and satisfies $\mathcal{K} \cap (-\mathcal{K}) = \{ \mathbf 0 \}$. By Theorem \ref{parabolic}, we obtain that, if $\mathbf g \in \mathcal{K} \backslash \{ \mathbf 0 \}$, then $\Phi \in\mathring{\mathcal{K}}$, which implies that  $T(\mathbf d)(\mathcal{K} \backslash \{ \mathbf 0 \})\subset\mathring{\mathcal{K}}$. Then it follows from the Krein-Rutman theorem (see e.g., \cite[Theorem 1.2]{Du}) that the spectral radius $r (\mathbf d)$  is a simple eigenvalue of $T (\mathbf d)$  with a corresponding positive eigenfunction $\mathbf g^{*}$.
This implies that
\begin{equation*}
	-d_i (\phi_i^*)_{xx} + [\Lambda - c_i] \phi_i^* = \frac{\phi_i^*}{r(\mathbf d)} \quad \text{in } \Omega_i, \quad i = 1, 2, \cdots, n,
\end{equation*}
where $\mathbf\Phi^*=(\phi_1^*,\cdots,\phi_n^*)=T\mathbf g^{*}$ is positive.
Consequently, $\lambda_{1} (\mathbf d)= 1/r(\mathbf d) - \Lambda$ is a simple eigenvalue of \eqref{am3} with a positive eigenfunction $\Phi^{*}$.
Then, using arguments similar to those in the proof of \cite[Theorem 1.4]{Du}, we deduce that any eigenvalue $\lambda \neq \lambda_1$ of \eqref{am3} satisfies
$\operatorname{Re}(\lambda) > \lambda_{1}$.

Finally, we will show that $\la_1(\mathbf d)$ is continuous with respect to $\mathbf d$. By \eqref{esti2}, the operator $T(\mathbf d)$ is also bounded and compact on $X_u$.
We first show that that $T(\mathbf d)$ is continuous in $L(X_u)$, where $L(X_u)$ denotes the set of bounded linear operators on $X_u$. That is, for any sequence $\left\{\mathbf d^k\right\}_{k=1}^\infty$ satisfying
$\ds\lim_{k\to \infty}\mathbf d^k=\mathbf d$, we have
\begin{equation}\label{strong}
\ds\lim_{k\to \infty}\|T(\mathbf d^k)-T(\mathbf d)\|_{L(X_u)}=0,
\end{equation}
 where $\|\cdot\|_{L(X_u)}$ is the operator norm in $L(X_u)$.
Since the constant $\mathcal C_4$ in \eqref{esti2} depends only on $d_i$ and $1/d_i$ ($i=1,\cdots,n$), we see from \eqref{esti2} that
$$\bigcup_{k=1}^\infty\left\{T(\mathbf d^k)\mathbf g:\|\mathbf g\|_{X_u}\le1\right\}$$ is compact in $X_u$, and $\left\{\|T(\mathbf d^k)\|_{L(X_u)}\right\}_{k=1}^\infty$ is bounded. Consequently, $\{T(\mathbf d^k)\}_{k=1}^\infty$ is collectively compact. Since $\widetilde{\mathcal{A}}_u$ is symmetric, it follows that $T(\mathbf d^k)$ is self-adjoint. Then it follow from
\cite[Lemma 5.2 and Corollary 5.7]{1967Collectively} that \eqref{strong} holds if and only if
\begin{equation}\label{kin-2}
	\lim_{k\to\infty} T(\mathbf d^k)\mathbf g=T(\mathbf d)\mathbf g\;\;\text{in}\;\;X_u\;\;\text{for any fixed}\;\;\mathbf g\in X_u.
\end{equation}
Note that $\left\{\|T(\mathbf d^k)\|_{L(X_u)}\right\}_{k=1}^\infty$ is bounded, and $G$ is
dense in $X_u$. Therefore, to prove  \eqref{kin-2}, it suffices to show that
\begin{equation}\label{kin}
	\lim_{k\to\infty} T(\mathbf d^k)\mathbf g=T(\mathbf d)\mathbf g\;\;\text{in}\;\;G\;\;\text{for any fixed}\;\;\mathbf g\in G.
\end{equation}

 Now we construct an auxiliary function for further application. Let $\mathbf g=(g_1,\cdots,g_n)\in G$. For any given $a>0$, we inductively
solve the following ODEs. For $i=1$:
\begin{equation}\label{phi1}
\begin{cases}
	\ds\frac{{\rm d} \phi_1}{{\rm d }x}=\psi_1, &x\in(0,x_1),\\[10pt]
	\ds\frac{{\rm d} \psi_1}{{\rm d }x}=\dfrac{1}{d_1}\left[(\Lambda-c_1) \phi_1-g_1\right],&x\in(0,x_1),\\[10pt]
\phi_{1}(0^{+})= a, \;\;\psi_1(0^{+})=0.
\end{cases}
\end{equation}
For $i=2,\cdots,n$:
\begin{equation}\label{phii}
	\begin{cases}
		\ds\frac{{\rm d} \phi_i}{{\rm d }x}=\psi_i,&x\in (x_{i-1},x_i),\\[10pt]
		\ds\frac{{\rm d} \psi_i}{{\rm d }x}=\dfrac{1}{d_i}\left[(\Lambda-c_i) \phi_i-g_i\right],&x\in (x_{i-1},x_i),\\[10pt]
		\phi_{i}(x_{i-1}^{+})=p_{i-1} \phi_{i-1}(x_{i-1}^{-}),\;\;\psi_i(x_{i-1}^{+})=\dfrac{d_{i-1}}{d_i}\psi_{i-1}(x_{i-1}^{-}).
	\end{cases}
\end{equation}
Define the auxiliary  function  $\mathcal T:(0,\infty)\times \mathbb R^n \to \mathbb R$ by $\mathcal T(a,\mathbf d):=\psi_n(L)$.

By the continuous dependence of solutions to ODEs on initial values and parameters, we see that $\mathcal T(a,\mathbf d)$ is continuous with respect to $(a,\mathbf d)$.
Since \eqref{am4} has a unique solution $\Phi=(\phi_1,\cdots,\phi_n)\in G$ , it follows that $\mathcal T(a,\mathbf d)=0$ admits a unique solution $a(\mathbf d)$ for fixed $\mathbf d$. We now prove that
\begin{equation}\label{adp}
\ds\lim_{k\to\infty} a(\mathbf d^k)=a(\mathbf d).
\end{equation}
Since the constant $\mathcal C_5$ in \eqref{boundphi} depends only on $d_i$ and $1/d_i$($i=1,\cdots,n$), it follows from \eqref{boundphi} that $\left\{a(\mathbf{d}^k)\right\}_{k=1}^{\infty}$ is bounded. Consider any convergent subsequence $\{a(\mathbf{d}^{k_j})\}_{j=1}^{\infty}$ with limit $a^*$. Since $\mathcal{T}(a(\mathbf{d}^{k_j}), \mathbf{d}^{k_j}) = 0$ for all $j$ and $\mathcal{T}(a,\mathbf d)$ is continuous with respect to $(a,\mathbf d)$, taking the limit as $j \to \infty$ yields $\mathcal{T}(a^*, \mathbf{d}) = 0$. The uniqueness of the solution to $\mathcal{T}(a, \mathbf{d}) = 0$ then implies that $a^* = a(\mathbf{d})$. Thus, \eqref{adp} holds. Note that
\begin{equation*}
\left(T(\mathbf d)\mathbf g, \ds\frac{{\rm d}  T(\mathbf d)\mathbf g}{{\rm d} x}\right)
\end{equation*}
is the solution of \eqref{phi1}-\eqref{phii} for $a=a(\mathbf{d})$.
Similarly,
\begin{equation*}
\left(T(\mathbf d^k)\mathbf g, \ds\frac{{\rm d}T(\mathbf d^k) \mathbf g}{{\rm d} x}\right )
\end{equation*}
is the solution of \eqref{phi1}-\eqref{phii} for $\mathbf d=\mathbf d^k$ and  $a=a(\mathbf{d}^k)$.
By the continuous dependence of solutions to ODEs on initial values and parameters, we see from \eqref{adp} that \eqref{kin} holds. Consequently, $T(\mathbf d)$ is continuous in $L(X_u)$ with respect to $\mathbf d$.

Note from \cite[Chapter III, Section 3.5]{43-Kato-1966} that a simple and isolated eigenvalue depends continuously on the operator.
Since the spectral radius $r(\mathbf{d})$ is a simple and isolated eigenvalue of $T(\mathbf d)$ on $G$ (and hence $X_u$), it follows that
 $r(\mathbf{d})$ is continuous with respect to $\mathbf{d}$, and consequently
$\la_1(\mathbf{d})$ is continuous with respect to $\mathbf{d}$.  This completes the proof.
\end{proof}

\end{appendices}
\section*{Acknowledgments}

This work was supported by National Natural Science Foundation of China (No. 12171117), Taishan Scholars Program of Shandong Province (No. tsqn 202306137) and Shandong Provincial Natural Science Foundation of China (No.
ZR2020YQ01).

\end{document}